\documentclass[reqno]{amsart}
 \usepackage[letterpaper, left=30mm,right=30mm,top=30mm,bottom=30mm,marginpar=25mm]{geometry}

\usepackage{enumitem}
\usepackage{color}
\usepackage{amsmath}
\usepackage{amssymb}
\numberwithin{equation}{section}
\newtheorem*{assumption*}{Spectral Assumption}
\usepackage{graphicx}
\newtheorem{definition}{Definition}
\newtheorem{theorem}{Theorem}

\newtheorem{rem}[theorem]{Remark}
\newtheorem{lemma}[theorem]{Lemma}
\newcommand{\loc}{\operatorname{loc}}

\let\eps\varepsilon
\allowdisplaybreaks
\usepackage[normalem]{ulem}

\usepackage{amssymb}

\title[Geometric correction for radiative transfer]{Diffusive limits of the steady state radiative heat transfer system: Curvature effects}
\author{Mohamed Ghattassi}
\address{Department of Mathematics, New York University in Abu Dhabi, Saadiyat Island, P.O. Box 129188, Abu Dhabi, United Arab Emirates, {\sf mg6888@nyu.edu}}
\author{Xiaokai Huo}
\address{Department of Mathematics, Iowa State University, 411 Morrill Rd, Ames, 50011, IA, USA, {\sf xhuo@iastate.edu}}
\author{Nader Masmoudi}
\address{Department of Mathematics, New York University in Abu Dhabi, Saadiyat Island, P.O. Box 129188, Abu Dhabi, United Arab Emirates--
Courant Institute of Mathematical Sciences, New York University, 251 Mercer Street, New York, NY 10012, USA, {\sf nm30@nyu.edu}}

\begin{document}

 \begin{abstract}
This paper is devoted to the diffusive limit of the nonlinear radiative heat transfer system with curved boundary domain (\textit{two dimensional disk}). The solution constructed in \cite{ghattassi2022convergence} by the leading order interior solution and the boundary layer corrections fails here to approximate the  solutions in $L^\infty$ sense for the diffusive limit. The present paper aims to construct a geometric correction to the boundary layer problem and obtain a valid approximate solution in $L^\infty$ sense. The main tools to overcome the convergence problem, are to use matched asymptotic expansion techniques, fixed-point theorems, linear and nonlinear stability analysis of the boundary layer problem. In particular, the spectral assumption on the leading order interior solution, which was proposed for the flat case in \cite{Bounadrylayer2019GHM2}, is shown to be still valid which guarantee the stability of the boundary layer expansion with geometric corrections. Moreover, the convergence result established in \cite[Lemma 10]{ghattassi2022convergence} remain  applicable for the approximate solution with geometric corrections.


\end{abstract}

\maketitle   
 \tableofcontents
\section{Introduction}
We consider the following radiative heat transfer system 
\begin{align}
	\eps^2 \Delta T^\eps + \langle \psi^\eps - (T^\eps)^4 \rangle = 0, \label{eq:1}\\
	\eps \beta\cdot\nabla \psi^\eps + \psi^\eps - (T^\eps)^4 = 0,\label{eq:2}
\end{align}
on the two dimensional disk $\Omega = \{x=(x_1,x_2):|x|\le 1\}$ with radiative velocity $\beta \in \mathbb{S}^1$. Here $T^\eps=T^\eps(x)$ is the temperature and $\psi^\eps=\psi^\eps(x,\beta)$ is the radiative intensity and $\langle \cdot \rangle$ denotes the integration over $\mathbb{S}^1$, i.e. $\langle f\rangle := \int_{\mathbb{S}^1} f(\beta)d\beta$. We denote by  $\eps>0$  a small parameter, describing the ratio of a typical photon mean free path to a typical length scale of the problem, for more details we refer to \cite{ghattassi2020diffusive}. The system \eqref{eq:1}-\eqref{eq:2} is supplemented with the non-homogenous Dirichlet boundary conditions 
\begin{align}
	&T^\eps(x)=T_b(x),\quad \text{for any } x\in \partial\Omega, \label{eq:1b}\\
	&\psi^\eps(x,\beta)=\psi_b(x,\beta),\quad \text{for any }(x,\beta) \in \Gamma_-,\label{eq:2b}
\end{align}
where $\partial\Omega$ is the boundary of $\Omega$ and, the inflow boundary $\Gamma_-$ and out flow boundary $\Gamma_+$ are defined respectively by 
\begin{align*}
	&\Gamma_-:=\{0(x,\beta) \in \partial \Omega\times\mathbb{S}^1: n(x)\cdot \beta<0\},\\
	&\Gamma_+:=\{(x,\beta) \in \partial \Omega\times\mathbb{S}^1: n(x)\cdot \beta>0\},
\end{align*}
where $n(x)$ is the exterior normal vector on the boundary. Similarly, the out-flow boundary  and $\Gamma:=\Gamma_+\cup \Gamma_-\cup \Gamma_0$ with $\Gamma_0:=\{(x,\beta)\partial\Omega\times\mathbb{S}^1:n(x)\cdot \beta=0\}$. 

However, when $\psi_b=T_b^4$, there is no boundary layer for \eqref{eq:1}-\eqref{eq:2} and one can prove the solution to \eqref{eq:1}-\eqref{eq:2} converges to the solution of the nonlinear elliptic equation 
\begin{align}\label{eq:ne0}
	\Delta T_0 + \pi \Delta T_0^4 = 0,\quad \psi_0=T_0^4,
\end{align}
with Dirichlet boundary conditions $T_0(x)=T_b(x)$ for $x\in\partial\Omega$. However, for general boundary data, boundary layer exists and
when $\eps\to 0$, $T^\eps$ no longer converges to $T_0$ solution to the above nonlinear elliptic equation \eqref{eq:ne0}. A boundary layer correction $(\bar{T}_0,\bar{\psi}_0)$ should be added to $(T_0,\psi_0)$ to get the approximate solution. For the flat boundary, it was proved in \cite{ghattassi2022convergence} that
\begin{align}\label{eq:conv_0}
	\|T^\eps - T_0 -\bar{T}_0\|_{L^\infty(\Omega)} = O(\eps),\quad \|\psi^\eps - T_0^4 - \bar{\psi}_0\|_{L^\infty(\Omega)} =O(\eps).
	\end{align}
The boundary layer $\bar{T}_0$ is defined by a suitable cut-off of the corresponding half-space boundary Milne problem of system \eqref{eq:1}-\eqref{eq:2}. Let $(\tilde{T}_0,\tilde{\psi}_0)$ be the leading order of the half-space Milne problem.
The above inequalities were proved under the condition that $\tilde{T}_0$ satisfy a spectral assumption \cite[\ref{asA}]{Bounadrylayer2019GHM2}, which reads as 
\begin{enumerate}[label=({\bf\Alph*})]
	\item \label{asA} 
	A function $h \in C^1(\mathbb{R}_{+})$ satisfies the spectral assumption if there exists a constant $\beta>0$, such that 
	\begin{align}
		M\int_0^\infty e^{2\beta \eta} (2h^{\frac32})^2 |\partial_\eta f|^2 d\eta  \ge  4\int_0^\infty  e^{2\beta \eta}  |\partial_\eta (2h^{\frac32})|^2 f^2 d\eta,
		\label{eq:spassump2}
	\end{align}
holds for some constant $M<1$ and for any measurable function $f \in C^{1}(\mathbb{R}_{+})$ with $f(0)=0$.
\end{enumerate}
The spectral assumption implies the linear stability of the boundary layer problem which is crucial for establishing the convergence result \eqref{eq:conv_0}. However, for the not-flat boundary case, curvature effects brings singularities of the boundary layer problem at the boundary and the inequalities \eqref{eq:conv_0} fail to hold. We address this problem by constructing a correction $(\bar{T}_0^\eps,\bar{\psi}^\eps)$ of the boundary layer to improve the approximations near the boundary and erase the singularities. Moreover, it is expected that the spectral assumption \ref{asA} still implies the linear stability of the boundary layer problem with geometric correction.

\subsection{Main results}
\begin{theorem}\label{thm:mainresult}
Let $T_b \in C^2(\partial\Omega)$, $\psi_b \in C^1(\Gamma_-)$ be non-negative on the boundary. Let $(T_0,\psi_0)$ be the solution to \eqref{eq:ap0_inside} and $(\bar{T}_0^\eps,\bar{\psi}_0^\eps)$ be the solution to \eqref{eq:g0}. $\tilde{T}_0^\eps$ is the solution to the geometric correction in the half-space Milne problem \eqref{eq:g0}. Assume $\tilde{T}^\eps_0$ satisfy the spectral assumption \ref{asA} and $\tilde{T}^\eps_0\ge a$ for some constant $a>0$. Then for $\eps$ sufficiently small, the solution to system \eqref{eq:1}-\eqref{eq:2} with boundary conditions \eqref{eq:1b}-\eqref{eq:2b} is unique and satisfies 
\begin{align}\label{eq:mainresult}
	\|T^\eps - T_0 - \bar{T}_0^\eps \|_{L^\infty(\Omega)} =O(\eps),\quad \|\psi^\eps - \psi_0 - \bar{\psi}_0^\eps \|_{L^\infty(\Omega\times\mathbb{S}^1)}=O(\eps).
\end{align}
\end{theorem}
The above theorem shows that by adding the geometric correction boundary layer solution $(\bar{T}_0^\eps,\bar{\psi}_0^\eps)$, \eqref{eq:conv_0} still holds for the non-flat boundary and the sum of the interior solution $(T_0,\psi_0)$ and the geometric correction $(\bar{T}_0^\eps,\bar{\psi}_0^\eps)$ give an approximate solution to system \eqref{eq:1}-\eqref{eq:2} valid in the diffusive limit.
\subsection{Contributions}
When $T^4 = \langle \psi\rangle/2\pi$, system \eqref{eq:1}-\eqref{eq:2} reduce to the classical transport equation $\eps\beta\cdot \nabla \psi^\eps + \psi^\eps - \langle \psi^\eps\rangle/2\pi=0$. The diffusive limit was first studied in \cite{bensoussan1979boundary}, where an approximate solution is constructed by adding interior solution and boundary layer corrections. In \cite{wu2015geometric}, a counter-example shows that such an approximation can fail when the boundary is not flat. By introducing a geometric correction to the boundary layer problem, a valid approximate solution was constructed and can be proved to satisfy $\|\psi^\eps-\psi_0-\bar{\psi}_0\|_{L^\infty(\Omega\times\mathbb{S}^1)} =O(\eps)$, where $\psi_0 \in L^\infty(\Omega)$ satisfy $\Delta \psi_0 =0$ and $\bar{\psi}_0^\eps$ is the solution of the geometric corrected boundary correction. The aim of this paper is to find a similar geometric correction to the approximate solution for system \eqref{eq:1}-\eqref{eq:2}. However, the above discussion is not comprehensive: we refer the reader to the review articles, \cite{gie2010boundary,gie2016recent} and references therein for a more thorough review of the boundary layer analysis of some singular perturbation problems.

We now state the difference between our work  and \cite{wu2015geometric}. First, our system \eqref{eq:1}-\eqref{eq:2} is nonlinear which generates a coupling between  the interior solution and the boundary layer corrections. Indeed, to construct the boundary layer expansions, we need to approximate the interior solutions by their Taylor's expansion. Moreover, the corresponding boundary half-space Milne problems associated to system \eqref{eq:1}-\eqref{eq:2} are nonlinear for the leading order and linear for higher orders, different from \cite{wu2015geometric} where the Milne problems are linear for any order of expansions. The coupling between the elliptic and transport equation brings additional difficulties for proving the existence, uniqueness and stability of the Milne problems. In particular, the spectral assumption \ref{asA} is needed to show the linear stability of the Milne problems in our problem.
We use techniques from our previous work \cite{Bounadrylayer2019GHM2} but accounting for the geometric corrections. We can show the spectral assumption \ref{asA} also implies the linear stability of the geometric corrected Milne problem. Furthermore, due to the nonlniearity of our problem  it's not possible to derive the estimates \eqref{eq:mainresult} by using the Maximum principle, since the difference $(T^\eps-T_0-\bar{T}_0^\eps,\psi^\eps-\psi_0-\bar{\psi}^\eps)$ satisfies a different system from \eqref{eq:1}-\eqref{eq:2}. Instead, the Banach fixed point theorem is used. We here show that such a technique used in \cite{ghattassi2022convergence} for the flat case, can be extended to the curvature boundary case by considering the geometric corrections.
\subsection{Organizations} Our paper is organized as follows: In the next section, we use asymptotic analysis to construct the approximate interior solution and boundary layer corrections with geometric corrections. Section \ref{sec3} is devoted to the analysis of leading order boundary layer Milne problem. Section \ref{sec4} is devoted to the analysis of linear Milne problem of the higher order boundary expansions. In section \ref{sec5}, the proof of Theorem \ref{thm:mainresult} is given.

\textbf{Notations :} Throughout the paper, we use $T_k$, $k\ge 0$ to denote the interior solution and $\bar{T}_k$, $k\ge 0$ for the boundary layer correction,  $\bar{T}_k^\eps$, $k\ge 0$ for the geometric corrected boundary layer correction., and $\tilde{T}_k$, $k\ge 0$ for the solution of the half-space Milne problem. We use $C$ to denote the generic constant.
\section{Approximate solutions} \label{sec:2}
In this section, we construct approximation solutions to system \eqref{eq:1}-\eqref{eq:2} using asymptotic expansions.
\subsection{Interior expansions}
We take the ansatz 
\begin{align*}
	T^\eps \sim \sum_{k=0}^\infty \eps^k T_k, \quad \psi^\eps \sim \sum_{k=0}^\infty \eps^k \psi_k
\end{align*}
into system \eqref{eq:1}-\eqref{eq:2}. Collecting terms with the same order gives 
\begin{align}
	&\left\{\begin{matrix}\langle \psi_0 - T_0^4 \rangle = 0,  \nonumber \\
	\psi_0 - T_0^4  = 0,\end{matrix}\right. \nonumber \\
	&\left\{\begin{matrix}\langle \psi_1 - 4 T_0^3 T_1 \rangle = 0,  \nonumber \\
	\beta\cdot\nabla \psi_0 + \psi_1 - 4T_0^3 T_1 = 0,\end{matrix} \right.  \nonumber \\
	&\left\{\begin{matrix} \Delta T_0 + \langle \psi_2 - 4T_0^3 T_2 - 6 T_0^2 T_1^2 \rangle =0, \\ \beta\cdot \nabla \psi_1 + \psi_2 -  4T_0^3 T_2 - 6 T_0^2 T_1^2 =0, \end{matrix}\right.  \nonumber \\
	&\qquad \vdots  \nonumber \\
	&\left\{\begin{matrix}\Delta T_{k-2} + \langle \psi_k  -\mathcal{C}(T,k)\rangle =0, \\  \beta\cdot \nabla \psi_{k-1} + \psi_k -\mathcal{C}(T,k) = 0,\end{matrix}\right. \label{eq:ii-1}
\end{align}
where $\mathcal{C}(T,k)$ is defined by 
\begin{align*}
	\mathcal{C}(T,k) = \sum_{\substack{i+j+l+m=k\\i,j,l,m\ge 0}} T_i T_j T_l T_m.
\end{align*}
Comparing the two equations in \eqref{eq:ii-1} gives 
\begin{align*}
	\Delta T_{k-2} = \nabla\cdot \langle \beta \psi_{k-1} \rangle.
\end{align*}
Using $\psi_{k-1} = \mathcal{C}(T,k-1) - \beta\cdot \nabla \psi_{k-2}$ in the above equation leads to
\begin{align*}
	\Delta T_{k-2} =- \nabla\cdot \langle \beta \beta \cdot\nabla \psi_{k-2} \rangle.
\end{align*}
Using  $\psi_{k-2} = \mathcal{C}(T,k-2) - \beta\cdot \nabla \psi_{k-3}$ gives 
\begin{align*}
	\Delta T_{k-2} = - \nabla \cdot \langle \beta\beta\cdot \nabla \mathcal{C}(T,k-2) \rangle + \langle (\beta\cdot\nabla)^3 \psi_{k-3}\rangle, 
\end{align*}
which is 
\begin{align*}
	\Delta T_{k-2} + \pi \Delta \mathcal{C}(T,k-2) =  \langle (\beta\cdot\nabla)^3 \psi_{k-3}\rangle.
\end{align*}
Thus we get $(T_k,\psi_k)$ satisfy for any $k\ge 0$,
\begin{align*}
	&\Delta T_k + \pi \Delta \mathcal{C}(T,k) = \langle (\beta\cdot\nabla)^3 \psi_{k-1} \rangle,\\
	&\psi_k = \mathcal{C}(T,k) - \beta\cdot \nabla \psi_{k-1}.
\end{align*}
Note here $T_k,\psi_k$ are zero for any index $k<0$. Considering $\mathcal{C}(T,0)=T_0^4$ and for $k\ge 1$, $\mathcal{C}(T,k)=4T_0T_k + \mathcal{E}(T,k-1)$ with 
\begin{align*}
	\mathcal{E}(T,k-1) =  \sum_{\substack{i+j+l+m=k\\1\le i,j,l,m\le k-1}} T_i T_j T_l T_m,
\end{align*}
which only depends on $T_s$, $s\ge k-1$, we can rewrite the previous equations as for $k=0$,
\begin{align*}
	\Delta T_0 + \pi \Delta T_0^4 = 0,\quad \psi_0 = T_0^4,
\end{align*}
and for $k\ge 1$,
\begin{align*}
	\Delta T_k + \pi\Delta (4T_0^4 T_k) =  \langle (\beta\cdot\nabla)^3 \psi_{k-1} \rangle - \pi \Delta \mathcal{E}(T,k-1),\quad \psi_{k} = 4T_0^3 T_k +  \mathcal{E}(T,k-1) - \beta\cdot \nabla \psi_{k-1}.
\end{align*}
Therefore, the leading order system is a nonlinear elliptic equation while the higher order system is a linear elliptic equation.



\subsection{Boundary layer corrections}
The interior expansions $T^o:=\sum_{k=0}^\infty \eps^k T_k$ and $\psi^o := \sum_{k=0}^\infty \eps^k \psi_k$ provide a good approximation to system \eqref{eq:1}-\eqref{eq:2} inside the domain. Near the boundary we need to introduce a correction term $(T^i,\psi^i)$ describing the effects of boundary layers. Suppose $T^\eps = T^o + T^i$, $\psi^\eps = \psi^o + \psi^i$, then system \eqref{eq:1}-\eqref{eq:2} becomes 
\begin{align*}
	&\eps^2 \Delta (T^o + T^i) + \langle \psi^o+ \psi^i - (T^o + T^i)^4 \rangle =0, \\
	&\eps \beta\cdot \nabla (\psi^o + \psi^i) + \psi^o+ \psi^i - (T^o + T^i)^4 =0.
\end{align*}
Using the fact that $T^o$, $\psi^o$ solves system \eqref{eq:1}-\eqref{eq:2}, the above system becomes  
\begin{align}
	&\eps^2 \Delta T^i + \langle \psi^i - (T^o + T^i)^4 + (T^o)^4 \rangle =0,\label{eq:i1}\\
	&\eps \beta\cdot \nabla\psi^i + \psi^i - (T^o + T^i)^4 + (T^o)^4 =0.\label{eq:i2}
\end{align}

We now follow \cite{wu2015geometric} by  introducing the transformation of variables $(x_1,x_2)\to (r,\theta)$ given by
\begin{align*}
	\left\{\begin{array}[]{ll}
		x_1 = (1-r)\cos\theta, \\
		x_2 = (1-r)\sin\theta, 
	\end{array}\right.
\end{align*}
with $r$ denotes the distance to the boundary $\partial\Omega$. With the coordinate $(r,\theta)$ and $\beta=(\beta_1,\beta_2)$, system \eqref{eq:i1}-\eqref{eq:i2} can be rewritten as 
\begin{align*}
	&\eps^2 \partial_r^2 T^i - \frac{\eps^2}{1-r} \partial_r T^i + \frac{\eps^2}{(1-r)^2}\partial^{2}_\theta T^i + \langle \psi^i - (T^o + T^i)^4 + (T^o)^4 \rangle = 0,\\
	&-\eps (\beta_1 \cos\theta+\beta_2\sin\theta)\partial_r \psi^i - \frac{\eps}{1-r}(\beta_1 \sin\theta - \beta_2 \cos\theta)\partial_\theta \psi^i \nonumber\\
	&\quad + \psi^i - (T^o + T^i)^4 + (T^o)^4  = 0,
\end{align*}
and the boundary conditions are 
\begin{align*}
	&T^i(0,\theta) + T^o(0,\theta)= T_b(\theta),\\
	&\psi^i(0,\theta,\beta) + \psi^o(0,\theta) = \psi_b(\theta,\beta),\quad \text{for } \beta_1 \cos\theta + \beta_2 \sin\theta <0. 
\end{align*}
We introduce a new coordinate $\eta=r/\eps$ and rewrite the above system as 
\begin{align*}
	&\partial_\eta^2 T^i - \frac{\eps}{1-\eps \eta} \partial_\eta T^i + \frac{\eps^2}{(1-\eps\eta)^2} \partial_\theta^2 T^i +   \langle \psi^i - (T^o + T^i)^4 + (T^o)^4 \rangle = 0,\\
	&-(\beta_1 \cos\theta+\beta_2\sin\theta) \partial_\eta \psi^i - \frac{\eps}{1-\eps\eta}(\beta_1\sin\theta - \beta_2\cos\theta) \partial_\theta \psi^i \nonumber\\
	&\quad + \psi^i - (T^o + T^i)^4 + (T^o)^4 =0.
\end{align*}
By introducing the transformation of variable 
\begin{align*}
	\left\{\begin{array}[]{ll}
		\beta_1 = -\sin\xi, \\
		\beta_2 = -\cos\xi,
	\end{array}\right.
\end{align*}
we can rewrite the system as 
\begin{align}
	&\partial_\eta^2 T^i - \frac{\eps}{1-\eps \eta} \partial_\eta T^i + \frac{\eps^2}{(1-\eps\eta)^2} \partial_\theta^2 T^i +   \langle \psi^i - (T^o + T^i)^4 + (T^o)^4 \rangle = 0, \label{eq:i/1}\\
	&\sin(\theta+\xi) \partial_\eta \psi^i - \frac{\eps}{1-\eps\eta}\cos(\theta+\xi) \partial_\theta\psi^i  + \psi^i - (T^o + T^i)^4 + (T^o)^4 =0. \label{eq:i/2}
\end{align}
Since we are considering boundary layers, we may take Taylor's expansion of $(T^o=T^o(\eps\eta,r),\psi^o=\psi^o(\eps\eta,r))$ around the boundary $\eta=0$ and get 
\begin{align*}
	T^o = \sum_{k=0}^\infty \eps^k P_k,\quad \psi^o = \sum_{k=0}^\infty \eps^k Q_k,
\end{align*}
where 
\begin{align*}
	P_k (\eta,\theta) = \sum_{l=0}^k\frac{\eta^l}{l!} \frac{\partial^l}{\partial_r^l} T_{k-l}(0,\theta),\quad Q_k(\eta,\theta,\xi) = \sum_{l=0}^k\frac{\eta^l}{l!} \frac{\partial^l}{\partial_r^l} \psi_{k-l}(0,\theta)
\end{align*}
Taking the ansatz
\begin{align}\label{eq:ansatz2}
	T^i = \sum_{k=1}^\infty \eps^k \bar{T}_k,\quad \psi^i = \sum_{k=1}^\infty \eps^k \bar{\psi}_k
\end{align}
into the system \eqref{eq:i/1}-\eqref{eq:i/2} and collecting terms with the same order, we get
\begin{align}
	&\left\{\begin{array}{ll}
		\partial_\eta^2 \bar{T}_0 + \langle \bar{\psi}_0 - (\bar{T}_0+P_0)^4 + P_0^4 \rangle = 0,\\
		\sin(\theta+\xi)\partial_\eta \bar{\psi}_0 + \bar{\psi}_0 - (\bar{T}_0+P_0)^4 + P_0^4 =0,
	\end{array}\right. \label{eq:bar/1} \\
	&\left\{\begin{array}[]{ll}
		\partial_\eta^2 \bar{T}_1 + \langle \bar{\psi}_1 - 4(P_0+\bar{T}_0)^3 (P_1+\bar{T}_1) + 4 P_0^3 P_1 \rangle = \frac{1}{1-\eps\eta} \partial_\eta \bar{T}_0, \\
		\sin(\theta+\xi)\partial_\eta \bar{\psi}_1 + \bar{\psi}_1 - 4(P_0+\bar{T}_0)^3 (P_1+\bar{T}_1) + 4 P_0^3 P_1 = \frac{1}{1-\eps\eta}\cos(\theta+\xi) \partial_\theta \bar{\psi}_0, \\
	\end{array}\right.\\
	&\qquad \vdots \nonumber\\
	&\left\{\begin{array}{ll}
		\partial_\eta^2 \bar{T}_k + \langle \bar{\psi}_k - 4 (P_0+\bar{T}_0)^3 (P_k+\bar{T}_k) + 4 P_0^3 P_k \rangle = \frac{1}{1-\eps\eta}\partial_\eta \bar{T}_{k-1} - \frac{1}{(1-\eps\eta)^2} \partial_\theta \bar{T}_{k-2} \\
		\qquad + \langle \mathcal{E}(P+\bar{T},k-1) - \mathcal{E}(P,k-1) \rangle , \\
		\sin(\theta+\xi) \partial_\eta \bar{\psi}_k + \bar{\psi}_k - 4 (P_0+\bar{T}_0)^3 (P_k+\bar{T}_k) + 4 P_0^3 P_k = \frac{1}{1-\eps\eta}\cos(\theta+\xi) \partial_\theta \bar{\psi}_{k-1} \\
		\quad +  \mathcal{E}(P+\bar{T},k-1) - \mathcal{E}(P,k-1),
	\end{array}\right. \label{eq:bar/3}
\end{align}
with the boundary conditions 
\begin{align*}
	&\bar{T}_0(0,\theta) + T_0(0,\theta) = T_b(\theta),\\
	&\bar{\psi}_0(0,\theta,\xi) + \psi_0(0,\theta,\xi) = \psi_b(\theta,\xi),\quad \text{for }\sin(\theta+\xi) >0,
\end{align*}
and for $k\ge 1$,
\begin{align*}
	&\bar{T}_k(0,\theta) + T_k(0,\theta) =0,\\
	&\bar{\psi}_k(0,\theta,\xi) + \psi_k(0,\theta,\xi) = 0,\quad \text{for }\sin(\theta+\xi) >0,
\end{align*}
However, the boundary condition $T_k(0,\theta)$, $\psi_k(0,\theta)$ are not known. Hence we introduce the variables 
\begin{align*}
	&\tilde{T}_k = \bar{T}_k + T_k(\eps\eta=0)= \bar{T}_k + P_k(\eta=0),\\
	&\tilde{\psi}_k = \bar{\psi}_k + 4 P_0^3P_k(\eta=0).
\end{align*}
We denote $P_k(0)=P_k(\eta=0)$, $Q_k(0)=Q_k(\eta=0)$. Equations \eqref{eq:bar/1}-\eqref{eq:bar/3} can be written as 
\begin{align} 
	&\left\{\begin{array}{ll}
		\partial_\eta^2 \tilde{T}_0 + \langle \tilde{\psi}_0 - \tilde{T}_0^4 \rangle =0, \\
		\sin(\theta+\xi) \partial_\eta \tilde{\psi}_0 + \tilde{\psi}_0 - \tilde{T}_0^4  = 0, 
	\end{array} \right. \label{eq:tilde/1} \\
	&\left\{\begin{array}[]{ll}
		\partial_\eta^2 \tilde{T}_1 + \tilde{\psi}_1 - 4 \tilde{T}_0^3 \tilde{T}_1 = \frac{1}{1-\eps\eta}\partial_\eta \bar{T}_0 + \langle 4(\tilde{T}_0^3-P_0^3)(P_1 - P_1(0)) \rangle, \\
		\sin(\theta+\xi) \partial_\eta \tilde{\psi}_1 + \tilde{\psi}_1 - 4 \tilde{T}_0^3 \tilde{T}_1 = \frac{1}{1-\eps\eta}\cos(\theta+\xi)\partial_\theta \bar{\psi}_0 +  4(\tilde{T}_0^3-P_0^3)(P_1 - P_1(0))
	\end{array}\right. \label{eq:tilde/2}\\
	&\qquad \vdots\nonumber\\
	&\left\{\begin{array}[]{ll}
		\partial_\eta^2 \tilde{T}_k + \tilde{\psi}_k - 4\tilde{T}_0^3 \tilde{T}_k = \frac{1}{1-\eps\eta}\partial_\eta \tilde{T}_{k-1} + \frac{1}{(1-\eps\eta)^2} \partial_\theta^2 \tilde{T}_{k-2} + \langle 4(\tilde{T}_0^3 - P_0^3)(P_k-P_k(0))  \rangle \\
		\quad + \langle \mathcal{E}(\bar{T}+P,k-1) -\mathcal{E}(P,k-1),\rangle \\
		\sin(\theta+\xi)\partial_\eta \tilde{\psi}_k + \tilde{\psi}_k - 4\tilde{T}_0^3 \tilde{T}_k = \frac{1}{1-\eps\eta}\cos(\theta+\xi)\partial_\theta \bar{\psi}_{k-1} + 4(\tilde{T}_0^3 - P_0^3)(P_k-P_k(0))\\
		\quad + \mathcal{E}(\bar{T}+P,k-1) -\mathcal{E}(P,k-1) .
	\end{array}\right. \label{eq:tilde/3}
\end{align}
The boundary conditions are 
\begin{align*}
	\tilde{T}_0(0,\theta) = T_b(\theta),\quad \tilde{\psi}_0(0,\theta,\xi) = \psi_b(\theta,\xi), \quad \text{for }\sin(\theta+\xi) >0,
\end{align*}
and 
\begin{align*}
	\tilde{T}_k(0,\theta)= 0,\quad \tilde{\psi}_k(0,\theta,\xi) = \beta\cdot \nabla\psi_{k-1}(0) - \mathcal{E}(T(0),k-1), \quad \text{for }\sin(\theta+\xi) >0.
\end{align*}
The reason why the boundary condition for $\tilde{\psi}_k$ is not zero is because we define $\tilde{\psi}_k = \bar{\psi}_k + 4P_0^3 P_k(0) = \bar{\psi}_k + 4P_0^3 T_k(0)$ so that 
\begin{align*}
	\bar{\psi}_k(\eta=0)+ \psi_k(\eps\eta=0)& = \tilde{\psi}_k(0)-4P_0^3 T_k(0) + \psi_k(0)\nonumber
	\\
	& = \beta\cdot \nabla\psi_{k-1}(0) - \mathcal{E}(T(0),k-1)  -4P_0^3 T_k(0) + \psi_k(0) =0 ,
\end{align*}
due to the relation $\psi_k = -\beta\cdot \nabla \psi_{k-1} + 4T_0^3 T_k +\mathcal{E}(T,k-1)$. Here we call \eqref{eq:tilde/1}-\eqref{eq:tilde/3} the \emph{Milne problems} corresponding to \eqref{eq:1}-\eqref{eq:2} and \eqref{eq:bar/1}-\eqref{eq:bar/3} the \emph{boundary layer correction}.

\subsection{Construction of approximate solutions} 
Let $\delta >0$ be a small constant. Define the cut-off function
\begin{align*}
	\chi_0(r) = \left\{\begin{array}[]{ll}
		1, &\text{for } 0 \le r \le \tfrac12 \delta, \\
		0, &\text{for } r \ge \tfrac{3}{4}\delta, \\
		\in (0,1), & \text{otherwise}, 
	\end{array} \right.
\end{align*}
and 
\begin{align*}
	\chi(r) = \left\{\begin{array}[]{ll}
		1, &\text{for } 0 \le r \le \tfrac14 \delta, \\
		0, &\text{for } r \ge \tfrac{3}{8}\delta, \\
		\in (0,1), & \text{otherwise}.
	\end{array} \right.
\end{align*}
We begin by constructing the leading order approximations. The zeroth-order boundary layer correction $(\bar{T}_0,\bar{\psi}_0)$ is provided by solving 
\begin{align}\label{eq:u0}
	\left\{\begin{array}[]{ll}
	 \bar{T}_0(\eta,\theta) = \chi(\eps\eta)(\tilde{T}_0(\eta,\theta) - \tilde{T}_{0,\infty}(\infty,\theta)),~ \bar{\psi}_0(\eta,\theta,\xi) =\chi(\eps\eta)(\tilde{\psi}_0(\eta,\theta,\xi) - \tilde{\psi}_{0,\infty}(\infty,\theta,\xi)) \\
	\partial_\eta^2 \tilde{T}_0 + \langle \tilde{\psi}_0 - \tilde{T}_0^4 \rangle =0, \\
	\sin(\theta+\xi) \partial_\eta \tilde{\psi}_0 + \tilde{\psi}_0 - \tilde{T}_0^4 =0,
	\end{array}\right.
\end{align}
combined to the following boundary condition 
\[
\tilde{T}_0(0,\theta) = T_b(\theta),\quad \tilde{\psi}_0(0,\theta,\xi)=\psi_b(\theta,\xi),\text{ for }\sin(\theta+\xi)>0.
\]

Assuming sufficient regularity on the boundary data, one can prove as $\eta\to 0$, there exists a constant function $\tilde{T}_\infty =\tilde{T}_{0,\infty}(\theta)$ independent of $\eta$ such that $\lim_{\eta\to\infty} \tilde{T}_0(\eta)= \tilde{T}_{0,\infty}$ and $\lim_{\eta\to\infty} \tilde{\psi}_0(\eta)=\tilde{T}^4_{0,\infty}$. This provides the boundary condition for the interior expansions. The leading order of the interior expansion is 
\begin{equation}\label{eq:ap0_inside}
\begin{aligned}
	&\Delta T_0 + \pi \Delta T_0^4 = 0,\quad \text{in }\Omega\\
	&\psi = T_0^4 \qquad \text{in }\Omega\times\mathbb{S}^1 \\
	& T_0(x) = \tilde{T}_{0,\infty}, \quad \text{for }x\in\partial\Omega.
\end{aligned}
\end{equation}
We next construct the first order approximations. The boundary layer correction $(\bar{T}_1,\bar{\psi}_1)$ is given by solving 
\begin{align} \label{eq:u1}
	\left\{\begin{array}[]{ll}
		\bar{T}_1(\eta,\theta) = \chi(\eps\eta)(\tilde{T}_1(\eta,\theta) - \tilde{T}_{1,\infty}(\infty,\theta)),~ \bar{\psi}_1(\eta,\theta,\xi) =\chi(\eps\eta)(\tilde{\psi}_1(\eta,\theta,\xi) - \tilde{\psi}_{1,\infty}(\infty,\theta,\xi)) \\\\
		\partial_\eta^2 \tilde{T}_1 + \tilde{\psi}_1 - 4\tilde{T}_0^3 \tilde{T}_1 = \frac{\chi_0(\eps\eta)}{1-\eps\eta}\partial_\eta \tilde{T}_{0} \langle 4(\tilde{T}_0^3 - P_0^3)(P_1-P_1(0))  \rangle ,\\\\
		\sin(\theta+\xi)\partial_\eta \tilde{\psi}_1 + \tilde{\psi}_1 - 4\tilde{T}_0^3 \tilde{T}_1 = \frac{\chi_0(\eps\eta)}{1-\eps\eta}\cos(\theta+\xi)\partial_\theta \bar{\psi}_{0} + 4(\tilde{T}_0^3 - P_0^3)(P_1-P_1(0)) ,\\\\
		\text{B.C.: } 	\tilde{T}_1(0,\theta)= 0,\quad \tilde{\psi}_1(0,\theta,\xi) = \beta\cdot \nabla\psi_{0}(0) , \quad \text{for }\sin(\theta+\xi) >0.
	\end{array}\right.
\end{align}
The interior expansion $(T_0,\psi_0)$ is given by 
\begin{equation*}
\begin{aligned}
	&\Delta T_1 + \pi \Delta (4T_0^3 T_1) = 0, \quad \psi_1 = 4T_0^3 T_1 - \beta\cdot \nabla \psi_0,\quad \text{in }\Omega\times\mathbb{S}^1, \\
	&T_1(x) = \tilde{T}_{1,\infty}, \quad \text{for }x\in\partial\Omega.
\end{aligned}
\end{equation*}
The boundary layer corrections $(\bar{T}_k,\bar{\psi}_k)$, $k\ge 1$ are given by 
\begin{align*}
	\left\{\begin{array}[]{ll}
		\bar{T}_k(\eta,\theta) = \chi(\eps\eta)(\tilde{T}_k(\eta,\theta) - \tilde{T}_{k,\infty}(\infty,\theta)),~ \bar{\psi}_k(\eta,\theta,\xi) =\chi(\eps\eta)(\tilde{\psi}_k(\eta,\theta,\xi) - \tilde{\psi}_{k,\infty}(\infty,\theta,\xi)) \\\\
		\partial_\eta^2 \tilde{T}_k + \tilde{\psi}_k - 4\tilde{T}_0^3 \tilde{T}_k = \frac{\chi_0(\eps\eta)}{1-\eps\eta}\partial_\eta \tilde{T}_{k-1} + \frac{\chi_0(\eps\eta)}{(1-\eps\eta)^2} \partial_\theta^2 \tilde{T}_{k-2} + \langle 4(\tilde{T}_0^3 - P_0^3)(P_k-P_k(0))  \rangle \nonumber\\\\
		\qquad+ \langle \mathcal{E}(\bar{T}+P,k-1) -\mathcal{E}(P,k-1),\rangle \\\\
		\sin(\theta+\xi)\partial_\eta \tilde{\psi}_k + \tilde{\psi}_k - 4\tilde{T}_0^3 \tilde{T}_k = \frac{\chi_0(\eps\eta)}{1-\eps\eta}\cos(\theta+\xi)\partial_\theta \bar{\psi}_{k-1} + 4(\tilde{T}_0^3 - P_0^3)(P_k-P_k(0))\nonumber\\\\
		\qquad + \mathcal{E}(\bar{T}+P,k-1) -\mathcal{E}(P,k-1),\\\\
		\text{B.C.: } 	\tilde{T}_k(0,\theta)= 0,\quad \tilde{\psi}_k(0,\theta,\xi) = \beta\cdot \nabla\psi_{k-1}(0) - \mathcal{E}(T(0),k-1), \quad \text{for }\sin(\theta+\xi) >0.
	\end{array}\right.
\end{align*}
The interior expansions $(T_k,\psi_k)$, $k\ge 1$ are solved by 
\begin{align*}
	&\Delta T_k + \pi \Delta (4T_0^3 T_k) =-\pi \Delta \mathcal{E}(T,k-1) +  \langle (\beta\cdot\nabla)^3 \psi_{k-1} \rangle,\\
	&\psi_k = 4T_0^3 T_k + \mathcal{E}(T,k-1) - \beta\cdot \nabla \psi_{k-1}, \\
	&\text{B.C.: } T_k(x) = \tilde{T}_{k,\infty}, \quad \text{for }x\in\partial\Omega.
\end{align*}
It is expected that when $\eps$ is small, the addition of the boundary layer corrects well describes the boundary layer effects and we should have for example 
\begin{align*}
	\|T^\eps - T_0 - \bar{T}_0 \|_{L^\infty(\Omega)} = O(\eps),\quad \|\psi^\eps - \psi_0 - \bar{\psi}_0\|_{L^\infty(\Omega\times\mathbb{S}^1)} = O(\eps).
\end{align*}
For the flat boundary case, this has been proved in \cite{ghattassi2022convergence}. However, for the unit disk we consider here, a similar problem arises as the case of linear transport equation in \cite{wu2015geometric}. When solving \eqref{eq:u1}, the term $\cos(\theta+\xi) \partial_\theta \tilde{\psi}_0$ appears on the right of the equation of $\tilde{\psi}_1$. In order for $\tilde{\psi}_1 \in L^\infty(\mathbb{R}_+\times [-\pi,\pi)\times [-\pi,\pi))$, we need 
\begin{align}\label{eq:problematicterm}
	\frac{\chi_0(\eps\eta)}{1-\eps\eta} \cos(\theta+\xi) \partial_\theta \tilde{\psi}_0 \in L^\infty(\mathbb{R}_+\times [-\pi,\pi)\times [-\pi,\pi)).
\end{align}
Considering that the support of $\chi_0$ depends on $\eps$, we need $\partial_\theta \tilde{\psi}_0 \in L^\infty$. Note that $\partial_\theta \tilde{\psi}_0$ is a solution to the equations 
\begin{align*}
	&\partial_\eta^2 (\partial_\theta \tilde{T}_0) + \langle \partial_\theta \tilde{\psi}_0 - 4 \tilde{T}_0^3 \partial_\theta \tilde{T}_0 \rangle =0, \\
	&\sin(\theta+\xi) \partial_\eta (\partial_\theta \tilde{\psi}_0) + \partial_\theta \tilde{\psi}_0 - 4 \tilde{T}_0^3 \partial_\theta \tilde{T}_0 = - \cos(\theta+\xi) \partial_\eta \tilde{\psi}_0.
\end{align*}
Thus in order for $\partial_\theta \tilde{\psi}_0\in L^\infty$ we need $\partial_\eta \tilde{\psi}_0 \in L^\infty$. However, this fails to be hold, according to the following lemma.
\begin{lemma}
	For the Milne problem 
	\begin{align}\label{eq:ml/1}
		\left\{ \begin{array}{lll}
			&\partial_\eta^2 \tilde{T}_0 + \langle \tilde{\psi}_0 - \tilde{T}_0^4 \rangle =0,\\
			&\sin(\theta+\xi)\partial_\eta\tilde{\psi}_0 + \tilde{\psi}_0 - \tilde{T}_0^4 =0, \\
			&\mbox{B.C.:}~ \tilde{T}_0(0,\theta)=T_b(\theta),\quad \tilde{\psi}_0(0,\theta,\xi) = \psi_b(\theta,\xi),\quad \text{for }\sin(\theta+\xi) >0,
		\end{array}\right.
	\end{align}
	if $T_b=0$, $\psi_b(\theta,\xi)=\cos(2(\theta+\xi))$, then 
	\begin{align*}
		\partial_\eta \tilde{\psi}_0 \not\in L^\infty([0,\infty)\times[-\pi,\pi)\times[-\pi,\pi)).
	\end{align*}
\end{lemma}
\begin{proof}
	For $\eta=0$, according to \eqref{eq:ml/1}
	\begin{align*}
		\partial_\eta \tilde{\psi}_0(0,\theta,\xi) = - \frac{\tilde{\psi}_0(0,
		\theta,\xi) - \tilde{T}_0^4(0,\theta)}{\sin(\theta+\xi)}.
	\end{align*}
	If we take $\theta+\xi=0$, then 
	\begin{align*}
		\lim_{\xi\to -\theta^+}	\partial_\eta \tilde{\psi}_0(0,\theta,\xi) = - \lim_{\xi\to -\theta^+} \frac{\tilde{\psi}_0(0,
		\theta,\xi) - \tilde{T}_0^4(0,\theta)}{\sin(\theta+\xi)} = - \frac{1-0}{0} = -\infty,
	\end{align*}
	hence $\partial_\eta \tilde{\psi}_0 \not\in L^\infty([0,\infty)\times [-\pi,\pi)\times [-\pi,\pi))$. 
\end{proof}



\subsection{Geometric corrections} 
As can be seen from above analysis, the term \eqref{eq:problematicterm} does not lie in $L^\infty$. In order to overcome this difficulty, we follow the work \cite{wu2015geometric} and introduce a geometric correction to the boundary layer corrections. First we introduce the transformation of variables 
$(\theta,\xi) \to (\theta,\phi)$ defined by 
\begin{align*}
	\left\{\begin{array}[]{ll}
		\theta=\theta,\\
		\phi = \theta+\xi.
	\end{array}\right.
\end{align*}
With the new coordinate,  $\partial_\theta f(\theta,\xi)$ becomes 
\begin{align*}
	\partial_\theta f(\theta,\xi) = \partial_\theta  f(\theta,\phi) + \partial_\phi f(\theta,\phi), 
\end{align*}
and $\partial_{\theta}^2 f$ becomes
\begin{align*}
	\partial_\theta^2 f = \partial_\theta^2 f(\theta,\phi) + \partial_\phi^2 f(\theta,\phi) + 2 \partial_{\theta\phi} f(\theta,\phi).
\end{align*}
Note that any function $h=h(\theta)$ remains the same in the new coordinate. 

Equations \eqref{eq:i/1}-\eqref{eq:i/2} become 
\begin{align*}
	&\partial_\eta^2 T^i - \frac{\eps}{1-\eps\eta}\partial_\eta T^i + \frac{\eps^2}{(1-\eps\eta)^2} \partial_\theta^2 T^i  +  \langle \psi^i - (T^o + T^i)^4 + (T^o)^4 \rangle = 0,\\\\
	&\sin\phi \partial_\eta \psi^i - \frac{\eps}{1-\eps\eta} \cos\phi (\partial_\theta \psi^i + \partial_\phi \psi^i) + \psi^i - (T^o + T^i)^4 + (T^o)^4 =0.
\end{align*}
Taking the ansatz \eqref{eq:ansatz2} into the above equations and collect the terms with the same order while treating the terms $\eps^2 \partial_\phi^2 \bar{T}_k$, $\eps \partial_\phi \bar{\psi}_{k}$ to be $O(1)$ gives 
\begin{equation}\label{Eqcorc0}
\begin{aligned}
	&\left\{\begin{array}[]{ll}
		\partial_\eta^2 \bar{T}_0 - \frac{\eps}{1-\eps\eta}\partial_\eta \bar{T}_0+ \bar{\psi}_0 - (\bar{T}_0+P_0)^4 + P_0^4 \rangle  = 0,\\
		\sin\phi \partial_\eta \bar{\psi}_0 - \frac{\eps}{1-\eps\eta} \cos\phi \partial_\phi \bar{\psi}_0 + \bar{\psi}_0 - (\bar{T}_0+P_0)^4 + P_0^4 = 0,
	\end{array}\right.\\\\
	&\left\{\begin{array}[]{ll}
		\partial_\eta^2 \bar{T}_1 - \frac{\eps}{1-\eps\eta} \partial_\eta \bar{T}_1 + \langle \bar{\psi}_1 - 4(P_0+\bar{T}_0)^3 (P_1+\bar{T}_1) + 4 P_0^3 P_1 \rangle = 0,\\ 
		\\
		\sin \phi \partial_\eta \bar{\psi}_1 - \frac{\eps}{1-\eps\eta} \cos\phi \partial_\phi \bar{\psi}_1 + \bar{\psi}_1 - 4(P_0+\bar{T}_0)^3 (P_1+\bar{T}_1) + 4 P_0^3 P_1 = \frac{1}{1-\eps\eta}\cos(\theta+\xi) \partial_\theta \bar{\psi}_0, \\
	\end{array}\right.\\
	&\qquad \vdots \\\\
	&\left\{\begin{array}{ll}
		\partial_\eta^2 \bar{T}_k - \frac{\eps}{1-\eps\eta}\partial_\eta \bar{T}_k + \langle \bar{\psi}_k - 4 (P_0+\bar{T}_0)^3 (P_k+\bar{T}_k) + 4 P_0^3 P_k \rangle = 
		- \frac{1}{(1-\eps\eta)^2} \partial_\theta \bar{T}_{k-2} \\\\
		\qquad + \langle \mathcal{E}(P+\bar{T},k-1) - \mathcal{E}(P,k-1) \rangle, \\\\
		\sin \phi \partial_\eta \bar{\psi}_k - \frac{\eps}{1-\eps\eta} \cos\phi \partial_\phi \bar{\psi}_k + \bar{\psi}_k - 4 (P_0+\bar{T}_0)^3 (P_k+\bar{T}_k) + 4 P_0^3 P_k = \frac{1}{1-\eps\eta}\cos(\theta+\xi) \partial_\theta \bar{\psi}_{k-1} \\\\
		\qquad +  \mathcal{E}(P+\bar{T},k-1) - \mathcal{E}(P,k-1),
	\end{array}\right. 
\end{aligned}
\end{equation}
In the above, the terms 
\begin{align*}
	-\frac{\eps}{1-\eps\eta}\partial_\eta \bar{T}_k,\quad -\frac{\eps}{1-\eps\eta}\cos\phi \partial_\phi \bar{\psi}_k
\end{align*}
are included in the boundary layer problem, which is used to overcome the singularity of the normal derivatives of boundary layer solutions. The second term was introduced in \cite{wu2015geometric}. Here since we need to consider the elliptic equation, the first term should also be introduced.

Following a similar procedure, we can reformulate the equations using $\tilde{T}_k,\tilde{\psi}_k$. Define $F(\eps;\eta)$ by 
\begin{align}\label{eq:Fdef}
	F(\eps;\eta)=-\frac{\eps \chi_0(\eps\eta)}{1-\eps\eta}.
\end{align} 
The approximate solutions can be constructed as follows. 

The zeroth-order boundary layer correction $\bar{T}_0^\eps,\bar{\psi}_0^\eps$ with geometric correction is given by 
\begin{align}\label{eq:g0}
	\left\{\begin{array}[]{ll}
	 \bar{T}^\eps_0(\eta,\theta) = \chi(\eps\eta)(\tilde{T}^\eps_0(\eta,\theta) - \tilde{T}_{0,\infty}(\infty,\theta)),~ \bar{\psi}^\eps_0(\eta,\theta,\xi) =\chi(\eps\eta)(\tilde{\psi}^\eps_0(\eta,\theta,\xi) - \tilde{\psi}^\eps_{0,\infty}(\infty,\theta,\xi)) \\\\
	\partial_\eta^2 \tilde{T}^\eps_0 + F(\eps;\eta) \partial_\eta \tilde{T}^\eps_0 + \langle \tilde{\psi}_0^\eps - (\tilde{T}^\eps_0)^4 \rangle =0, \\\\
	\sin\phi \partial_\eta \tilde{\psi}_0^\eps + F(\eps;\eta) \cos\phi \partial_\phi \tilde{\psi}_0^\eps + \tilde{\psi}_0^\eps - (\tilde{T}_0^\eps)^4 =0,\\\\
	\text{B.C.: } \tilde{T}^\eps_0(0,\theta) = T_b(\theta),\quad \tilde{\psi}_0^\eps(0,\theta,\phi)=\psi_b(\theta,\phi),\text{ for }\sin \phi>0.
	\end{array}\right.
\end{align}

The first order boundary layer correction $\bar{T}_1^\eps,\bar{\psi}_1^\eps$ with geometric correction is given by 
\begin{align} \label{eq:g1}
	\left\{\begin{array}[]{ll}
		\bar{T}_1^\eps(\eta,\theta) = \chi(\eps\eta)(\tilde{T}_1^\eps(\eta,\theta) - \tilde{T}_{1,\infty}(\infty,\theta)),~ \bar{\psi}_1^\eps(\eta,\theta,\xi) =\chi(\eps\eta)(\tilde{\psi}_1^\eps(\eta,\theta,\xi) - \tilde{\psi}^\eps_{1,\infty}(\infty,\theta,\xi)) \\\\
		\partial_\eta^2 \tilde{T}_1^\eps + F(\eps;\eta) \partial_\eta \tilde{T}_1^\eps + \tilde{\psi}_1^\eps - 4(\tilde{T}^\eps_0)^3 \tilde{T}_1^\eps =   \langle 4((\tilde{T}^\eps_0)^3 - (P_0^\eps)^3)(P_1^\eps-P_1^\eps(0))  \rangle ,\\\\
		\sin\phi \partial_\eta \tilde{\psi}_1^\eps + F(\eps;\eta)\cos\phi\partial_\phi \tilde{\psi}_1^\eps + \tilde{\psi}_1^\eps - 4(\tilde{T}_0^\eps)^3 \tilde{T}_1^\eps = \frac{\chi_0(\eps\eta)}{1-\eps\eta}\cos\phi \partial_\theta \bar{\psi}_{0}^\eps \\\\
		\qquad  + 4((\tilde{T}^\eps_0)^3 - (P_0^\eps)^3)(P_1^\eps-P_1^\eps(0))  ,\\\\
		\text{B.C.: } 	\tilde{T}^\eps_1(0,\theta)= 0,\quad \tilde{\psi}_1^\eps (0,\theta,\phi) = \beta\cdot \nabla\psi^\eps_{0}(0) , \quad \text{for }\sin \phi >0.
	\end{array}\right.
\end{align}
The boundary layer corrections with geometric corrections $(\bar{T}^\eps_k,\bar{\psi}_k^\eps)$, $k\ge 1$ are given by 
\begin{equation}
\begin{aligned}\label{eq:g2}
	\left\{\begin{array}[]{ll}
		\bar{T}_k^\eps(\eta,\theta) = \chi(\eps\eta)(\tilde{T}_k^\eps(\eta,\theta) - \tilde{T}^\eps_{k,\infty}(\infty,\theta)),~ \bar{\psi}_k^\eps(\eta,\theta,\xi) =\chi(\eps\eta)(\tilde{\psi}_k^\eps(\eta,\theta,\xi) - \tilde{\psi}^\eps_{k,\infty}(\infty,\theta,\xi)) \\\\
		\partial_\eta^2 \tilde{T}^\eps_k + F(\eps;\eta)\partial_\eta \tilde{T}^\eps_k + \tilde{\psi}_k^\eps - 4(\tilde{T}_0^\eps)^3 \tilde{T}_k^\eps = + \frac{\chi_0(\eps\eta)}{(1-\eps\eta)^2} \partial_\theta^2 \tilde{T}_{k-2}^\eps + \langle 4((\tilde{T}^\eps_0)^3 - (P_0^\eps)^3)(P^\eps_k-P^\eps_k(0))  \rangle \\\\
		\qquad + \langle \mathcal{E}(\bar{T}^\eps+P^\eps,k-1) -\mathcal{E}(P^\eps,k-1)\rangle \\\\
		\sin\phi\partial_\eta \tilde{\psi}_k^\eps + F(\eps;\eta) \cos\phi \partial_\phi \tilde{\psi}^\eps_k+ \tilde{\psi}_k^\eps - 4(\tilde{T}_0^\eps)^3 \tilde{T}_k^\eps = \frac{\chi_0(\eps\eta)}{1-\eps\eta}\cos\phi\partial_\theta \bar{\psi}_{k-1} \\\\
		\qquad + 4((\tilde{T}^\eps_0)^3 - (P_0^\eps)^3)(P^\eps_k-P^\eps_k(0))+ \mathcal{E}(\bar{T}^\eps+P^\eps,k-1) -\mathcal{E}(P^\eps,k-1),\\\\
		\text{B.C.: } 	\tilde{T}^\eps_k(0,\theta)= 0,\quad \tilde{\psi}_k^\eps(0,\theta,\phi) = \beta\cdot \nabla\psi_{k-1}(0) - \mathcal{E}(T(0),k-1), \quad \text{for }\sin \phi >0.
	\end{array}\right.
\end{aligned}
\end{equation}

We will show later that
\[
\lim_{\eta\to\infty} \tilde{T}_k^\eps(\eta,\theta) = \tilde{T}_{k,\infty}^\eps, \quad \text{and}\,\,\,\lim_{\eta\to\infty} \tilde{\psi}_k^\eps(\eta,\theta,\phi) = \tilde{\psi}^\eps_{k,\infty}\quad \text{exist.}
\]
The leading order interior expansion $({T}_0^\eps,{\psi}_0^\eps)$ is given by 
\begin{align} \label{eq:interior/1}
	\left\{\begin{array}{lll}
		\Delta T^\eps_0 + \pi \Delta (T^\eps_0)^4 = 0,\quad \text{in }\Omega,\\
		\psi_0^\eps = (T_0^\eps)^4 \quad \text{in }\Omega\times\mathbb{S}^1\\
		\text{B.C.: } T_0^\eps(x) = \tilde{T}^\eps_{0,\infty}, \quad \text{for }x\in\partial\Omega.
	\end{array}\right.
\end{align}
The first order interior expansion $(T_1^\eps,\psi_1^\eps)$ is given by 
\begin{align}\label{eq:interior/2}
	\left\{\begin{array}{lll}
	\Delta T_1^\eps + \pi \Delta (4(T_0^\eps)^3 T_1^\eps) = 0,\quad \text{in }\Omega\times\mathbb{S}^1,\\
	\psi_1^\eps= 4 (T_0^\eps)^3 T_1^\eps - \beta\cdot \nabla \psi_0^\eps,\quad \text{in }\Omega\times\mathbb{S}^1,\\
	\text{B.C.: } T_1^\eps(x)=\tilde{T}^\eps_{1,\infty},\quad \text{for }x\in\partial\Omega.
	\end{array}\right.
\end{align}
The higher order interior expansion $(T_k^\eps,\psi_k^\eps)$, $k\ge 1$ is given by 
\begin{align}\label{eq:interior/3}
	\left\{\begin{array}{lll}
		\Delta T_k^\eps + \pi\Delta (4(T_0^\eps)^4 T_k^\eps) =  \langle (\beta\cdot\nabla)^3 \psi_{k-1} \rangle - \pi \Delta \mathcal{E}(T^\eps,k-1),\\
		\psi_{k}^\eps = 4(T_0^\eps)^3 T_k^\eps +  \mathcal{E}(T^\eps,k-1) - \beta\cdot \nabla \psi_{k-1}^\eps,\\
		\text{B.C.: } T_k^\eps(x)=\tilde{T}^\eps_{k,\infty},\quad \text{for }x\in\partial\Omega.
	\end{array}\right.
\end{align}

Here we call the equations of $(\tilde{T}_k^\eps,\tilde{\psi}_k^\eps)$ the \emph{Milne problem with geometric corrections} and $(\bar{T}^\eps,\bar{\psi}^\eps)$ the \emph{boundary layer solution with geometric corrections}, and $(T_k^\eps,\psi_k^\eps)$ the interior solutions.
\subsection{Approximate solution with geometric corrections}
We thus obtain the approximate solution with geometric corrections by adding the interior solution and the boundary layer solution with geometric corrections as
\begin{align}\label{ApproEs}
	T^{a,\eps}= \sum_{k=0}^N \eps^k (T_k^\eps + \bar{T}_k^\eps),\quad \psi^{a,\eps} = \sum_{k=0}^N \eps^k (\psi_k^\eps + \bar{\psi}_k^\eps),
\end{align}
where $T_k,\psi_k$ are given by \eqref{eq:interior/1}-\eqref{eq:interior/2} and $\bar{T}_k^\eps,\bar{\psi}_k^\eps$ are given by \eqref{eq:g0}-\eqref{eq:g2}.

One can show that $(T^{a,\eps},\psi^{a,\eps})$ satisfies system \eqref{eq:1}-\eqref{eq:2} with boundary conditions \eqref{eq:1b}-\eqref{eq:2b} in the perturbative sense.
\begin{lemma}\label{lem:errs}
	Let $(T^{a,\eps},\psi^{a,\eps})$ be the pair of functions constructed above. Then formally, it satisfy
	\begin{align*}
		&\eps^2 \Delta T^{a,\eps} + \langle \psi^{a,\eps} - (T^{a,\eps})^4 \rangle = \mathcal{R}_1(T^{a,\eps},\psi^{a,\eps}), \\
		&\eps \beta\cdot\nabla \psi^{a,\eps} + \psi^{a,\eps} - (T^{a,\eps})^4 = \mathcal{R}_2(T^{a,\eps},\psi^{a,\eps}),
	\end{align*}
	with boundary conditions 
	\begin{align*}
		&T^{a,\eps}(x)=T_b(x),\quad \text{for any } x\in \partial\Omega, \\
		&\psi^{a,\eps}(x,\beta)=\psi_b(x,\beta),\quad \text{for any }(x,\beta) \in \Gamma_-,
	\end{align*}
	where $ \mathcal{R}_1(T^{a,\eps},\psi^{a,\eps}) = O(\eps^{N+1})$, $ \mathcal{R}_1(T^{a,\eps},\psi^{a,\eps})=O(\eps^{N+1})$. 
\end{lemma}
The lemma will be proved later in section 5 after we show the property of the Milne problems with geometric corrections, which is needed to obtain the error estimates.

\section{Nonlinear Milne problem with geometric correction}\label{sec3}
One clearly notes that the leading order of the Milne problem with geometric corrections \eqref{eq:g0} is nonlinear and the higher orders problem \eqref{eq:g1}, \eqref{eq:g2} are linear, we will consider these two problems separately. 

In this section, we first study the nonlinear Milne problem \eqref{eq:g0}. We rewrite the system as follows
\begin{align}
	&\partial_\eta^2 G^\eps + F(\eps;\eta) \partial_\eta G^\eps + \langle \varphi^\eps - (G^\eps)^4 \rangle = 0,\label{eq:n1}\\
	&\sin\phi\partial_\eta \varphi^\eps + F(\eps;\eta) \cos\phi \partial_\phi \varphi^\eps + \varphi^\eps - (G^\eps)^4 = 0,\label{eq:n2}
\end{align}
with boundary conditions 
\begin{align}
	&G^\eps(0) = G_b,\label{eq:n1b}\\
	&\varphi^\eps(0,\phi) = \varphi_b(\phi),\quad \text{for }\sin\phi >0.\label{eq:n2b}
\end{align}
Note that here we drop the dependence of variables on $\theta$ since it is irrelevant here. This section is devoted to the proof of  the existence for the above nonlinear system \eqref{eq:n1}-\eqref{eq:n2b}.
\begin{theorem}\label{thm.nm}
	Given $0\le G_b \le \gamma$, $0\le \varphi_b(\phi) \le \gamma^4$ for $\sin\phi>0$. Then there exists a weak solution $(\varphi^\eps,G^\eps)\in L^2_{\rm loc}([0,\infty))\times L^2_{\rm loc}([0,\infty)\times [-\pi,\pi))$ to system \eqref{eq:n1}-\eqref{eq:n2} with boundary conditions \eqref{eq:n1b}-\eqref{eq:n2b}. Moreover, the solutions satisfy the following estimates
	\begin{align}\label{eq:est/n1}
		&\int_0^\infty 4(G^\eps)^3 |\partial_\eta {G^\eps}|^2 d\eta + \frac12 \int_0^\infty \int_{-\pi}^\pi(\varphi^\eps-(G^\eps)^4)^2 d\phi d\eta  + \frac12 \int_{\sin\phi<0} |\sin\phi| (\varphi^\eps(0,\phi)-G_b^4)^2  d\phi \nonumber \\
		&\quad \le \frac12 \int_{\sin\phi>0} \sin\phi (\varphi_b(\phi)-G_b^4)^2  d\phi,	
	\end{align}
	and for any $0<\alpha<1/4$,
	\begin{align} \label{eq:est/n2}
		&\int_0^\infty e^{2\alpha\eta} 4(G^\eps)^3 |\partial_\eta G^\eps|^2 d\eta + \frac14 \int_0^\infty \int_{-\pi}^\pi e^{2\alpha\eta} (\varphi^\eps-(G^\eps)^4)^2 d\phi d\eta + \frac12 \int_{\sin\phi<0} |\sin\phi| (\varphi^\eps_B(0,\phi)-G^4_b)^2  d\phi \nonumber\\
		&\qquad = \frac12 \int_{\sin\phi>0} \sin\phi (\varphi_b - G_b^4)^2 d\phi.	
	\end{align}
	Furthermore, the limit $G^\eps_\infty=\lim_{\eta\to\infty} G^\eps(\eta)$ exists and 
	\begin{align}\label{eq:nmdecay}
		|G^\eps(\eta) - G^\eps_\infty| \le C e^{-\alpha \eta},
	\end{align}
	where $C>0$ is a constant only depending on $G_b$ and $\varphi_b$.
\end{theorem}

The proof of the above theorem follows the similar procedure as the case for flat boundary without geometric corrections in \cite{Bounadrylayer2019GHM2}. We first establish the existence in a bounded domain and then use uniform weighted estimates to carry out the solutions to half-space. However, due to the additional terms by geometric corrections, we need to use properties of the function $F(\eps;\eta)$ defined by \eqref{eq:Fdef} to obtain the existence proofs.

The following property of $F(\eps;\eta)$ was proved in \cite[Lemma 4.1]{wu2015geometric}. 
\begin{lemma}[\cite{wu2015geometric}]
	There exists a potential function $V(\eps;\eta)$ satisfying $V(\eps;0)=0$ and $F(\eps;\eta) = -\partial_\eta V(\eps;\eta)$. Moreover, $V(\eps;\eta)$ is a monotonically increasing function of $\eta$ satisfying 
	\begin{align*}
		0\le V(\eps;\eta) \le -\log(1-\tfrac34\delta),
	\end{align*}
	and that $V_\infty(\eps):=\lim_{\eta\to\infty} V(\eps;\eta)$ exists and is independent of $\eps$. For any $\sigma>0$, $\eta>0$, 
	\begin{align*}
		e^{V(\eta+\sigma)} - e^{V(\eta)} \le 1 + \frac{4}{4-3\delta} \eps \sigma.
	\end{align*}
	For $F(\eps;\eta)$, the following inequalities hold:
	\begin{align}
		&-\log (1-\tfrac34\delta) \le \int_y^z F(\eps;\eta)d\eta \le 0,\quad \text{for any }0\le y\le z\le \infty, \label{eq:Fin1}\\
		&\int_0^\infty\int_\eta^\infty |F(y)^2| dy d\eta \le \frac{3\delta}{4-3\delta} + \log(1-\tfrac34\delta),\label{eq:Fin2}\\
		&\int_0^\infty|F(\eps;\eta)|^2 d\eta \le \frac{3\delta}{4-3\delta}\eps, \label{eq:Fin3}\\
		&\|F(\eps;\eta)\|_{L^\infty(\mathbb{R}_+)} \le \frac{4}{4-3\delta}\eps. \label{eq:Fin4}
	\end{align}
	
\end{lemma}

\subsection{Existence on the bounded interval}
We first consider the nonlinear Milne problem on the bounded domain with $\eta\in [0,B]$, $\phi \in [-\pi,\pi)$.
\begin{align}
	&\partial_\eta^2 G^\eps_B + F(\eps;\eta) \partial_\eta G^\eps_B + \langle \varphi^\eps_B - (G^\eps_B)^4 \rangle =0, \label{eq:B1} \\
	&\sin\phi \partial_\eta \varphi^\eps_B + F(\eps;\eta) \cos\phi \partial_\phi \varphi^\eps_B + \varphi^\eps_B - (G^\eps_B)^4 =0,\label{eq:B2}
\end{align}
subject to the boundary conditons 
\begin{align}
	&G_B^\eps(0) = G_b,\quad \partial_\eta G^\eps_B(B) = 0, \label{eq:B1b}\\
	&\varphi^\eps_B(0,\phi) = \varphi_b(\phi), \quad \varphi^\eps_B(B,\phi)  = \varphi^\eps_B(B,-\phi),\quad \text{for }\sin\phi>0.\label{eq:B2b}
\end{align}

The existence of solutions to the above problem is guaranteed by the following lemma.
\begin{lemma}\label{lem.B}
	Assume $0\le G_b\le \gamma$ and $0\le \varphi_b \le \gamma^4$ for some constant $\gamma>0$. Then there exists a unique solution $(G_B^\eps,\varphi_B^\eps) \in C^2([0,B])\times C^1([0,B]\times [-\pi,\pi))$ to system \eqref{eq:B1}-\eqref{eq:B2} with boundary conditions \eqref{eq:B1b}-\eqref{eq:B2b}, and the solution satisfies 
	\begin{align*}
		0\le G_B^\eps(\eta)\le \gamma,\quad 0\le \varphi^\eps_B(\eta,\phi)\le \gamma^4,\quad \text{for }\eta\in [0,B],\, \phi\in [-\pi,\pi)
	\end{align*}
\end{lemma}

In order to show the Lemma \ref{lem.B}, we follow our previous work on the Milne problem \cite{Bounadrylayer2019GHM2}. The proof for the flat case in \cite{Bounadrylayer2019GHM2} relys on the monotonicity of solutions for the elliptic equation and transport equation. Here we show that the geometric correction does not alter the monotonicity of the solutions. We start by proving the monotonicity of the elliptic equation \eqref{eq:B1}.
\begin{lemma}\label{lem.B1}
	Suppose $0\le G_b\le \gamma$ for some constant $\gamma>0$. Suppose $\zeta=\zeta(x)$ is an increasing monotonic function defined on $x\in\mathbb{R}$ and $g \in C([0,B])$ is a continuous bounded function satisfying $0\le g \le \zeta(\gamma)$. Then there exists a unique bounded solution $G^\eps_B \in C^2([0,B])$ to the equation 
	\begin{align}\label{eq:GB}
		-\partial_\eta^2 G + F(\eps;\eta) \partial_\eta G + \zeta(G) = g(\eta),\quad \text{on }\eta\in [0,B],
	\end{align}
	with boundary conditions 
	\begin{align}\label{eq:GBb}
		G(0) = G_b, \quad \partial_\eta G(B)=0,
	\end{align}
	and the solution satisfies $0\le G(\eta)\le \gamma$ for any $\eta\in [0,B]$.

	Moreover, let $G_1$, $G_2$ be two solutions of the above equation corresponding to source terms $g_1$, $g_2$ and boundary data $G_{b1}$, $G_{b2}$, respectively. Suppose $0\le g_1 \le g_2 \le \gamma$ on $[0,B]$ and $0\le G_{b1} \le G_{b2} \le \gamma$, then $0\le G_1 \le G_2 \le \gamma$ on $[0,B]$.
\end{lemma}
\begin{proof}
The proof of the lemma is based on the Leray-Schauder fixed point theorem and the maximum principle. 

\emph{Existence.} Let $h \in C([0,B])$ be the solution to the problem 
\begin{align*}
	-\partial_\eta^2 f - F(\eps;\eta) \partial_\eta f = h, \text{ in [0,B]}, \quad f(0)=0,\quad \partial_\eta f(B) = 0.
\end{align*}
Then the solution can be explicitly solved by 
\begin{align}\label{eq:f/formula}
	f(\eta) = \int_0^\eta e^{-\int_0^s F(\eps;p)dp} \left(\int_0^s - e^{\int_0^q F(\eps;p)dp} h(q) dq + \int_0^B e^{\int_0^q F(\eps;p) dp} h(q) dq \right) ds.
\end{align}
Due to \eqref{eq:Fin1}, $|\int_0^s F(\eps;p) dp| \le -\log(1-\tfrac34\delta)$ is bounded, so  $e^{\pm \int_0^s F(\eps;p) dp}$ is also bounded and the above formula is well-defined. Denote the above relation by $f=\mathcal{G} h$ with $\mathcal{G}$ being a linear operator mapping from $\mathcal{C}([0,B])$ to $C^2([0,B])$. We next show the nonlinear operator $\mathcal{N}$ defined by 
\begin{align*}
	\mathcal{N} G:=\mathcal{G}(g-\zeta(G)) + G_b,
\end{align*}
has a fixed point. First, the operator $\mathcal{N}: C([0,B])\times\mathbb{C}^2([0,B])$ is a compact mapping from the Banach space $(C[0,B])$ to $C([0,B])$, due to the compact embedding $C^2([0,B]) \subset C([0,B])$. Next we show the set 
\begin{align}\label{eq:setG}
	\{G \in C([0,B]) : G = \sigma \mathcal{N} G,\quad \text{for some }\sigma \in [0,1]\}
\end{align}
is bounded. To show this, suppose $G$ is in this set, then 
\begin{align*}
	&-\partial_\eta^2 G - F(\eps;\eta)\partial_\eta G + \sigma \zeta(G) = \sigma g,\quad \text{ in }[0,B]\\
	&G(0) = G_b, \quad \partial_\eta G(B)=0.
\end{align*}
First we show $G\ge 0$. Otherwise the minimum value of $G(\eta)$ on $[0,B]$ is small than $0$. Let $\eta_{\min}$ be the minimum point, then $G(\eta_{\min}) < 0$. If $\eta_{\min}=0$ then $G(\eta_{\min})=G_b \ge 0$. If $\eta_{\min} \in (0,B]$, then $\partial_\eta G(\eta_{\min})=0$ and 
\begin{align}
	-\partial_\eta^2 G(\eta_{\min}) =\sigma g(\eta_{\min}) - \sigma \zeta(G(\eta_{\min})) >0,
\end{align}
due to $\zeta$ being an increasing function and $G(\eta_{\min})<0$. Hence $G({\eta})$ is locally concave around $\eta_{\min}$, which contradicts to the assumption that $\eta_{\min}$ is the minimum point. Hence $G(\eta)\ge 0$ for any $\eta \in[0,B]$. 

We can also show $G \le \gamma$. Otherwise the maxium value of $G(\eta)$ is larger than $\gamma$. Let $\eta_{\max}$ be the maximum point, then $G(\eta_{\max}) > \gamma$. Since $G_b \le \gamma$, $\eta_{\max}\neq 0$. For $\eta_{\max} \in (0,B]$, we have $\partial_\eta G(\eta_{\max}) =0$ and 
\begin{align*}
	-\partial_\eta^2 G(\eta_{\max}) = \sigma g(\eta_{\max}) - \sigma \zeta(G(\eta_{\max})) \le  \sigma \zeta(\gamma) - \sigma \zeta(G(\eta_{\max}))  <0,
\end{align*}
due to $\zeta$ being an increasing function and $\zeta(G(\eta_{\max})) > \zeta(\gamma)$. This implies that $G(\eta)$ is locally strictly convex near $\eta_{\max}$, which contradicts the assumption that $\eta_{\max}$ is a maximum point. Hence $G\le \gamma$ holds. 

From $0\le G\le \gamma$ on $[0,B]$, we have $\|G\|_{C([0,B])} \le \gamma$ and the set \eqref{eq:setG} is bounded. Consequently, we can apply the Leray-Schauder fixed point theorem and conclude that $\mathcal{N}$ has a fixed point in $C([0,B])$. Let $G=G(\eta)\in C([0,B])$ be a fixed point, then $\mathcal{N} G = G$ and by the definition of $\mathcal{N}$, $G\in C^2([0,B])$ and solves equation \eqref{eq:GB} with boundary condition \eqref{eq:GBb}. By the same argument as above we also have $0\le G \le \gamma$ in $[0,B]$.

\emph{Uniqueness.} Note that the Leray-Schauder fixed point theorem only guarantees the existence for equation \eqref{eq:GB}. To show the uniqueness of solutions, we use the maximum principle. Suppose $G_1 \neq G_2$ be two solutions to \eqref{eq:GB}. Let $\ell:=G_1-G_2$, suppose $\ell(\eta)$ attains its maximum at $\eta_{\max}$ and its minimum at $\eta_{\min}$. First we show $\ell(\eta_{\max})\le 0$. Otherwise suppose $\ell(\eta_{\max})>0$, then $\eta_{\max}\neq 0$ since $G_{1}(0)=G_2(0)=G_b.$ For $\eta_{\max} \in (0,B]$, we have $\partial_\eta \ell(\eta_{\max})=0$, hence 
\begin{align*}
	-\partial_\eta^2 \ell (\eta_{\max}) = -  (\zeta(G_1(\eta_{\max})) - \zeta(G_2(\eta_{\max}))) <0,
\end{align*}
due to $G_1(\eta_{\max}) > G_2(\eta_{\max})$ and $\zeta$ being an increasing function. This implies $\ell$ is locally convex near $\eta_{\max}$, which contradicts the assumption that $\eta_{\max}$ is a maximum point. Therefore, $\ell(\eta)\le \ell(\eta_{\max})\le0$ for any $\eta \in [0,B]$. By a similar argument, one can show $\ell(\eta)\ge \ell(\eta_{\min}) \ge 0$ and we conclude that $\ell(\eta)=0$ for $\eta\in [0,B]$. Therefore, $G_1(\eta)=G_2(\eta)$ on $\eta \in [0,B]$ and the solution to \eqref{eq:GB} is unique.

\emph{Monotonicity.} Let $G_1$, $G_2$ be two solutions as described in Lemma \ref{lem.B1}. Let $\ell:=G_1-G_2$. We first show the maximum value of $\ell$ cannot be bigger than zero. Otherwise, suppose $\ell(\eta)$ attains its maximum at $\eta_{\max}$ with $\ell(\eta_{\max}) >0.$ Then due to $G_1(0)-G_2(0)=G_{b1} - G_{b2} \le 0$, $\eta_{\max} \neq 0$. For $\eta_{\max}\in(0,B]$, we have $\partial_\eta \ell(\eta_{\max})=0$ and 
\begin{align*}
	-\partial_\eta^2 \ell (\eta_{\max}) = g_1(\eta_{\max}) - g_2(\eta_{\max}) -  (\zeta(G_1(\eta_{\max})) - \zeta(G_2(\eta_{\max}))) <0,
\end{align*}
due to the assumption $g_1\le g_2$ and the monotonicity of $\zeta$. Hence $\ell(\eta)$ is strictly convex near $\eta_{\max}$, which contradicts the assumption that $\eta_{\max}$ is a maximum point. Therefore, $\ell(\eta) \le \ell(\eta_{\max}) \le 0$ for any $\eta \in [0,B]$. That is, $G_1(\eta) \le G_2(\eta)$ for any $\eta\in [0,B]$. Thus we finish the proof of the lemma.
\end{proof}
\begin{rem}
	Compared to the proof for the flat case in \cite{Bounadrylayer2019GHM2}, we can see here that adding $F(\eps;\eta)\partial_\eta G$ does not change the maximum principle since the stationary point $\partial_\eta G=0$ vanishes.
\end{rem}

Next, we show the monotonicity of the transport equation \eqref{eq:B2}.
\begin{lemma}\label{lem.B2}
	Given $\varphi_b=\varphi_b(\phi) \in L^\infty((0,\pi))$. 
	Let $h=h(\eta,\phi) \in C([0,B]\times [-\pi,\pi))$.
	 Then there exists a unique solution $\varphi \in C^1([0,B]\times [-\pi,\pi))$ to the equation 
	\begin{align}\label{eq:varphiB}
		&\sin\phi \partial_\eta \varphi + F(\eps;\eta) \cos\phi \partial_\phi \varphi + \varphi = h,\text{ in }[0,B]\times [-\pi,\pi),
	\end{align}
	with boundary conditions 
	\begin{align}\label{eq:varphiBb}
		\varphi(0,\phi) = \varphi_b(\phi),\quad \varphi(B,\phi) = \varphi(B,-\phi),\quad \text{for }\sin\phi>0,
	\end{align}
	and the solution satisfies 
	\begin{align}\label{eq:varphilinf}
		\|\varphi\|_{L^\infty([0,B])} \le \|h\|_{L^\infty([0,B]\times [-\pi,\pi))} + \|\varphi_b\|_{L^\infty((0,\pi))}.
	\end{align}

	Moreover, if $0\le \varphi_b\le \gamma$ and $0\le h \le \gamma$, then the solution $\varphi \ge 0$ on $[0,B]\times [-\pi,\pi)$. Furthermore,
	let $\varphi_1,\varphi_2$ be two solutions of the above equation corresponding to the source term $h_1$, $h_2$ and boundary data $\varphi_{b1}$, $\varphi_{b2}$, respectively. Suppose $0\le h_1 \le h_2 \le \gamma$ on $[0,B]$ and $0\le \varphi_{b1}\le \varphi_{b2} \le \gamma$ on $(0,\pi)$, then $0\le \varphi_1 \le \varphi_2 \le \gamma$ on $[0,B]\times[-\pi,\pi)$.
\end{lemma}
\begin{proof}
	To solve \eqref{eq:varphiB}, we use the method of characteristics.  Suppose $\phi=\phi(\eta)$ is the characteristic line given by 
	\begin{align*}
		\frac{d}{d\eta} \phi = F(\eps;\eta)\cos\phi  = -\cos\phi \partial_\eta V(\eps;\eta),\quad \phi(\eta=0) = \phi_0 
	\end{align*}
	whose solution is given by 
	\begin{align*}
		\cos\phi e^{-V(\eps;\eta)} = \cos\phi_0.
	\end{align*}
	Define the energy $E(\eta,\phi) = \cos\phi e^{-V(\eta)}$. 
	The characteristics lines for \eqref{eq:varphiB} is the curves with constant energy. Let the curves are parameterized by $|E|=e^{-V(\eta_+)}$. We define $\phi' = \phi'(\phi,\eta,\eta')$ to denote the solution of $\cos\phi e^{-V(\eta)} = \cos \phi' e^{-V(\eta')} = e^{-V(\eta_+)}$, which is for $0\le \eta'\le \eta_+$, 
	\begin{align*}
		\phi'(\phi,\eta,\eta') = \arccos \left(\cos\phi e^{V(\eta')-V(\eta)}\right),
	\end{align*}
	where $\arccos$ is defined in $[0,\pi)$ as a single-valued function. In order to extend the above relation to $[-\pi,\pi]$, we use the symmetric property of $\cos\phi$ and define for $\phi<0$,
	\begin{align*}
		\phi'(-\phi,\eta,\eta')= -\phi'(\phi,\eta,\eta').
	\end{align*}
	Thus $\phi'=\phi'(\phi,\eta,\eta')$ defines a function on $[-\pi,\pi]\times [0,B]\times [0,B]$.

	Equation \eqref{eq:varphiB} becomes 
	\begin{align}\label{eq:varphi/ch}
		\sin\phi(\eta) \frac{d}{d\eta} \varphi(\eta,\phi(\eta)) + \varphi = h,
	\end{align}
	with boundary condition 
	\begin{align*}
		\varphi(0,\phi_0) = \varphi_b(\phi_0),\quad \text{for }\sin\phi_0 >0,\quad \varphi(B,\phi(B)) = \varphi(B,-\phi(B)),\quad \text{ for }\sin\phi_B >0.
	\end{align*}
	Based on the values of $\sin\phi$ and $\eta$, we divide into three different cases.

	\emph{Case 1. $\sin \phi >0$.} We integrate \eqref{eq:varphi/ch} over $[0,\eta]$ and get 
	\begin{align*}
		\varphi(\eta,\phi(\eta)) = \varphi_b(\phi_0) e^{-\int_0^\eta \frac{1}{\sin \phi(\xi)} d\xi} + \int_0^\eta \frac{h(\xi,\phi'(\phi,\eta,\xi))}{\sin \phi(\xi)} e^{-\int_\xi^\eta \frac{1}{\sin\phi(\rho)} d\rho} d\xi,
	\end{align*}
	which using $\phi_0=\phi'(\phi,\eta,0)$, $\phi(\xi)=\phi'(\phi,\eta,\xi)$, can be written as 
	\begin{align}\label{eq:varphi/1}
		\varphi(\eta,\phi) = \varphi_b(\phi'(\phi,\eta,0)) e^{-\int_0^\eta \frac{1}{\sin \phi'(\phi,\eta,\xi)} d\xi} + \int_0^\eta \frac{h(\xi,\phi'(\phi,\eta,\xi))}{\sin \phi'(\phi,\eta,\xi)} e^{-\int_\xi^\eta \frac{1}{\sin \phi'(\phi,\eta,\rho)}d\rho} d\xi.
	\end{align}

	\emph{Case 2. $\sin \phi<0$ and $|E(\phi,\eta)|\le e^{-V(\eps;B)}$.} In this situation, we can integrate \eqref{eq:varphi/ch} over $[\eta,B]$ to get 
	\begin{align*}
		\varphi(B,\phi(B)) e^{\int_0^L \frac{1}{\sin \phi(\xi)} d\xi } - \varphi(\eta,\phi(\eta)) e^{\int_0^\eta \frac{1}{\sin\phi(\xi)}d\xi} = \int_\eta^B \frac{h(\xi,\phi'(\phi,\eta,\xi))}{\sin\phi(\xi)} e^{\int_0^\xi \frac{1}{\sin\phi(\rho)}d\rho } d\xi.
	\end{align*}
	Since $\sin \phi <0$, we have $\phi(\xi)= \phi'(\phi,\eta,\xi) = - \phi'(-\phi,\eta,\xi)$ and $\phi(B)=\phi'(\phi,B,\xi)=-\phi'(-\phi,\eta,B)$. The above relationship can be rewritten as 
	\begin{align*}
		\varphi(\eta,\phi) = \varphi(B,-\phi'(-\phi,\eta,L)) e^{\int_\eta^B \frac{1}{-\sin\phi'(-\phi,\eta,\xi)}d\xi} -  \int_\eta^B \frac{h(\xi,-\phi'(-\phi,\eta,\xi))}{-\sin \phi'(-\phi,\eta,\xi)} e^{-\int_\xi^\eta \frac{1}{-\sin \phi'(-\phi,\eta,\rho)}d\rho} d\xi.
	\end{align*}
	Since $\sin \phi(B)<0$, we have $\varphi(B,\phi(B)) = \varphi(B,-\phi(B)) = \varphi(B,\phi'(-\phi,\eta,B))$ due to the boundary condition \eqref{eq:varphiBb}. By \eqref{eq:varphi/1}, we have 
	\begin{align*}
		\varphi(B,\phi'(-\phi,\eta,B)) = \varphi_b(\phi'(-\phi,\eta,0))e^{-\int_0^B \frac{1}{\sin\phi'(-\phi,\eta,\xi)}d\xi} + \int_0^B \frac{h(\xi,\phi'(-\phi,\eta,\xi))}{\sin\phi'(-\phi,\eta,\xi)} e^{-\int_\xi^B\frac{1}{\sin\phi'(-\phi,\eta,\rho)}d\rho} d\xi.
	\end{align*}
	Taking this into the previous equation leads to 
	\begin{align}\label{eq:varphi/2}
		\varphi(\eta,\phi) &= \varphi_b(\phi'(-\phi,\eta,0)) e^{-\int_0^B + \int_\eta^B \frac{1}{\sin\phi'(-\phi,\eta,\xi)}d\xi} + \int_0^B \frac{h(\xi,\phi'(-\phi,\eta,\xi))}{\sin\phi'(-\phi,\eta,\xi)} e^{-\int_\xi^B + \int_{\eta}^B \frac{1}{\sin\phi'(-\phi,\eta,\rho)}d\rho} d\xi \nonumber\\
		&\quad + \int_\eta^B \frac{h(\xi,-\phi'(-\phi,\eta,\xi))}{\sin \phi'(-\phi,\eta,\xi)} e^{\int_\xi^\eta \frac{1}{\sin \phi'(-\phi,\eta,\rho)}d\rho} d\xi.
	\end{align}
	Here we use the notation $\int_0^B + \int_\eta^B \cdot d\xi = \int_0^B \cdot d\xi + \int_\eta^B \cdot d\xi$.

	\emph{Case 3. $\sin\phi<0$ and $|E(\phi,\eta)|>e^{-V(\eps;B)}$.} In this case, the characteristic curve does not touch the boundary $\eta=B$ but reaches $(\phi=0,\eta=\eta_+)$. We integrate \eqref{eq:varphi/ch} over $[\eta,\eta_+]$ and get 
	\begin{align*}
		\varphi(\eta,\phi) = \varphi(\eta_+,0) e^{\int_\eta^{\eta_+} \frac{1}{-\sin\phi'(-\phi,\eta,\xi)}d\xi} -  \int_\eta^{\eta_+} \frac{h(\xi,-\phi'(-\phi,\eta,\xi))}{-\sin \phi'(-\phi,\eta,\xi)} e^{-\int_\xi^{\eta} \frac{1}{-\sin \phi'(-\phi,\eta,\rho)}d\rho} d\xi.
	\end{align*}
	Since $(\eta_+,0)$ also lies on the right characteristic lines with $\phi'(-\phi,\eta,\eta_+)=0$, we get from \eqref{eq:varphi/1},
	\begin{align*}
		\varphi(\eta,\phi'(-\phi,\eta,\eta_+)) &= \varphi_b(\phi'(-\phi,\eta,0)) e^{-\int_0^{\eta_+} \frac{1}{-\sin \phi'(-\phi,\eta,\xi)} d\xi} \nonumber\\
		&\quad + \int_0^{\eta_+} \frac{h(\xi,\phi'(-\phi,\eta,\xi))}{-\sin \phi'(-\phi,\eta,\xi)} e^{-\int_\xi^\eta \frac{1}{-\sin \phi'(-\phi,\eta,\rho)}d\rho} d\xi.
	\end{align*}
	Taking it into the previous equation leads to 
	\begin{align}\label{eq:varphi/3}
		\varphi(\eta,\phi) &= \varphi_b(\phi'(-\phi,\eta,0)) e^{-\int_0^{\eta_+} + \int_\eta^{\eta_+} \frac{1}{\sin\phi'(-\phi,\eta,\xi)}d\xi} + \int_0^{\eta_+} \frac{h(\xi,\phi'(-\phi,\eta,\xi))}{\sin\phi'(-\phi,\eta,\xi)} e^{-\int_\xi^{\eta_+} + \int_{\eta}^{\eta_+} \frac{1}{\sin\phi'(-\phi,\eta,\rho)}d\rho} d\xi \nonumber\\
		&\quad + \int_\eta^{\eta_+} \frac{h(\xi,-\phi'(-\phi,\eta,\xi))}{\sin \phi'(-\phi,\eta,\xi)} e^{\int_\xi^{\eta} \frac{1}{\sin \phi'(-\phi,\eta,\rho)}d\rho} d\xi.
	\end{align}

	In any case, we can show that \eqref{eq:varphi/1}, \eqref{eq:varphi/2} and \eqref{eq:varphi/3} is a solution to \eqref{eq:varphiB} with boundary condition \eqref{eq:varphiBb}. Moreover, inequality \eqref{eq:varphilinf} holds. To show this, we have in case 1, from \eqref{eq:varphi/1},
	\begin{align*}
		\|\varphi(\eta,\phi)\|_{L^\infty([0,B]\times [-\pi,\pi))} &\le \|\varphi_b\|_{L^\infty((0,\pi))} + \|h\|_{L^\infty([0,B]\times [-\pi,\pi)} \int_0^\eta \frac{1}{\sin\phi(\xi)} e^{-\int_\xi^\eta \frac{1}{\sin\phi(\rho)}d\rho} d\xi \nonumber\\
		&=\|\varphi_b\|_{L^\infty((0,\pi))} + \|h\|_{L^\infty([0,B]\times [-\pi,\pi)} \left(1- e^{-\int_0^\eta \frac{1}{\sin\phi(\rho)}d\rho} \right) \nonumber\\
		&\le \|\varphi_b\|_{L^\infty((0,\pi))} + \|h\|_{L^\infty([0,B]\times [-\pi,\pi)},
	\end{align*}
	and in case 2 from \eqref{eq:varphi/2},
	\begin{align*}
		&\|\varphi(\eta,\phi)\|_{L^\infty([0,B]\times [-\pi,\pi))}\nonumber\\
		 &\quad \le \|\varphi_b\|_{L^\infty((0,\pi))} + \|h\|_{L^\infty([0,B]\times [-\pi,\pi)} (1-e^{-\int_0^B \frac{1}{\sin \phi'(-\phi,\eta,\rho)}d\rho}) e^{-\int_\eta^B \frac{1}{\sin \phi'(-\phi,\eta,\rho)}d\rho} \nonumber\\
		&\qquad + \|h\|_{L^\infty([0,B]\times [-\pi,\pi)} (1-e^{-\int_\eta^B \frac{1}{\sin \phi'(-\phi,\eta,\rho)}d\rho}) \nonumber\\
		&\quad\le \|\varphi_b\|_{L^\infty((0,\pi))} + \|h\|_{L^\infty([0,B]\times [-\pi,\pi)} ,
	\end{align*}
	and in case 3 from \eqref{eq:varphi/3},
	\begin{align*}
		&\|\varphi(\eta,\phi)\|_{L^\infty([0,B]\times [-\pi,\pi))} \nonumber\\
		&\quad \le \|\varphi_b\|_{L^\infty((0,\pi))} + \|h\|_{L^\infty([0,B]\times [-\pi,\pi)} (1-e^{-\int_0^{\eta_+} \frac{1}{\sin \phi'(-\phi,\eta,\rho)}d\rho}) e^{-\int_\eta^{\eta_+} \frac{1}{\sin \phi'(-\phi,\eta,\rho)}d\rho} \nonumber\\
		&\qquad + \|h\|_{L^\infty([0,B]\times [-\pi,\pi)} (1-e^{-\int_\eta^{\eta_+} \frac{1}{\sin \phi'(-\phi,\eta,\rho)}d\rho}) \nonumber\\
		&\quad\le \|\varphi_b\|_{L^\infty((0,\pi))} + \|h\|_{L^\infty([0,B]\times [-\pi,\pi)} ,
	\end{align*}

	Next assume $0\le \varphi_b \le \gamma$, $0\le h\le \gamma$, we show the boundness of solutions, we can see obviously from \eqref{eq:varphi/1}, \eqref{eq:varphi/2} and \eqref{eq:varphi/3} that as long as $\varphi_b \ge 0$ and $h\ge 0$, $\varphi(\eta,\phi)\ge 0$ for all $\eta\in[0,B]$ and $\phi\in [-\pi,\pi)$. Next we show if $\varphi_b \le \gamma$, $h\le \gamma$, then $\varphi\le \gamma$. We show this for case 1 and leave the other cases to the reader. For case 1, from \eqref{eq:varphi/1}, 
	\begin{align*}
		\varphi(\eta,\phi) &\le  \gamma e^{-\int_0^\eta \frac{1}{\sin \phi'(\phi,\eta,\xi)} d\xi} + \gamma \int_0^\eta \frac{1}{\sin \phi'(\phi,\eta,\xi)} e^{-\int_\xi^\eta \frac{1}{\sin \phi'(\phi,\eta,\rho)}d\rho} d\xi \nonumber\\
		&= \gamma e^{-\int_0^\eta \frac{1}{\sin \phi'(\phi,\eta,\xi)} d\xi} + \gamma \int_0^\eta d e^{-\int_\xi^\eta \frac{1}{\sin \phi'(\phi,\eta,\rho)}d\rho} \nonumber\\
		&=\gamma e^{-\int_0^\eta \frac{1}{\sin \phi'(\phi,\eta,\xi)} d\xi} + \gamma e^{-\int_\xi^\eta \frac{1}{\sin \phi'(\phi,\eta,\rho)}d\rho} \big|_{\xi=0}^\eta = \gamma.
	\end{align*}
	Therefore, $0\le \varphi(\eta,\phi) \le \gamma$ for any $\eta\in [0,B]$, $\phi\in [-\pi,\pi)$.

	Finally, the monotonicity of solutions for equation \eqref{eq:varphi/ch} can be seen from the formulas \eqref{eq:varphi/1}, \eqref{eq:varphi/2} and \eqref{eq:varphi/3} directly. For example, if $\varphi_{b1}\le \varphi_{b2}$ and $h_1\le h_2$, then for the case 1,
	\begin{align*}
		(\varphi_1 - \varphi_2)(\eta,\phi) &= (\varphi_{b1}-\varphi_{b2})(\phi'(\phi,\eta,0)) e^{-\int_0^\eta \frac{1}{\sin \phi'(\phi,\eta,\xi)} d\xi}\nonumber\\
		&\quad + \int_0^\eta \frac{(h_1-h_2)(\xi,\phi'(\phi,\eta,\xi))}{\sin \phi'(\phi,\eta,\xi)} e^{-\int_\xi^\eta \frac{1}{\sin \phi'(\phi,\eta,\rho)}d\rho} d\xi\le 0,
	\end{align*}
	hence $\varphi_1\le \varphi_2$ on $[0,B]\times [-\pi,\pi)$ and thus finish the proof.
\end{proof}

\begin{rem}
	Compared to the proof of the flat case in \cite{Bounadrylayer2019GHM2}, the additional variable $\phi$ could be integrated in the characteristics and the monotonicity of solution remains true. Moreover,  the estimate \eqref{eq:varphilinf} has already been proved in \cite[equations (4.52)-(4.54)]{wu2015geometric} where the monotonicity property is new which is crucial in our approach.
\end{rem}
We now introduce the concept of \emph{sub-solutions}. Our definition is consistent from the theory in the \cite{evans1997partial,clouet2009milne} or also in \cite{Bounadrylayer2019GHM2}.
\begin{definition}
	A pair of functions $(G,\varphi) \in C^2([0,B])\times C^1([0,B]\times[-\pi,\pi))$ is called a \emph{sub-solution} to system \eqref{eq:B1}-\eqref{eq:B2} if it satisfies the inequality
	\begin{align*}
		&-\partial_\eta^2 G - F(\eps;\eta)\partial_\eta G - \langle \varphi - G^4 \rangle \le 0,\\
		&\sin\phi \partial_\eta \varphi + F(\eps;\eta) \cos\phi \partial_\phi \varphi + \varphi - G^4 \le 0,
	\end{align*}
	and on the left boundary $G(0)\le G_b$, $\varphi(0,\phi)\le \varphi_b(\phi)$ for $\sin\phi >0$ and $\partial_\eta G(B)=0$, $\varphi(B,\phi)=\varphi(B,-\phi)$ on the right boundary.
\end{definition}

Next we proceed for the proof of Lemma \ref{lem.B}.
\begin{proof}[Proof of Lemma \ref{lem.B}]
	The proof of the lemma follows the same way as the proof of \cite[Theorem 4]{Bounadrylayer2019GHM2}. Here we give a brief description of the proof. For the details, we refer the reader to \cite{Bounadrylayer2019GHM2}.

	First we introduce the function $\zeta$ by 
	\begin{align*}
		\zeta(G) = \left\{\begin{array}[]{ll}
			\frac{G}{1-G}, &\text{for }G<0,\\
			G^4,&\text{for }0\le G \le \gamma,\\
			G^4 + \frac{G-\gamma}{1+G-\gamma},&\text{for }G>\gamma,
		\end{array}\right.
	\end{align*}
	which extends the power function $G^4$ to $\mathbb{R}$ and $\zeta$ is an increasing function. Starting with any sub-solution $G^0,\varphi^0$ (e.g. $G^0=0,\varphi^0=0$), we construct a sequence of functions $\{G^k,\varphi^k\}$ by solving iteratively
	\begin{align}
		&-\partial_\eta^2 G^k - F(\eps;\eta)  \partial_\eta G^k + \langle \zeta(G^4) \rangle =  \langle \varphi^{k-1}\rangle, \label{eq:lm2/k1} \\
		&\sin\phi\partial_\eta \varphi^k + F(\eps;\eta)\cos\phi\partial_\phi \varphi^k + \varphi^k = (G^{k-1})^4.\label{eq:lm2/k2}
	\end{align}
	We can use Lemma \ref{lem.B1} and Lemma \ref{lem.B2} and show by induction that $\gamma\ge G^k(\eta)\ge G^{k-1}(\eta)\ge 0$ and $\gamma\ge \varphi^k (\eta,\phi)\ge \varphi^{k-1}(\eta,\phi)\ge 0$ for any $\eta\in[0,B]$ and $\phi\in[-\pi,\pi)$. Therefore, $\{G^k,\varphi^k\}$ are bounded increasing sequences. Hence there existence bounded functions $(G,\varphi)$ such that $(G^k,\varphi^k) \to (G,\varphi)$ pointwise as $k\to \infty$. Moreover, since the solutions to \eqref{eq:lm2/k1}-\eqref{eq:lm2/k2} could be expressed using integral formulas using Lemma \ref{lem.B1} and Lemma \ref{lem.B2}, we can pass to the limit in these formulas and use Beppo Levi's lemma to conclude that $(G,\varphi) \in C^2([0,B])\times C^2([0,B]\times[-\pi,\pi])$. Therefore, by Dini's theorem,
	$\{G^k,\varphi^k\}$ converge uniformly to $(G,\varphi)$ and $(G,\varphi)$ solves \eqref{eq:B1}-\eqref{eq:B2} with boundary conditions \eqref{eq:B1b}-\eqref{eq:B2b}. 

	The proof of uniqueness is slightly different than that of \cite{Bounadrylayer2019GHM2}. Suppose $(G_1,\varphi_1)$ and $(G_2,\varphi_2)$ be two solutions to system \eqref{eq:B1}-\eqref{eq:B2}. Let 
	\begin{align*}
		G^0 = \max\{G_1,G_2\},\quad \varphi^0=\max\{\varphi_1,\varphi_2\},
	\end{align*}
	and we can construct sequence of functions $\{G^k,\varphi^k\}$ by the above precedure. Following the above arguments, $(G^k,\psi^k)$ converges to $(G,\varphi)$ which is a solution to \eqref{eq:B1}-\eqref{eq:B2} and $G\ge G^0,\varphi\ge \varphi^0$. Thus $G-G_1$, $\varphi-\varphi_1$ is a non-negative solution to the system 
	\begin{align}
		&-\partial_\eta^2 (G-G_1) - F(\eps;\eta) \partial_\eta (G-G_1) + \langle G^4 - (G_1)^4\rangle = \langle \varphi-\varphi_1\rangle,\label{eq:GG1/1}\\
		&\sin\phi  \partial_\eta (\varphi-\varphi_1) + F(\eps;\eta)\cos\phi\partial_\phi (\varphi-\varphi_1) + \varphi-\varphi_1 = G^4 - (G_1)^4,\label{eq:GG1/2}
	\end{align} 
	with 
	\begin{align*}
		&G(0)-G_1(0)=0,\quad \partial_x (G-G_1)(B)=0,\\
		&\varphi(0,\phi) -\varphi_1(0,\phi)=0,\quad \varphi(B,\phi)-\varphi_1(B,\phi) = \varphi(B,-\phi)-\varphi_1(B,-\phi) ,\quad \text{for }\sin\phi>0.
	\end{align*}
	Comparing \eqref{eq:GG1/1} with \eqref{eq:GG1/2}, 
	\begin{align*}
		\partial_\eta^2 (G-G_1) + F(\eps;\eta) \partial_\eta (G-G_1) & = \partial_\eta \langle \sin\phi (\varphi-\varphi_1) \rangle + F(\eps;\eta) \langle \cos\phi \partial_\phi(\varphi-\varphi_1)\rangle \nonumber\\
		&=\partial_\eta \langle \sin\phi (\varphi-\varphi_1) \rangle - F(\eps;\eta) \langle \partial_\phi\cos\phi (\varphi-\varphi_1)\rangle \nonumber\\
		&=\partial_\eta \langle \sin\phi (\varphi-\varphi_1) \rangle + F(\eps;\eta) \langle \sin\phi (\varphi-\varphi_1)\rangle.
	\end{align*}
	That is 
	\begin{align*}
		\partial_\eta \left(\partial_\eta (G-G_1) - \langle \sin\phi (\varphi-\varphi_1)\right) = -F(\eps;\eta) \left(\partial_\eta (G-G_1) - \langle \sin\phi (\varphi-\varphi_1)\right).
	\end{align*}
	Due to $F(\eps;\eta)\ge 0$ on $[0,B]$ and at $\eta=B$,
	\begin{align*}
		\partial_\eta (G-G_1)(B) - \langle \sin\phi (\varphi-\varphi_1)(B,\phi)\rangle = 0,
	\end{align*}
	the previous equation implies 
	\begin{align}\label{eq:GG1=}
		\partial_\eta (G-G_1)(\eta) - \langle \sin\phi (\varphi-\varphi_1)(\eta,\phi)\rangle =0,\quad \text{for any }\eta \in [0,B].
	\end{align}
	When $|E(\phi,\eta)|\ge e^{-V(\eps;B)}$, for $\sin\phi<0$, the solution $\varphi-\varphi_1$ can be written using \eqref{eq:varphi/2} as 
	\begin{align*}
		\varphi(\eta,\phi) &= \int_0^B \frac{G(\xi)-G_1(\xi)}{\sin\phi'(-\phi,\eta,\xi)} e^{-\int_\xi^B + \int_{\eta}^B \frac{1}{\sin\phi'(-\phi,\eta,\rho)}d\rho} d\xi  + \int_\eta^B \frac{G(\xi)-G_1(\xi)}{\sin \phi'(-\phi,\eta,\xi)} e^{\int_\xi^B \frac{1}{\sin \phi'(-\phi,\eta,\rho)}d\rho} d\xi.
	\end{align*}
	If $G(\xi)-G_1(\xi) >0$ on some interval $\xi\in (a,b)\in [0,B]$, then $\varphi(\eta,\phi)>0$. Therefore,
	\begin{align*}
		\langle \sin\phi (\varphi-\varphi_1)(0,\phi)\rangle &=\int_{\sin\phi>0} \sin\phi (\varphi-\varphi_1)(\eta,\phi) d\phi + \int_{\sin\phi<0} \sin\phi (\varphi-\varphi_1)(\eta,\phi) d\phi <0,
	\end{align*}
	which together with \eqref{eq:GG1=} implies 
	\begin{align*}
		\partial_\eta (G-G_1)(0) <0,
	\end{align*}
	which contradicts the fact that $G(0)-G_1(0)=0$ and $G-G_1$ is non-negative on $[0,B]$. Similarly, when $|E(\phi,\eta)|<e^{-V(\eps;B)}$, we can use \eqref{eq:varphi/3} and get the same contradiction argument. Thus $G(\eta)=G_1(\eta)$ for any $\eta \in [0,B]$. The uniqueness for \eqref{eq:GG1/2} then implies $\varphi-\varphi_1=0$. Therefore, $G=G_1$ and $\varphi=\varphi_1$. By the same argument, we can show $G=G_2$ and $\varphi=\varphi_2$. Hence $G_1=G_2$ and $\varphi_1=\varphi_2$ and the solution to \eqref{eq:B1}-\eqref{eq:B2} is unique.
\end{proof}

\subsection{Weighted estimates on the solutions on bounded interval}
Next we show a weighted estimate on the solution to system \eqref{eq:B1}-\eqref{eq:B2}. 
\begin{lemma}
Assuming the assumption of Lemma \ref{lem.B} holds and $(G_B^\eps,\varphi_B^\eps)$ be the unique solution to system \eqref{eq:B1}-\eqref{eq:B2} with boundary conditions \eqref{eq:B1b}-\eqref{eq:B2b}. Then the following inequalities hold for $0\le \eps\le 1/4$:
	\begin{align}\label{eq:est/B1}
		&\int_0^B 4(G_B^\eps)^3 |\partial_\eta {G_B^\eps}|^2 d\eta + \frac12 \int_0^B \int_{-\pi}^\pi(\varphi_B^\eps-(G_B^\eps)^4)^2 d\phi d\eta  + \frac12 \int_{\sin\phi<0} |\sin\phi| (\varphi_B^\eps(0,\phi)-G_b^4)^2  d\phi \nonumber \\
		&\quad \le \frac12 \int_{\sin\phi>0} \sin\phi (\varphi_b(\phi)-G_b^4)^2  d\phi,
	\end{align}
	and for $0\le \eps \le 1/4$, $0\le \alpha\le 1/4$,
	\begin{align}\label{eq:est/B2}
		&\int_0^B e^{2\alpha\eta} 4(G_B^\eps)^3 |\partial_\eta G_B^\eps|^2 d\eta + \frac14 \int_0^B \int_{-\pi}^\pi e^{2\alpha\eta} (\varphi^\eps_B-(G_B^\eps)^4)^2 d\phi d\eta + \frac12 \int_{\sin\phi<0} |\sin\phi| (\varphi^\eps_B(0,\phi)-G^4_b)^2  d\phi \nonumber\\
		&\qquad = \frac12 \int_{\sin\phi>0} \sin\phi (\varphi_b - G_b^4)^2 d\phi.
	\end{align}
	Moreover, setting 
	\begin{align*}
		\omega_B^\eps = \varphi_B^\eps - \frac{1}{2\pi} \int_{-\pi}^\pi \varphi_B^\eps d\phi,\quad q_B^\eps = \varphi_B^\eps - \omega_B^\eps,
	\end{align*}
	then the following inequalities hold:
	\begin{align} 
		&\|\partial_\eta G^\eps_B\|_{L^2([0,B])},\quad \|\omega_B^\eps\|_{L^2([0,B]\times[-\pi,\pi))} \le  \left(\int_{\sin\phi>0} \sin\phi (\varphi_b - G_b^4)^2 d\phi\right)^{1/2},\quad    \label{eq:est/B3}\\
		&|q_B^\eps(\eta)| \le C(\delta,\varphi_b,G_b)(1+\sqrt{\eta}+ \|\omega_B^\eps(\eta,\cdot)\|_{L^2([-\pi,\pi))}), ~ |G(\eta)|\le C(G_b,\varphi_b)(1+\sqrt{\eta}), \text{for any } \eta \in [0,B].\label{eq:est/B4}
	\end{align}
	\end{lemma}
\begin{proof}
	For simplicity of notations, we drop the superscript $\eps$ and the subscript $B$. 

	\emph{\textbf{The energy estimate.}}
	We first multiply \eqref{eq:B1} by $G^4$ and \eqref{eq:B2} by $\varphi$, and integrate to get 
	\begin{align} \label{eq:lm6/1}
		 &- \int_0^B G^4 \partial_\eta^2 G d\eta + \int_0^B \int_{-\pi}^\pi \sin\phi \partial_\eta \frac{\varphi^2}{2} d\phi d\eta + \int_0^B \int_{-\pi}^\pi (\varphi-G^4)^2 d\phi d\eta \nonumber\\
		 &\quad - \int_0^B F(\eps;\eta) G^4 \partial_\eta G d\eta + \int_0^B \int_{-\pi}^\pi F(\eps;\eta)\cos\phi \varphi \partial_\phi \varphi d\phi d\eta = 0.
 	\end{align}
	Using integration-by-parts and the boundary condition \eqref{eq:B1b}, we obtain 
	\begin{align} \label{eq:lm6/2}
		- \int_0^B G^4 \partial_\eta^2 G d\eta = -G^4 \partial_\eta G\bigg|_0^B + \int_0^B 4G^3 |\partial_\eta G|^2 d\eta = G^4(0)\partial_\eta G(0) +  \int_0^B 4G^3 |\partial_\eta G|^2 d\eta.
	\end{align}
	Using integration-by-parts and the boundary condition \eqref{eq:B2b}, we obtain 
	\begin{align} \label{eq:lm6/3}
		\int_0^B \int_{-\pi}^\pi \sin\phi \partial_\eta \frac{\varphi^2}{2} d\phi d\eta = \int_{-\pi}^\pi \sin\phi \frac{\varphi^2}{2} d\phi \bigg|_0^B = - \frac12 \langle \sin\phi \varphi^2(0,\phi) \rangle .
	\end{align}
	Comparing \eqref{eq:B1} and $\langle \eqref{eq:B2} \rangle$, it holds that 
	\begin{align*}
		\partial_\eta(\partial_\eta G - \langle \sin\phi \varphi\rangle) = -F(\eps;\eta) (\partial_\eta G-\langle \sin \phi \varphi\rangle),
	\end{align*}
	which with $\partial_\eta G(B) - \langle \sin\phi \varphi(B,\phi)\rangle=0$, implies that 
	\begin{align}\label{eq:G=}
		\partial_\eta G(\eta) - \langle \sin\phi \varphi(\eta,\phi)\rangle = 0,\quad \text{ for any }\eta\in[0,B].
	\end{align}
	Therefore, 
	\begin{align*}
		\partial_\eta G(0) = \langle \sin\phi \varphi(0,\phi) \rangle,
	\end{align*}
	and 
	\begin{align} \label{eq:lm6/4}
		G^4(0)\partial_\eta G(0) - \frac12 \langle \sin\phi \varphi^2(0,\phi)\rangle &= G^4(0)\sin\phi \varphi(0,\phi) \rangle - \frac12 \langle \sin\phi \varphi^2(0,\phi)\rangle \nonumber\\
		&= -\frac12 \langle \sin\phi (\varphi(0,\phi)-G^4(0))^2 \rangle,
	\end{align}
	due to $\langle \sin\phi G^4(0)\rangle =0.$ 

	Using \eqref{eq:G=}, we also have 
	\begin{align*}
		&-\int_0^B F(\eps;\eta) G^4 \partial_\eta G d\eta + \int_0^B\int_{-\pi}^\pi F(\eps;\eta) \cos\phi \varphi\partial_\phi \varphi d\phi d\eta \nonumber\\
		&\quad = -\int_0^B F(\eps;\eta) G^4 \partial_\eta G d\eta + \int_0^B \int_{-\pi}^\pi F(\eps;\eta) \sin\phi \frac{\varphi^2}{2} d\phi d\eta \nonumber\\
		&\quad = -\int_0^B \int_{-\pi}^\pi F(\eps;\eta) G^4 \sin\phi \varphi d\phi d\eta + \int_0^B \int_{-\pi}^\pi F(\eps;\eta) \sin\phi \frac{\varphi^2}{2} d\phi d\eta \nonumber\\
		&\quad = \frac12 \int_0^B\int_{-\pi}^\pi F(\eps;\eta) \sin\phi (\varphi-G^4)^2 d\phi d\eta.
	\end{align*}
	Combing this above relation with \eqref{eq:lm6/1}, \eqref{eq:lm6/2},  \eqref{eq:lm6/3}, \eqref{eq:lm6/4} gives 
	\begin{align*}
		&\int_0^B 4G^3 |\partial_\eta G|^2 d\eta + \int_0^B \int_{-\pi}^\pi(\varphi-G^4)^2 d\phi d\eta +  \frac12 \int_0^B\int_{-\pi}^\pi F(\eps;\eta) \sin\phi (\varphi-G^4)^2 d\phi d\eta \nonumber\\
		&\quad +  \frac12 \langle \sin\phi (\varphi(0,\phi)-G^4(0))^2 \rangle = 0.
	\end{align*}
	The last term on the left side of the above equation can be rewritten using the boundary conditon \eqref{eq:B2b} as 
	\begin{align*}
		\frac12 \langle \sin\phi (\varphi(0,\phi)-G^4(0))^2 \rangle &= \frac12 \int_{\sin\phi>0} \sin\phi (\varphi(0,\phi)-G^4(0))^2  d\phi + \frac12 \int_{\sin\phi<0} \sin\phi (\varphi(0,\phi)-G^4(0))^2  d\phi \nonumber\\
		&=\frac12 \int_{\sin\phi>0} \sin\phi (\varphi_b(\phi)-G^4(0))^2  d\phi - \frac12 \int_{\sin\phi<0} |\sin\phi| (\varphi(0,\phi)-G^4(0))^2  d\phi.
	\end{align*}
	Taking this into the previous equation, we arrive at the following energy equality: 
	\begin{align}\label{eq:lm6/E}
		&\int_0^B 4G^3 |\partial_\eta G|^2 d\eta + \int_0^B \int_{-\pi}^\pi(\varphi-G^4)^2 d\phi d\eta +  \frac12 \int_0^B\int_{-\pi}^\pi F(\eps;\eta) \sin\phi (\varphi-G^4)^2 d\phi d\eta \nonumber\\
		&\quad + \frac12 \int_{\sin\phi<0} |\sin\phi| (\varphi(0,\phi)-G^4(0))^2  d\phi = \frac12 \int_{\sin\phi>0} \sin\phi (\varphi_b(\phi)-G^4(0))^2  d\phi.
	\end{align}
	Note due to \eqref{eq:Fin4}, $\|F\|_{L^\infty}\le \tfrac{4}{4-3\delta}\eps$, hence for $\delta\le 1$,
	\begin{align}\label{eq:lm6/F}
		|F(\eps;\eta)\sin\phi| \le \frac{4}{4-3\delta} \eps \le 4\eps,
	\end{align}
	and for $\eps \le 1/4$,
	\begin{align*}
		1+\frac12 F(\eps;\eta)\sin\phi \ge 1- 2\eps \ge \frac12.
	\end{align*}
	The energy equality \eqref{eq:lm6/E} thus implies \eqref{eq:est/B1}.

	\bigskip

	\emph{\textbf{The weighted estimate.}} Let $\alpha>0$ be a positive constant. We multiply \eqref{eq:B1} by $e^{2\alpha \eta} G^4$ and \eqref{eq:B2} by $e^{2\alpha \eta}\varphi$ and integrate to get 
	\begin{align} \label{eq:lm6/5}
		&- \int_0^B e^{2\alpha \eta} G^4 \partial_\eta^2 G d\eta + \int_0^B \int_{-\pi}^\pi e^{2\alpha \eta}\sin\phi \partial_\eta \frac{\varphi^2}{2} d\phi d\eta + \int_0^B \int_{-\pi}^\pi e^{2\alpha\eta}(\varphi-G^4)^2 d\phi d\eta \nonumber\\
		&\quad - \int_0^B e^{2\alpha\eta} F(\eps;\eta) G^4 \partial_\eta G d\eta + \int_0^B \int_{-\pi}^\pi e^{2\alpha\eta } F(\eps;\eta)\cos\phi \varphi \partial_\phi \varphi d\phi d\eta = 0.
	\end{align}
	Using integration-by-parts and the boundary conditions \eqref{eq:B1b}-\eqref{eq:B2b}, we obtain 
	\begin{align*}
		&- \int_0^B e^{2\alpha \eta} G^4 \partial_\eta^2 G d\eta + \int_0^B \int_{-\pi}^\pi e^{2\alpha \eta}\sin\phi \partial_\eta \frac{\varphi^2}{2} d\phi d\eta \nonumber\\
		&\quad =- e^{2\alpha\eta} G^4 \partial_\eta G\bigg|_0^B +\int_0^B e^{2\alpha \eta} 4 G^3|\partial_\eta G|^2 d\eta + \int_0^B 2\alpha e^{2\alpha\eta} G^4 \partial_\eta G d\eta + \frac12 \int_{-\pi}^{\pi} e^{2\alpha\eta} \sin\phi \varphi^2 d\phi \bigg|_0^B \nonumber\\
		&\qquad - \frac12 \int_0^B\int_{-\pi}^{\pi} 2\alpha e^{2\alpha\eta} \sin\phi \varphi^2 d\phi d\eta \nonumber\\
		&\quad =  G^4(0) \partial_\eta G(0)+\int_0^B e^{2\alpha \eta} 4 G^3|\partial_\eta G|^2 d\eta + \int_0^B 2\alpha e^{2\alpha\eta} G^4 \partial_\eta G d\eta - \frac12 \int_{-\pi}^{\pi}  \sin\phi \varphi^2(0,\phi) d\phi \nonumber\\
		&\qquad - \frac12 \int_0^B\int_{-\pi}^{\pi} 2\alpha e^{2\alpha\eta} \sin\phi \varphi^2 d\phi d\eta \nonumber\\
		&\quad = -\frac12 \langle \sin\phi (\varphi(0,\phi) - G_b^4)^2\rangle + \int_0^B e^{2\alpha \eta} 4 G^3|\partial_\eta G|^2 d\eta -\alpha \int_0^B\int_{-\pi}^\pi \alpha e^{2\alpha\eta} \sin\phi (\varphi-G^4)^2 \phi d\eta \nonumber\\
		&\quad = -\frac12 \int_{\sin\phi>0} \sin\phi (\varphi_b(\phi)-G_b^4)^2 d\phi + \frac12 \int_{\sin\phi<0} |\sin\phi|(\varphi(0,\phi)-G_b^4)^2 d\phi + \int_0^B e^{2\alpha \eta} 4 G^3|\partial_\eta G|^2 d\eta \nonumber\\
		&\qquad -\alpha \int_0^B\int_{-\pi}^\pi  e^{2\alpha\eta} \sin\phi (\varphi-G^4)^2 \phi d\eta.
	\end{align*}
	Using integration-by-parts and the relation \eqref{eq:G=}, we get 
	\begin{align*}
		& - \int_0^B e^{2\alpha\eta} F(\eps;\eta) G^4 \partial_\eta G d\eta + \int_0^B \int_{-\pi}^\pi e^{2\alpha\eta } F(\eps;\eta)\cos\phi \varphi \partial_\phi \varphi d\phi d\eta \nonumber\\
		&\quad =  - \int_0^B e^{2\alpha\eta} F(\eps;\eta) G^4 \partial_\eta G d\eta + \frac12 \int_0^B \int_{-\pi}^\pi e^{2\alpha\eta } F(\eps;\eta)\sin\phi   \varphi^2 d\phi d\eta \nonumber\\
		&\quad = \frac12 \int_0^B \int_{-\pi}^\pi e^{2\alpha\eta} F(\eps;\eta) \sin\phi(\varphi-G^4)^2 d\phi d\eta .
	\end{align*}
	Taking the above two equations into \eqref{eq:lm6/5} gives 
	\begin{align*}
		&\int_0^B e^{2\alpha\eta} 4G^3 |\partial_\eta G|^2 d\eta + \int_0^B \int_{-\pi}^\pi e^{2\alpha\eta} (\varphi-G^4)^2 d\phi d\eta +  \frac12 \int_0^B\int_{-\pi}^\pi e^{2\alpha\eta} F(\eps;\eta) \sin\phi (\varphi-G^4)^2 d\phi d\eta \nonumber\\
		&\quad -\alpha \int_0^B\int_{-\pi}^\pi  e^{2\alpha\eta} \sin\phi (\varphi-G^4)^2 \phi d\eta+ \frac12 \int_{\sin\phi<0} |\sin\phi| (\varphi(0,\phi)-G^4_b)^2  d\phi \nonumber\\
		&\qquad = \frac12 \int_{\sin\phi>0} \sin\phi (\varphi_b - G_b^4)^2 d\phi.
	\end{align*}
	Using \eqref{eq:lm6/F} and for $\eps\le 1/4$ and $0<\alpha\le \tfrac14$ small,
	\begin{align}\label{eq:alpha/e}
		\left|1 + \frac12 F(\eps;\eta)\sin\phi - \alpha \right| \ge 1 - 2\eps - \alpha \ge \frac14.
	\end{align}
	Taking this into the previous inequality leads to the estimate \eqref{eq:est/B2}.

	\emph{\textbf{Boundness of $\varphi$.}} Note that from the estimates \eqref{eq:est/B1} and \eqref{eq:est/B2}, $\varphi-G^4$ is in $L^2$ but we lack an estimate on $\varphi$ or $G$. However, from the definition of $\omega$,
	 $\langle \omega\rangle=0$ and then
	\begin{align}\label{eq:omega2}
		\int_{-\pi}^\pi (\varphi - G^4)^2 d\phi = \int_{-\pi}^\pi \left( \omega + q - G^4\right)^2 d\phi = \int_{-\pi}^\pi \omega^2 d\phi + 2\pi (q-G^4)^2.
	\end{align}
	Hence we get from \eqref{eq:est/B1} and \eqref{eq:est/B2} that 
	\begin{align*}
		\|\omega\|_{L^2([0,B]\times[-\pi,\pi))}^2=\int_{0}^B \int_{-\pi}^\pi \omega^2 d\phi d\eta \le \int_{\sin\phi>0}(\varphi_b-G_b^4)^2 d\phi,
	\end{align*}
	and 
	\begin{align*}
		\|e^{\alpha\eta}\omega\|_{L^2([0,B]\times[-\pi,\pi))}=\int_{0}^B \int_{-\pi}^\pi e^{2\alpha\eta}\omega^2 d\phi d\eta \le 2\int_{\sin\phi >0 }(\varphi_b-G_b^4)^2 d\phi.
	\end{align*}
	Next we derive estimate on $q$. Multiplying \eqref{eq:B2} by $\sin\phi$ and integrate over $[-\pi,\pi)$ gives 
	\begin{align*}
		\frac{d}{d\eta}\langle \sin^2\phi \varphi \rangle &= - F(\eps;\eta)\langle \cos\phi \sin\phi \partial_\phi \varphi \rangle  - \langle (\varphi-G^4) \sin\phi  \rangle \nonumber\\
		&= \frac12 F(\eps;\eta) \langle \partial_\phi \sin(2\phi) \varphi \rangle - \langle \sin\phi \varphi \rangle \\
		&= F(\eps;\eta) \langle \cos(2\phi) \omega \rangle - \langle \sin\phi \omega \rangle.
	\end{align*}
	Integrating over $[0,\eta]$ and using \eqref{eq:Fin3} gives 
	\begin{align*}
		&|\langle \sin^2\phi \varphi(\eta,\cdot)\rangle - \langle \sin^2\phi \varphi(0,\cdot)\rangle| \\
		&\quad= \int_0^\eta (F(\eps;\tau) \langle \cos(2\phi) \omega(\tau,\cdot)\rangle - \langle \sin\phi \omega(\tau,\cdot)\rangle) d\tau \nonumber\\
		&\quad \le \left(\int_0^\eta \int_{-\pi}^\pi \cos^2 (2\phi)F^2(\eps;\tau)d\phi d\eta\right)^{1/2} \left(\int_0^\eta \int_{-\pi}^\pi \omega^2d\phi d\eta\right)^{1/2} \nonumber\\
		&\qquad + \left(\int_0^\eta \int_{-\pi}^\pi \sin^2\phi d\phi d\eta \right)^{1/2} \left(\int_0^\eta \int_{-\pi}^\pi \omega^2d\phi d\eta\right)^{1/2} \nonumber\\
		&\le \sqrt{\pi} \|F\|_{L^2([0,\eta])} \|\omega\|_{L^2([0,B]\times[-\pi,\pi))} + \sqrt{\pi}\sqrt{\eta} \|\omega\|_{L^2([0,B]\times[-\pi,\pi))} \nonumber \\
		&\le (\frac{3\delta}{4-3\delta}  \sqrt{\pi} + \sqrt{\pi}\sqrt{\eta})\|\omega\|_{L^2([0,B]\times[-\pi,\pi))} \nonumber \\
		&\le C(\delta)(1+\sqrt{\eta})  \left(\int_{\sin\phi>0}(\varphi_b-G_b^4)^2 d\phi\right)^{1/2}.
	\end{align*}
	Since 
	\begin{align*}
		|\langle \sin^2\phi \varphi(0,\cdot)\rangle| &= \left|\int_{\sin\phi>0} \sin^2\phi \varphi_b d\phi + \int_{\sin\phi<0}\sin^2\phi \varphi(0,\phi)d\phi\right| \nonumber\\
		&\le \int_{\sin\phi>0} \sin^2\phi \varphi_b d\phi  + \left(\int_{\sin\phi<0} |\sin \phi|^{\tfrac94} d\phi \right)^{1/2} \left(\int_{\sin\phi<0}|\sin\phi| \varphi^2(0,\phi) d\phi \right)^{1/2} \nonumber\\
		&\le \int_{\sin\phi>0} \sin^2\phi \varphi_b d\phi + C \left(\int_{\sin\phi>0}(\varphi_b-G_b^4)^2 d\phi\right)^{1/2},
	\end{align*}
	we get from the previous inequality that 
	\begin{align*}
		|\langle \sin^2\phi \varphi(\eta,\cdot)\rangle| &\le |\langle \sin^2\phi \varphi(0,\cdot)\rangle| +  C(\delta)(1+\sqrt{\eta})  \left(\int_{\sin\phi>0}(\varphi_b-G_b^4)^2 d\phi\right)^{1/2} \nonumber\\
		&\le C(\delta,\varphi_b,G_b)(1+\sqrt{\eta}),
	\end{align*}
	with $C(\delta,\varphi_b,G_b)>0$ is a positive constant depending on $\delta,\varphi_b,G_b$. Since 
	\[
	\langle \sin^2\phi \varphi(\eta,\cdot)\rangle = \langle \sin^2 \phi q(\eta) \rangle + \langle \sin^2\phi \omega(\eta,\cdot)\rangle = \pi q(\eta) +\langle \sin^2\phi \omega(\eta,\cdot)\rangle  
	\]
	and 
	$$|\langle \sin^2\phi \omega(\eta,\cdot)\rangle| \le \|\sin^2\phi\|_{L^2([-\pi,\pi))}\|\omega(\eta,\cdot)\|_{L^2([-\pi,\pi))} \le C \|\omega(\eta,\cdot)\|_{L^2([-\pi,\pi))},$$ 
	we obtain 
	\begin{align*}
		|q(\eta)| \le C(\delta,\varphi_b,G_b)(1+\sqrt{\eta}+ \|\omega(\eta,\cdot)\|_{L^2([-\pi,\pi))})
	\end{align*}
	for any $\eta \in [0,B]$. Similarly, using the relation \eqref{eq:G=}, we can get 
	\begin{align*}
		|G(\eta)| &= \left|G_b+ \int_0^\eta \partial_\eta G d\eta\right| \nonumber\\
		&\le G_b + \int_0^\eta | \langle \sin\phi \varphi \rangle | d\eta \nonumber\\
		&\le G_b + \int_0^\eta |\langle \sin\phi \omega \rangle| d\eta  \nonumber\\
		&\le G_b + \left(\int_0^\eta \int_{-\pi}^\pi \sin^2\phi d\phi d\eta\right)^{1/2} \|\omega\|_{L^2([0,\eta]\times [-\pi,\pi))} \nonumber\\
		&\le G_b + C \sqrt{\eta} \left(\int_{\sin\phi>0}(\varphi_b-G_b^4)^2 d\phi\right)^{1/2} \nonumber\\
		&\le C(G_b,\varphi_b)(1+\sqrt{\eta}),
	\end{align*}
	which finishes the proof.
\end{proof}


	\subsection{Existence on the half line} Next we show the existence of solutions for system \eqref{eq:n1}-\eqref{eq:n2}. 
	\begin{proof}[Proof of Theorem \ref{thm.nm}]
		From \eqref{eq:est/B3} and \eqref{eq:est/B4}, $\varphi_B^\eps = q_B^\eps + \omega_B^\eps \in L^2_{\rm loc}([0,\infty)\times [-\pi,\pi))$ and $G_B^\eps \in L^2_{\rm loc}([0,\infty))$, thus there exists a subsequence $\{B_k\}_{k=1}^\infty$ such that 
		\begin{align*}
			&G_{B_k}^\eps \rightharpoonup G^\eps,\quad \text{weakly in } L^2_{\rm loc}([0,\infty)), \\
			&\varphi^\eps_{B_k} \rightharpoonup \varphi^\eps,\quad \text{weakly in }L^2_{\rm loc}([0,\infty)\times [-\pi,\pi)).
		\end{align*}
		 By the estimate \eqref{eq:est/B3}, $\|\partial_\eta G_B^\eps\|_{L^2([0,B])}$ is uniformly bounded, hence by trace theorem we get $G^\eps(0)=G_B^\eps(0)=G_b$ holds. Moreover, from \eqref{eq:est/B1}, 
		\begin{align*}
			\sin\phi \partial_\eta \varphi_B^\eps + F(\eps;\eta) \cos\phi \partial_\phi \varphi_B^\eps \in L^2([0,B]\times [-\pi,\pi)),
		\end{align*}
		by the the Ukai's trace theorem \cite[Theorem 5.1]{ukai1986solutions}, $\varphi^\eps$ is well defined on the boundary and 
		\begin{align*}
			\varphi^\eps(0,\phi) = \varphi_b(\phi),\quad \text{for }\sin\phi>0.
		\end{align*}
		Since $\|G^\eps_B\|_{H^1([0,B])}$ is bounded for any finite $B$, the Sobolev embedding theorem implies that $\|G^\eps_B\|_{C([0,B])}$ is bounded and thus as $B\to 0$, 
		\begin{align*}
			(G_B^\eps)^4 \to (G^\eps)^4, \quad \text{uniformly in } C([0,B]).
		\end{align*}
		Therefore,  $(G^\eps,\phi^\eps)$ solves system \eqref{eq:n1}-\eqref{eq:n2} with boundary conditions \eqref{eq:n1b}-\eqref{eq:n2b}. Moreover, due to $0\le G^\eps_B \le \gamma$ and $0\le \varphi^\eps_B \le \gamma^4$, $0\le G^\eps \le \gamma$ and $0\le \varphi^\eps \le \gamma^4$ also hold.

		Next we show the uniform boundness of the solutions. First, by the lower semi-continuity of the $L^2$ norm, \eqref{eq:est/B1} and \eqref{eq:est/B2} implies \eqref{eq:est/n1} and \eqref{eq:est/n2}.

		From the relation \eqref{eq:G=}, $\partial_\eta G^\eps_B = \langle \sin\phi \varphi_B^\eps \rangle$, we can pass to the weak limit $B\to\infty$ and get that
		\begin{align*}
			\partial_\eta G^\eps = \langle \sin\phi \varphi^\eps \rangle
		\end{align*}
		holds in the weak sense in $L^2_{\rm loc}([0,\infty))$. Therefore, 
		\begin{align*}
			\int_0^\infty |\partial_\eta G^\eps|^2 d\eta &= \int_0^\infty |\langle \sin\phi \varphi^\eps \rangle|^2 d\eta = \int_0^\infty |\langle \sin\phi (\varphi^\eps - (G^\eps)^4) \rangle|^2 d\eta \nonumber\\
			&\le  \int_0^\infty \left(\int_{-\pi}^\pi \sin^2\phi d\phi \int_{-\pi}^\pi (\varphi^\eps-(G^\eps)^4)^2 d\phi\right) d\eta \le \pi\int_{\sin\phi>0} \sin\phi (\varphi_b(\phi)-G_b^4)^2  d\phi,
		\end{align*}
		and 
		\begin{align*}
			&\int_0^\infty e^{2\alpha\eta} |\partial_\eta G^\eps|^2 d\eta = \int_0^\infty  e^{2\alpha\eta} |\langle \sin\phi \varphi^\eps \rangle |^2 d\eta = \int_0^\infty e^{2\alpha\eta}  |\langle \sin\phi (\varphi^\eps - (G^\eps)^4) \rangle|^2d\eta \nonumber\\
			&\quad \le \pi \int_0^\infty \int_{-\pi}^\pi e^{2\alpha\eta} (\varphi^\eps-(G^\eps)^4)^2 d\phi d\eta\le 2\pi \int_{\sin\phi>0} \sin\phi (\varphi_b(\phi)-G_b^4)^2  d\phi.
		\end{align*}
		Therefore, 
		\begin{align*}
			G^\eps(\eta) &= G_b + \int_0^\eta \partial_\eta G^\eps d \eta \nonumber\\
			&= G_b + \int_0^\eta e^{-\alpha\eta} e^{\alpha\eta} \partial_\eta G^\eps d\eta \le G_b + \left(\int_0^\eta e^{-2\alpha\eta} d\eta \right)^{1/2} \left(\int_0^\eta e^{2\alpha\eta}|\partial_\eta G^\eps|^2 d\eta\right)^{1/2} \nonumber\\
			&\le G_b +  \frac{\sqrt{1-e^{-2\alpha\eta}}}{\sqrt{2\alpha}} \sqrt{2\pi} \left(\int_{\sin\phi>0} \sin\phi (\varphi_b(\phi)-G_b^4)^2  d\phi\right)^{1/2}.
		\end{align*}
		Hence as $\eta\to \infty$, $G^\eps(\eta)$ is uniformly bounded, and so $\lim_{\eta\to\infty}G^\eps(\eta)=:G^\eps_{\infty}$ exists. Moreover,
		\begin{align*}
			|G^\eps(\eta) - G_\infty^\eps| &= \left|-\int_{\eta}^\infty \partial_\eta G^\eps d\eta \right| = \left|\int_\eta^\infty e^{-\alpha\eta} e^{\alpha\eta} \partial_\eta G^\eps d\eta \right| \nonumber\\
			&\le \left(\int_\eta^\infty e^{-2\alpha\eta} d\eta\right)^{1/2} \left(\int_\eta^\infty e^{2\alpha\eta} |\partial_\eta G^\eps|^2 d\eta \right)^{1/2} \nonumber\\
			&\le \frac{e^{-\alpha\eta}}{\sqrt{2\alpha}} \left(\int_0^\infty e^{2\alpha\eta} |\partial_\eta G^\eps|^2 d\eta \right)^{1/2} \nonumber\\
			&\le e^{-\alpha\eta}\frac{\sqrt{\pi}}{\sqrt{\alpha}} \left(\int_{\sin\phi>0} \sin\phi (\varphi_b(\phi)-G_b^4)^2  d\phi\right)^{1/2}.
		\end{align*}

		Finally, we show $\varphi^\eps - (G^\eps_\infty)^4$ decays exponentially to $0$ as $\eta\to 0$. We use Lemma \ref{lem:formula_inf} in Appendix \ref{sec:apA} to get this. Note that from the exponential decay of $G^\eps$,
		\begin{align*}
			|(G^\eps(\eta))^4 - (G^\eps_{\infty})^4| &\le |G^\eps(\eta) - G^\eps_{\infty}| \cdot |(G^\eps(\eta))^3 - (G^\eps(\eta))^2G^\eps_{\infty} + G^\eps(\eta) (G^\eps_\infty)^2 - (G^\eps_\infty)^3| \nonumber\\
			&\quad \le e^{-\alpha\eta}\frac{\sqrt{\pi}}{\sqrt{\alpha}} \left(\int_{\sin\phi>0} \sin\phi (\varphi_b(\phi)-G_b^4)^2  d\phi\right)^{1/2} \cdot 4 \gamma^3.
		\end{align*}
		Let $C_b:= \frac{\sqrt{\pi}}{\sqrt{\alpha}} \left(\int_{\sin\phi>0} \sin\phi (\varphi_b(\phi)-G_b^4)^2  d\phi\right)^{1/2} \cdot 4 \gamma^3$, then $|(G^\eps(\eta))^4 - (G^\eps_{\infty})^4| \le C_b e^{-\alpha\eta}.$
		First, for $\sin\phi>0$, from \eqref{eq:varphi--1}, 
		\begin{align*}
			&|\varphi^\eps(\eta,\phi) -(G^\eps_\infty)^4 | \\
			&\quad= \left|\varphi_b(\phi'(\phi,\eta,0)) e^{-\int_0^\eta \frac{1}{\sin \phi'(\phi,\eta,\xi)} d\xi} + \int_0^\eta \frac{(G^\eps(\xi))^4}{\sin \phi'(\phi,\eta,\xi)} e^{-\int_\xi^\eta \frac{1}{\sin \phi'(\phi,\eta,\rho)}d\rho} d\xi - (G^\eps_\infty)^4\right| \nonumber\\
			&\quad= \left|(\varphi_b(\phi'(\phi,\eta,0)) - (G_\infty^\eps)^4) e^{-\int_0^\eta \frac{1}{\sin \phi'(\phi,\eta,\xi)} d\xi} + \int_0^\infty\frac{(G^\eps(\xi))^4 - (G_\infty^\eps)^4}{\sin \phi'(\phi,\eta,\xi)} e^{-\int_\xi^\eta \frac{1}{\sin \phi'(\phi,\eta,\rho)}d\rho} d\xi \right| \nonumber\\
			&\quad \le |\varphi_b(\phi'(\phi,\eta,0)) - (G_\infty^\eps)^4 | e^{-\eta} \nonumber\\
			&\qquad + \frac{\sqrt{\pi}}{\sqrt{\alpha}} \left(\int_{\sin\phi>0} \sin\phi (\varphi_b(\phi)-G_b^4)^2  d\phi\right)^{1/2} \cdot 4 \gamma^3 \int_0^\eta \frac{e^{-\alpha\xi}}{\sin \phi'(\phi,\eta,\xi)} e^{-\int_\xi^\eta \frac{1}{\sin \phi'(\phi,\eta,\rho)}d\rho} d\xi.
		\end{align*}
		Using integration-by-parts, we obtain 
		\begin{align*}
			\int_0^\eta \frac{e^{-\alpha\xi}}{\sin \phi'(\phi,\eta,\xi)} e^{-\int_\xi^\eta \frac{1}{\sin \phi'(\phi,\eta,\rho)}d\rho} d\xi &= \int_0^\eta e^{-\alpha\xi} d e^{-\int_\xi^\eta \frac{1}{\sin \phi'(\phi,\eta,\rho)}d\rho} \nonumber\\
			&= e^{-\alpha\xi} e^{-\int_\xi^\eta \frac{1}{\sin \phi'(\phi,\eta,\rho)}d\rho}\big|_0^\eta + \alpha \int_0^\eta e^{-\int_\xi^\eta \frac{1}{\sin \phi'(\phi,\eta,\rho)}d\rho} e^{-\alpha\xi} d\xi \nonumber\\
			&\le e^{-\alpha \eta} - e^{-\int_0^\eta \frac{1}{\sin \phi'(\phi,\eta,\rho)}d\rho} + \alpha \int_0^\eta e^{-(\eta-\xi)}e^{-\alpha\xi}d\xi \nonumber\\
			&\le e^{-\alpha \eta} + \frac{\alpha}{1-\alpha} (e^{-\alpha\eta}-e^{-\eta}) \le \frac{1}{1-\alpha} e^{-\alpha\eta}.
		\end{align*}
		Taking this into the previous inequality leads to 
		\begin{align*}
			&|\varphi^\eps(\eta,\phi) -(G^\eps_\infty)^4 | \nonumber\\
			&\quad \le  |\varphi_b(\phi'(\phi,\eta,0)) - (G_\infty^\eps)^4 | e^{-\eta} + e^{-\alpha\eta}\frac{4\gamma^3\sqrt{\pi}}{\sqrt{\alpha}\left(\int_{\sin\phi>0} \sin\phi (\varphi_b(\phi)-G_b^4)^2  d\phi\right)^{1/2}(1-\alpha)},
		\end{align*}
		for any $\eta \ge 0$ and $\sin\phi>0$.

		Similarly, for $\sin\phi<0$ and $|E(\phi,\eta)|\le e^{-V_\infty}$, we use \eqref{eq:varphi--2} and get 
		\begin{align*}
			\varphi^\eps(\eta,\phi) &= \varphi_b(\phi'(-\phi,\eta,0)) e^{-\int_0^\infty + \int_\eta^\infty \frac{1}{\sin\phi'(-\phi,\eta,\xi)}d\xi} + \int_0^\infty \frac{(G^\eps(\xi))^4}{\sin\phi'(-\phi,\eta,\xi)} e^{-\int_\xi^\infty + \int_{\eta}^\infty \frac{1}{\sin\phi'(-\phi,\eta,\rho)}d\rho} d\xi \nonumber\\
			&\quad + \int_\eta^\infty \frac{(G^\eps(\xi))^4}{\sin \phi'(-\phi,\eta,\xi)} e^{\int_\xi^\eta \frac{1}{\sin \phi'(-\phi,\eta,\rho)}d\rho} d\xi \nonumber\\
			&=\int_\eta^\infty \frac{(G^\eps(\xi))^4}{\sin \phi'(-\phi,\eta,\xi)} e^{\int_\xi^\eta \frac{1}{\sin \phi'(-\phi,\eta,\rho)}d\rho} d\xi. 
		\end{align*}
		Hence 
		\begin{align*}
			&|\varphi^\eps(\eta,\phi) - (G^\eps_\infty)^4| \nonumber\\
			&\quad= \bigg|\int_\eta^\infty \frac{(G^\eps(\xi))^4}{\sin \phi'(-\phi,\eta,\xi)} e^{\int_\xi^\eta \frac{1}{\sin \phi'(-\phi,\eta,\rho)}d\rho} d\xi- (G^\eps_\infty)^4\bigg| \nonumber\\
			&\quad= \bigg| \int_\eta^\infty \frac{(G^\eps(\xi))^4-(G^\eps_\infty)^4}{\sin \phi'(-\phi,\eta,\xi)} e^{\int_\xi^\eta \frac{1}{\sin \phi'(-\phi,\eta,\rho)}d\rho} d\xi \bigg| \nonumber\\
			&\quad\le \int_\eta^\infty \frac{|(G^\eps(\xi))^4-(G^\eps_\infty)^4|}{\sin \phi'(-\phi,\eta,\xi)} e^{\int_\xi^\eta \frac{1}{\sin \phi'(-\phi,\eta,\rho)}d\rho} d\xi \nonumber\\
			&\quad \le \frac{\sqrt{\pi}}{\sqrt{\alpha}} \left(\int_{\sin\phi>0} \sin\phi (\varphi_b(\phi)-G_b^4)^2  d\phi\right)^{1/2} \cdot 4 \gamma^3 \int_\eta^\infty \frac{e^{-\alpha \xi}}{\sin \phi'(-\phi,\eta,\xi)} e^{\int_\xi^\eta \frac{1}{\sin \phi'(-\phi,\eta,\rho)}d\rho} d\xi \nonumber\\
			&\quad \le \frac{4\gamma^3\sqrt{\pi}}{\sqrt{\alpha}} \left(\int_{\sin\phi>0} \sin\phi (\varphi_b(\phi)-G_b^4)^2  d\phi\right)^{1/2} \int_\eta^\infty -e^{-\alpha\xi} d e^{\int_\xi^\eta \frac{1}{\sin \phi'(-\phi,\eta,\rho)}d\rho} \nonumber\\
			&\quad \le \frac{4\gamma^3\sqrt{\pi}}{\sqrt{\alpha}} \left(\int_{\sin\phi>0} \sin\phi (\varphi_b(\phi)-G_b^4)^2  d\phi\right)^{1/2}\nonumber\\
			&\qquad \cdot \bigg(-e^{-\alpha\xi}e^{\int_\xi^\eta \frac{1}{\sin \phi'(-\phi,\eta,\rho)}d\rho}\big|_\eta^\infty + \alpha \int_\eta^\infty e^{-\alpha\xi} e^{\int_\xi^\eta \frac{1}{\sin \phi'(-\phi,\eta,\rho)}d\rho} d\xi\bigg) \nonumber\\
			&\quad \le \frac{4\gamma^3\sqrt{\pi}}{\sqrt{\alpha}} \left(\int_{\sin\phi>0} \sin\phi (\varphi_b(\phi)-G_b^4)^2  d\phi\right)^{1/2} (e^{-\alpha\eta} + \alpha\int_\eta^\infty e^{-\alpha\xi} e^{\eta-\xi} d\xi) \nonumber\\
			&\quad \le \frac{4\gamma^3\sqrt{\pi}}{\sqrt{\alpha}} \left(\int_{\sin\phi>0} \sin\phi (\varphi_b(\phi)-G_b^4)^2  d\phi\right)^{1/2}(e^{-\alpha\eta} + \frac{\alpha}{\alpha+1}e^{-\alpha \eta}) \nonumber\\
			&\quad \le e^{-\alpha\eta} \frac{4\gamma^3\sqrt{\pi}(2\alpha+1)}{\sqrt{\alpha}(\alpha+1)} \left(\int_{\sin\phi>0} \sin\phi (\varphi_b(\phi)-G_b^4)^2  d\phi\right)^{1/2}.
		\end{align*}
		For $\sin\phi<0$ and $E(\phi,\eta)> e^{-V_\infty}$, we use \eqref{eq:varphi--3} and get 
		\begin{align*}
			&|\varphi^\eps(\eta,\phi) - (G_\infty^\eps)^4| \nonumber\\
			&\quad= \Big|\varphi_b(\phi'(-\phi,\eta,0)) e^{-\int_0^{\eta_+} + \int_\eta^{\eta_+} \frac{1}{\sin\phi'(-\phi,\eta,\xi)}d\xi} + \int_0^{\eta_+} \frac{(G^\eps(\xi))^4}{\sin\phi'(-\phi,\eta,\xi)} e^{-\int_\xi^{\eta_+} + \int_{\eta}^{\eta_+} \frac{1}{\sin\phi'(-\phi,\eta,\rho)}d\rho} d\xi \nonumber\\
			&\qquad + \int_\eta^{\eta_+} \frac{(G^\eps(\xi))^4}{\sin \phi'(-\phi,\eta,\xi)} e^{\int_\xi^{\eta} \frac{1}{\sin \phi'(-\phi,\eta,\rho)}d\rho} d\xi - (G_\infty^\eps)^4 \Big| \nonumber\\
			&\quad= \Big| (\varphi_b(\phi'(-\phi,\eta,0))-(G^\eps_\infty)^4) e^{-\int_0^{\eta_+} + \int_\eta^{\eta_+} \frac{1}{\sin\phi'(-\phi,\eta,\xi)}d\xi} \nonumber\\
			&\qquad + \int_0^{\eta_+} \frac{(G^\eps(\xi))^4 - (G_\infty^\eps)^4}{\sin\phi'(-\phi,\eta,\xi)} e^{-\int_\xi^{\eta_+} + \int_{\eta}^{\eta_+} \frac{1}{\sin\phi'(-\phi,\eta,\rho)}d\rho} d\xi \nonumber\\
			&\qquad + \int_\eta^{\eta_+} \frac{(G^\eps(\xi))^4 - (G^\eps_\infty)^4}{\sin \phi'(-\phi,\eta,\xi)} e^{\int_\xi^{\eta} \frac{1}{\sin \phi'(-\phi,\eta,\rho)}d\rho} d\xi \Big| \nonumber\\
			&\quad\le | (\varphi_b(\phi'(-\phi,\eta,0))-(G^\eps_\infty)^4) |e^{-\int_0^{\eta_+} + \int_\eta^{\eta_+} \frac{1}{\sin\phi'(-\phi,\eta,\xi)}d\xi} \nonumber\\
			&\qquad + C_b \int_0^{\eta_+} \frac{e^{-\alpha\xi}}{\sin\phi'(-\phi,\eta,\xi)} e^{-\int_\xi^{\eta_+} + \int_{\eta}^{\eta_+} \frac{1}{\sin\phi'(-\phi,\eta,\rho)}d\rho} d\xi + C_b \int_\eta^{\eta_+} \frac{e^{-\alpha\xi}}{\sin \phi'(-\phi,\eta,\xi)} e^{\int_\xi^{\eta} \frac{1}{\sin \phi'(-\phi,\eta,\rho)}d\rho} d\xi \nonumber\\
			&\quad \le | (\varphi_b(\phi'(-\phi,\eta,0))-(G^\eps_\infty)^4)| e^{-(2\eta_+ - \eta)} + C_b \int_0^{\eta_+} e^{-\alpha\xi} d e^{-\int_\xi^{\eta_+} + \int_{\eta}^{\eta_+} \frac{1}{\sin\phi'(-\phi,\eta,\rho)}d\rho} \nonumber \\
			&\qquad - C_b \int_\eta^{\eta_+} e^{-\alpha\xi} de^{\int_\xi^{\eta} \frac{1}{\sin \phi'(-\phi,\eta,\rho)}d\rho}  \nonumber\\
			&\quad \le | (\varphi_b(\phi'(-\phi,\eta,0))-(G^\eps_\infty)^4)| e^{-(2\eta_+ - \eta)} + C_b e^{-\alpha\xi} e^{-\int_\xi^{\eta_+} + \int_{\eta}^{\eta_+} \frac{1}{\sin\phi'(-\phi,\eta,\rho)}d\rho}\Big|_{\xi=\eta}^{\eta_+} \nonumber\\
			&\qquad + \alpha C_b \int_0^{\eta_+}  e^{-\alpha\xi} e^{-\int_\xi^{\eta_+} + \int_{\eta}^{\eta_+}\frac{1}{\sin\phi'(-\phi,\eta,\rho)}d\rho} d\xi -C_b e^{-\alpha\xi} e^{\int_\xi^{\eta} \frac{1}{\sin \phi'(-\phi,\eta,\rho)}d\rho} \Big|_{\xi=\eta}^{\eta_+} \nonumber\\
			&\qquad - \alpha C_b \int_{\eta}^{\eta_+} e^{-\alpha\xi} e^{\int_\xi^{\eta} \frac{1}{\sin \phi'(-\phi,\eta,\rho)}d\rho} d\xi \nonumber\\
			&\quad \le | (\varphi_b(\phi'(-\phi,\eta,0))-(G^\eps_\infty)^4)| e^{-(2\eta_+ - \eta)} + C_b e^{-\alpha\eta_+} e^{-(\eta_+ - \eta)} - C_b e^{-\alpha\eta}e^{-2\int_\eta^{\eta_+} \frac{1}{\sin\phi'(-\phi,\eta,\rho)}d\rho}  \nonumber\\
			&\qquad + \alpha C_b \int_0^{\eta_+} e^{-\alpha\xi} e^{-(2\eta_+-\eta-\xi)} d\xi - C_b e^{-\alpha\eta_+} e^{\int_{\eta_+}^{\eta} \frac{1}{\sin \phi'(-\phi,\eta,\rho)}d\rho} + C_b e^{-\alpha\eta} \nonumber\\
			&\qquad - \alpha C_b \int_{\eta}^{\eta_+} e^{-\alpha\xi} e^{\int_\xi^{\eta} \frac{1}{\sin \phi'(-\phi,\eta,\rho)}d\rho} d\xi \nonumber\\
			&\quad \le | (\varphi_b(\phi'(-\phi,\eta,0))-(G^\eps_\infty)^4)| e^{-(2\eta_+ - \eta)} + C_b e^{-(\alpha+1)\eta_+ + \eta} + \alpha C_b\frac{e^{-(2\eta_+-\eta)}}{1-\alpha}(1-e^{\eta_+(1-\alpha)})  + C_b e^{-\alpha\eta} \nonumber\\
			&\le  | (\varphi_b(\phi'(-\phi,\eta,0))-(G^\eps_\infty)^4)| e^{-\eta} + C_b e^{-\alpha\eta} + \frac{\alpha}{1-\alpha}C_b e^{-\eta} + C_be^{-\alpha\eta},
		\end{align*}
		where in the last inequality we use $\eta\le \eta_+$, which implies $e^{-(2\eta_+-\eta)} \le e^{-\eta}$.
		
		Therefore, we conclude that in any of the three cases, 
		\begin{align}\label{eq:non/conv}
			|\varphi^\eps(\eta,\phi)-(G_\infty^\eps)^4| \le Ce^{-\alpha\eta},
		\end{align}
		where $C$ depends on $G_b$, $\varphi_b$, $G_\infty^\eps$, $C_b$ and $\alpha$.
	\end{proof}	
	\begin{rem}
We can see from the above proof that the geometric correction gives a slightly weaker estimate on the convergence rate.  By \eqref{eq:alpha/e}, $\alpha \le 1-\frac12 \|F\|_{L^\infty} \le 1-2\eps$. Hence the convergence rate $\alpha$ should be strictly less than $1$. In the flat case without geometric correction, $\alpha$ can be $1$ \cite{Bounadrylayer2019GHM2}.
	\end{rem}

\section{Linearized Milne problem}\label{sec4}
For the higher order boundary layer expansions \eqref{eq:g1} and \eqref{eq:g2}, the function $\tilde{T}_k$, $\tilde{\psi}_k$, $k\ge 1$ satisfies the linear Milne problem of the form 
\begin{align}
	&\partial_\eta^2 G^\eps + F(\eps;\eta)\partial_\eta G^\eps + \langle \varphi^\eps - 4 h^3 G^\eps \rangle = S_1, \label{eq:l1}\\
	&\sin\phi\partial_\eta \varphi^\eps + F(\eps;\eta) \cos\phi \partial_\phi \varphi^\eps + \varphi^\eps - 4h^3 G^\eps = S_2, \label{eq:l2}
\end{align}
with boundary conditions 
\begin{align}
	&G^\eps(0) = 0, \label{eq:l1b}\\
	&\varphi^\eps(0,\phi) = \varphi_b(\phi),\quad \text{for }\sin\phi>0,\label{eq:l2b}
\end{align}
where $S_1$, $S_2$ are known functions.
This section is devoted to the proof of the existence for the above linear Milne problem. 
	The boundary data $\varphi_b$ and the source term $S_1$, $S_2$ are assumed to satisfy 
	\begin{align}\label{eq:as/s}
		|\varphi_b(\phi)| \le C_b,\quad |S_1(\eta)|,\, |S_2(\eta,\phi)| \le C_S e^{-K\eta},\quad \text{for any } \eta\in \mathbb{R}_+,\, \phi \in [-\pi,\pi).
	\end{align}

	In this section, we focus on proving the existence and uniqueness of solutions for the above linear Milne problem, stated in the following theorem.
	\begin{theorem}\label{thm.linearmilne}
		Given $h=h(\eta) \in C^1_{\rm \loc}([0,\infty))$ satisfying the spectral assumption \ref{asA}. Given functions $S_1=S_1(\eta) \in L^2_{\rm loc}([0,\infty))$, $S_2=S_2(\eta,\phi)\in L^2_{\rm loc}([0,\infty)\times [-\pi,\pi))$ satisfying \eqref{eq:as/s}.
	Then there exists a unique solution $(G^\eps,\varphi^\eps) \in L^2_{\rm loc}(\mathbb{R}_+)\times L^2_{\rm loc}(\mathbb{R}_+\times[-\pi,\pi))$ to system \eqref{eq:l1}-\eqref{eq:l2} with boundary conditions \eqref{eq:l1b}-\eqref{eq:l2b}. Moreover, the solution satisfies 
	\begin{align}\label{eq:linearmilne/est}
		&\frac18 \int_0^\infty \int_{-\pi}^\pi  e^{2\beta'\eta} (\varphi^\eps - 4h^3 G^\eps)^2 d\phi d\eta  + \frac12 \int_{\sin\phi<0} |\sin\phi| (\varphi^\eps)^2(0,\cdot) d\phi \nonumber\\
		&\quad  \le  \frac12 \int_{\sin\phi>0} \sin\phi \varphi_b^2 d\phi + \|e^{\beta'\eta} S_2\|_{L^2([0,\infty)\times [-\pi,\pi))}^2,
	\end{align}
	 and there exists constants $G^\eps_\infty$, such that 
	\begin{align*}
		|G^\eps(\eta)-G^\eps_\infty| \le C e^{-\beta'\eta}, \quad |G^\eps(\eta,\phi)-4h_\infty^3 G_\infty^\eps| \le C e^{-\beta'\eta},
	\end{align*}
	for any $0\le \beta'\le \min\{1/2,\beta,K\}$, where $C>0$ is a constant depending on $\varphi_b$.
	\end{theorem}

	\subsection{Spectral assumption \ref{asA}}
	The spectral assumption is motivated by the corresponding eigenvalue problem of system \eqref{eq:l1}-\eqref{eq:l2} without the geometric correction terms. The assumption was proposed in \cite{Bounadrylayer2019GHM2} to study system \eqref{eq:l1}-\eqref{eq:l2} without the geometric correction terms. Here we will show that under the spectral assumption, the geometric corrections does not affect the linear stability of the system. We first recall the following lemma that was proved in \cite{Bounadrylayer2019GHM2}.
	\begin{lemma}[\cite{Bounadrylayer2019GHM2}]
		Assume  $h\in C^1_{\rm loc}(\mathbb{R}_+)$  atisfies the spectral assumption \ref{asA}. Let $F=F(\eps;\eta)$ be given by \eqref{eq:Fdef}. Then it holds the following  inequality
		\begin{align}\label{eq:sp/ineq}
			\int_0^\infty e^{2\beta' \eta} 4h^3|\partial_\eta f|^2 d\eta + \int_0^\infty e^{2\beta' \eta} \partial_\eta (4h^3) f\partial_\eta f d\eta \ge 0,
		\end{align}
		holds for any $0\le \beta'\le \beta$ and for any $f\in C^1_{\rm loc}(\mathbb{R}_+)$ satisfying $f(0)=0$.
	\end{lemma}
	It has been proved in \cite{Bounadrylayer2019GHM2} that without geometric correction, system \eqref{eq:l1}-\eqref{eq:l2} has a unique solution under the assumption that $h$ satisfies the spectral assumption \ref{asA}, see Appendix \ref{sec:apB} for details.

\subsection{Existence on the bounded interval} 
We first consider the existence for system \eqref{eq:l1}-\eqref{eq:l2} on the bounded interval $\eta\in[0,B]$:
\begin{align}
	&\partial_\eta^2 G_B^\eps + F(\eps;\eta)\partial_\eta G^\eps_B + \langle \varphi^\eps_B - 4h^3 G^\eps_B \rangle = S_1, \label{eq:l1B}\\
	&\sin\phi \partial_\eta \varphi^\eps_B + F(\eps;\eta) \cos \phi \partial_\phi \varphi^\eps_B + \varphi^\eps_B - 4h^3 G^\eps_B =S_2,\label{eq:l2B}
\end{align}
with boundary conditions 
\begin{align}
	&G^\eps_B(0)=0,\quad \partial_\eta G^\eps(B)= 0,\label{eq:l1Bb}\\
	&\varphi^\eps_B(0,\phi) = \varphi_b(\phi),\quad \varphi^\eps_B(B,\phi)=\varphi^\eps_B(B,-\phi),\quad \text{for }\sin\phi>0.\label{eq:l2Bb}
\end{align}

We will consider the case when $\langle S_2 \rangle =S_1$ and prove the following lemma.
\begin{lemma}\label{lem.lB}
	Let $h$ satisfy the spectral assumption \ref{asA}. Assume $\varphi_b$, $S_1$, $S_2$ satisfy \eqref{eq:as/s} and $S_1(\eta) = \langle S_2(\eta,\cdot)\rangle$ for any $\eta\in\mathbb{R}_+$ and $\varphi_b \in C^\infty(0,\pi)$. Then there exists a solution $(G_B^\eps,\varphi_B^\eps) \in L^2([0,B])\times L^2([0,B]\times [-\pi,\pi))$ to system \eqref{eq:l1B}-\eqref{eq:l2B} with boundary conditions \eqref{eq:l1Bb}-\eqref{eq:l2Bb}, and the solution verifies 
	\begin{align}\label{eq:l1B/est}
		&\frac14\int_0^B \int_{-\pi}^\pi (\varphi_B^\eps - 4h^3 G_B^\eps)^2 d\phi d\eta + \frac12 \int_{\sin\phi<0} |\sin\phi| (\varphi_B^\eps)^2(0,\phi) d\phi \nonumber\\
		&\quad  \le \frac12 \int_{\sin\phi>0} |\sin\phi|(\varphi_b)^2 d\phi+\|S_2\|_{L^2([0,B]\times [-\pi,\pi))}^2,
	\end{align}
	for $\eps$ sufficiently small.
\end{lemma}

\begin{proof}
	The proof is based on Theorem \ref{thm:apB2} for existence of the linear Milne problem without geometric correction and the Leray-Schauder fixed point theorem.

	First we construct a sequence of functions $\{(G^k,\varphi^k)\}_{k=0}^\infty$ by $G^0=0,\varphi^0=0$ and 
	\begin{align}
		&\partial_\eta^2 G^k + F(\eps;\eta)\partial_\eta G^{k-1} + \langle \varphi^k - 4h^3 G^k\rangle = S_1,\label{eq:lm/b1}\\
		&\sin\phi \partial_\eta \varphi^k + F(\eps;\eta) \cos\phi \partial_\phi \varphi^{k-1} + \varphi^k - 4h^3 G^{k} = S_2,\label{eq:lm/b2}
	\end{align}
	with boundary conditions 
	\begin{align}
		&G^k(0)=0,\quad \partial_\eta G^k(B)=0,\label{eq:lm/b1b}\\
		&\varphi^k(0,\phi) = \varphi_b(\phi),\quad \varphi^k(B,\phi) = \varphi^k(B,-\phi),\quad \text{for } \sin \phi>0.\label{eq:lm/b2b}
	\end{align}
	Given bounded functions $(G^{k-1},\varphi^{k-1})$, according to Theorem \ref{thm:apB2} in Appendix \ref{sec:apB}, there exists unique solution $(G^k,\varphi^k)$ for the above system.

	Next we show the boundness of $(G^k,\varphi^k)$ in $L^2([0,B])\times L^2([0,B]\times [-\pi,\pi])$. First, the relation $\partial_\eta G^k = \langle \sin\phi \varphi^k\rangle$ on $[0,B]$ is preserved during the iterations. Suppose $\partial_\eta G^{k-1}=\langle \sin\phi \varphi^{k-1}\rangle$ on $[0,B]$, then comparing \eqref{eq:lm/b1} to $\langle$\eqref{eq:lm/b2}$\rangle$ and using $S_1 = \langle S_2\rangle$ yields 
	\begin{align*}
		\partial_\eta (\partial_\eta G^k - \langle \sin\phi \varphi^k\rangle) &= - F(\eps;\eta) (\partial_\eta G^{k-1} - \langle \cos\phi \partial_\phi \varphi^{k-1}\rangle) \nonumber\\
		&= - F(\eps;\eta)(\partial_\eta G^{k-1} - \langle \sin\phi \varphi^{k-1}\rangle)=0.
	\end{align*}
	This, together with the boundary condition $\partial_\eta G^k(B) = \langle \varphi^k(B,\cdot) \rangle=0$, implies that $\partial_\eta G^k (\eta) = \langle \sin\phi\varphi^k(\eta,\cdot)\rangle$ for any $\eta \in [0,B]$. 

	We multiply \eqref{eq:lm/b1} by $4h^3 G^{k}$ and \eqref{eq:lm/b2} by $\varphi^k$ and integrate to get 
	\begin{align}\label{eq:lm/b-1}
		&-\int_{0}^B 4h^3 G^k \partial_\eta^2 G^k d\eta + \frac12 \int_0^B \int_{-\pi}^\pi \sin\phi \partial_\eta (\varphi^k)^2 d\phi d\eta + \int_0^B \int_{-\pi}^\pi (\varphi^k - 4h^3 G^k)^2 d\phi d\eta \nonumber\\
		&\quad + \int_0^B \int_{-\pi}^\pi F(\eps;\eta) \cos\phi \partial_\phi \varphi^{k-1}\cdot \varphi^k d\phi d\eta - \int_0^B F(\eps;\eta)4h^3 G^k\partial_\eta G^{k-1}d\eta \nonumber\\
		&\qquad = \int_0^B\int_{-\pi}^\pi S_2(\varphi^k - 4h^3 G^k) d\phi d\eta.
	\end{align}
	The inequality \eqref{eq:sp/ineq} from the spectral assumption implies 
	\begin{align}\label{eq:lm/b-2}
		-\int_{0}^B 4h^3 G^k \partial_\eta^2 G^k = \int_0^B 4h^3 |\partial_\eta G^k|^2 d\eta  + \int_0^B\partial_\eta(4h^3) G^k \partial_\eta G^k d\eta \ge 0.
	\end{align}
	Boundary conditions \eqref{eq:lm/b2b} implies 
	\begin{align}\label{eq:lm/b-3}
		\frac12\int_0^B\int_{-\pi}^\pi \sin\phi \partial_\eta (\varphi^k)^2 d\phi d\eta &=  \frac12 \int_{-\pi}^\pi \sin\phi (\varphi^k)^2 d\phi \bigg|_0^B\nonumber\\
		& = -\frac12 \int_{\sin\phi>0} |\sin\phi| (\varphi_b)^2 d\phi + \frac12 \int_{\sin\phi<0} |\sin\phi| (\varphi^k)^2(0,\phi) d\phi.
	\end{align}
	The relation $\partial_\eta G^{k-1}=\langle \sin\phi \varphi^{k-1}\rangle$ implies 
	\begin{align}\label{eq:lm/b-4}
		&\int_0^B \int_{-\pi}^\pi F(\eps;\eta) \cos\phi \partial_\phi \varphi^{k-1}\cdot \varphi^k d\phi d\eta - \int_0^B 4h^3 G^k\partial_\eta G^{k-1}d\eta \nonumber\\
		&\quad = \int_0^B \int_{-\pi}^\pi F(\eps;\eta) \cos\phi \partial_\phi \varphi^{k-1}\cdot \varphi^k d\phi d\eta - \int_0^B \int_{-\pi}^\pi \sin\phi \varphi^{k-1} 4h^3G^k d\phi d\eta \nonumber\\
		&\quad =  \int_0^B \int_{-\pi}^\pi F(\eps;\eta) \cos\phi \partial_\phi \varphi^{k-1}\cdot \varphi^k d\phi d\eta - \int_0^B \int_{-\pi}^\pi \cos\phi \partial_\phi \varphi^{k-1} \cdot 4h^3G^k d\phi d\eta \nonumber\\
		&\quad = \int_0^B \int_{-\pi}^\pi F(\eps;\eta) \cos\phi \partial_\phi \varphi^{k-1} \cdot (\varphi^k-4h^3 G^{k}) d\phi d\eta \nonumber\\
		&\quad \le \frac12 \int_0^B\int_{-\pi}^\pi(\varphi^k-4h^3 G^k)^2 d\phi d\eta+ \frac12 \int_{0}^B\int_{-\pi}^\pi |F(\eps;\eta)|^2 |\partial_\phi \varphi^{k-1}|^2 d\phi d\eta \nonumber \\
		&\quad \le  \frac12 \int_0^B\int_{-\pi}^\pi(\varphi^k-4h^3 G^k)^2 d\phi d\eta+ \frac12 \|F\|_{L^\infty(\mathbb{R}_+)} \|^2 \|\partial_\phi \varphi^{k-1}\|_{L^2([0,B]\times [-\pi,\pi))}^2 \nonumber\\
		&\quad \le   \frac12 \int_0^B\int_{-\pi}^\pi(\varphi^k-4h^3 G^k)^2 d\phi d\eta+ \frac{2}{4-3\delta} \eps   \|\partial_\phi \varphi^{k-1}\|_{L^2([0,B]\times [-\pi,\pi))}^2,
	\end{align}
	where the last inequality is due to \eqref{eq:Fin4}.

	Using the Young's inequality, the last term in \eqref{eq:lm/b-1} can be estimated by 
	\begin{align*}
		\int_0^B\int_{-\pi}^\pi S_2(\varphi^k - 4h^3 G^k) d\phi d\eta \le \frac14 \int_0^B \int_{-\pi}^\pi (\varphi^k - 4h^3 G^k)^2 d\phi d\eta + \|S_2\|_{L^2([0,B]\times [-\pi,\pi))}^2.
	\end{align*}
	Taking the above inequality with \eqref{eq:lm/b-2}-\eqref{eq:lm/b-4} into \eqref{eq:lm/b-1}, we obtain 
	\begin{align}
		&\frac14 \int_0^B \int_{-\pi}^\pi (\varphi^k - 4h^3 G^k)^2 d\phi d\eta + \frac12 \int_{\sin\phi<0} |\sin\phi| (\varphi^k)^2(0,\phi) d\phi \nonumber\\
		&\quad  \le \frac12 \int_{\sin\phi>0} |\sin\phi|(\varphi_b)^2 d\phi+\|S_2\|_{L^2([0,B]\times [-\pi,\pi))}^2 +  \frac{2}{4-3\delta} \eps   \|\partial_\phi \varphi^{k-1}\|_{L^2([0,B]\times [-\pi,\pi))}^2.
	\end{align}
	Hence, assume $\partial_\phi\varphi^{k-1} \in L^2([0,B]))$, one can get $\varphi^k-4h^3 G^k \in L^2([0,B]\times [-\pi,\pi))$. By equation \eqref{eq:lm/b1}, $\partial_\eta^2 G^k \in L^2([0,B])$ and so 
	\begin{align}
		G^k(\eta) = - \int_0^\eta \partial_s G^k(s) ds = \int_0^\eta \int_s^B \partial_\ell^2 G^k(\ell) d\ell ds \le \int_0^B \int_0^B |\partial_\eta^2 G^k| d\eta ds \le B\sqrt{B} \|\partial^2_\eta G\|_{L^2([0,B])},
	\end{align}
	and thus is bounded for any $\eta \in [0,B]$. Moreover, by the boundness of $h$, $\varphi^k=\varphi^k - 4h^3 G^k + 4h^3 G^k \in L^2([0,B]\times [-\pi,\pi))$. Note here since $\varphi^{k-1}(\eta,\cdot) \in C^\infty([-\pi,\pi))$, $\varphi^k(\eta,\cdot)\in C^\infty([-\pi,\pi))$, since the regularity on $\phi$ is preserved by system \eqref{eq:lm/b1}-\eqref{eq:lm/b2}.

	Therefore, we have shown that there exists a solution to the problem \eqref{eq:lm/b1}-\eqref{eq:lm/b2b} satisfying $(G^k,\varphi^k) \in L^2([0,B])\times L^2([0,B]\times [-\pi,\pi))$. Define the operator $\mathcal{L}: L^2([0,B])\times L^2([0,B]\times [-\pi,\pi))\mapsto  L^2([0,B])\times L^2([0,B]\times [-\pi,\pi))$ given by \eqref{eq:lm/b1}-\eqref{eq:lm/b2b} with $(G^k,\varphi^k) = \mathcal{L}((G^{k-1},\varphi^{k-1}))$. We can also get from the above derivation that $G^k \in H^2([0,B])$ and $\varphi^k\in H^1([0,B]\times [-\pi,\pi))$. Hence the operator $\mathcal{L}$ is continous and compact. In order to show $\mathcal{L}$ has a fixed point, we need to show the set 
	\begin{align}\label{eq:lm/b1/set}
		\{(G,\varphi)\in L^2([0,B])\times L^2([0,B]\times [-\pi,\pi)):(G,\varphi)=\sigma \mathcal{L}((G,\varphi))\text{ for some }0\le \sigma \le 1\}
	\end{align}
	is bounded. Then by the Schauder fixed-point theorem, there exists a fixed point of $\mathcal{L}$. To show the above set is bounded, let $(G^,\varphi)=\sigma \mathcal{L}((G,\varphi))$, then $(G,\varphi)$ solves 
	\begin{align}
		&\partial_\eta^2 G+ \sigma F(\eps;\eta)\partial_\eta G + \langle \varphi - 4h^3 G\rangle = S_1,\label{eq:lm/b1/1}\\
		&\sin\phi \partial_\eta \varphi + \sigma F(\eps;\eta) \cos\phi \partial_\phi \varphi + \varphi - 4h^3 G = S_2,\label{eq:lm/b1/2}
	\end{align}
	with boundary conditions 
	\begin{align*}
		&G(0)=0,\quad \partial_\eta G(B) =0, \\
		&\varphi(0,\phi) = \varphi_b(\phi),\quad \varphi(B,\phi) = \varphi(B,-\phi),\quad \text{for } \sin \phi>0.
	\end{align*}
	Note that $\partial_\eta \varphi=\langle \sin\phi \varphi\rangle$ also hold on $[0,B]$. Comparing equation \eqref{eq:lm/b1/1} to \eqref{eq:lm/b1/2} gives 
	\begin{align*}
		\partial_\eta (\partial_\eta G - \langle \sin\phi \varphi)\rangle) = -\sigma F(\eps;\eta) (\partial_\eta G - \langle \cos\phi \partial_\phi \varphi \rangle) = -\sigma  F(\eps;\eta) (\partial_\eta G - \langle \sin\phi \varphi)\rangle),
	\end{align*} 
	which together with the boundary condition $\partial_\eta G(B) - \langle \sin \phi \varphi(B,\cdot)\rangle = 0$, implies 
	\begin{align}\label{eq:lm/b1/=}
		\partial_\eta G=\langle \sin\phi \varphi\rangle.
	\end{align}

	We can estimate the above system similarly as the estimates of system \eqref{eq:lm/b1}-\eqref{eq:lm/b2} and the result is (see the estimate of \eqref{eq:lm/b-1})
	\begin{align*}
		&\frac34 \int_0^B \int_{-\pi}^\pi (\varphi - 4h^3 G)^2 d\phi d\eta + \frac12 \int_{\sin\phi<0} |\sin\phi| \varphi^2(0,\phi) d\phi \nonumber\\
		&\quad  \le \frac12 \int_{\sin\phi>0} |\sin\phi|(\varphi_b)^2 d\phi+\|S_2\|_{L^2([0,B]\times [-\pi,\pi))}^2 +  \int_0^B\int_{-\pi}^\pi F(\eps;\eta) \cos\phi \partial_\phi \varphi (\varphi-4h^3 G)d\phi d\eta.
	\end{align*}
	Integration by parts in $\phi$ on the last term gives 
	\begin{align*}
		&\int_0^B\int_{-\pi}^\pi F(\eps;\eta) \cos\phi \partial_\phi \varphi (\varphi-4h^3 G)d\phi d\eta \nonumber\\
		&\quad = \int_0^B\int_{-\pi}^\pi F(\eps;\eta) \sin\phi \varphi (\varphi - 4h^3 G) d\phi d\eta \nonumber\\
		&\quad = \frac12 \int_0^B \int_{-\pi}^\pi F(\eps;\eta) \sin\phi (\varphi-4h^3 G)^2 d\eta \nonumber\\
		&\quad \le \frac12 \|F\|_{L^\infty(\mathbb{R}_+)}  \int_0^B \int_{-\pi}^\pi (\varphi - 4h^3 G)^2 d\phi d\eta \nonumber\\
		&\quad \le \frac{2}{4-3\delta} \eps \int_0^B \int_{-\pi}^\pi (\varphi - 4h^3 G)^2 d\phi d\eta.
	\end{align*}
	Taking it into the previous inequality leads to  
	\begin{align*}
		&\left(\frac34 -  \frac{2}{4-3\delta} \eps\right) \int_0^B \int_{-\pi}^\pi (\varphi - 4h^3 G)^2 d\phi d\eta + \frac12 \int_{\sin\phi<0} |\sin\phi| \varphi^2(0,\phi) d\phi \nonumber\\
		&\quad  \le \frac12 \int_{\sin\phi>0} |\sin\phi|(\varphi_b)^2 d\phi+\|S_2\|_{L^2([0,B]\times [-\pi,\pi))}^2.
	\end{align*}
	For $\eps$ small, for example $\eps \le (4-3\delta)/8$, $\left(\frac34 -  \frac{2}{4-3\delta} \eps\right)\ge \frac14$, and the above estimate becomes 
	\begin{align}\label{eq:lm/b1/1/est}
		&\frac14\int_0^B \int_{-\pi}^\pi (\varphi - 4h^3 G)^2 d\phi d\eta + \frac12 \int_{\sin\phi<0} |\sin\phi| \varphi^2(0,\phi) d\phi \nonumber\\
		&\quad  \le \frac12 \int_{\sin\phi>0} |\sin\phi|(\varphi_b)^2 d\phi+\|S_2\|_{L^2([0,B]\times [-\pi,\pi))}^2.
	\end{align}
	Moreover, by equation \eqref{eq:lm/b1/1}, 
	\begin{align*}
		\|\partial_\eta^2 G\|_{L^2([0,B])} &\le \sigma \|F\partial_\eta G\|_{L^2([0,B])} + \|\langle (\varphi-4h^3 G)\rangle \|_{L^2([0,B])} + \|S_1\|_{L^2([0,B])} \nonumber\\
		&\le \sigma \|F\|_{L^\infty(\mathbb{R}_+)} \|\langle \sin\phi (\varphi - 4h^3 G)\rangle \|_{L^2[0,B]} +  \|\langle (\varphi-4h^3 G)\rangle \|_{L^2([0,B])} + \|S_1\|_{L^2([0,B])} \nonumber\\
		&\le \left(\frac{4\sigma}{4-3\delta}\eps +1 \right)\sqrt{2\pi}\|(\varphi-4h^3 G)\|_{L^2([0,B]\times [-\pi,\pi))} +  \|S_1\|_{L^2([0,B])} \nonumber\\
		&\le C \left(\int_{\sin\phi>0} |\sin\phi|(\varphi_b)^2 d\phi\right)^{\frac12}+C\|S_2\|_{L^2([0,B]\times [-\pi,\pi))} + \|S_1\|_{L^2([0,B])}.
	\end{align*}
	Hence $\|G\|_{H^2([0,B])}$ is bounded. Moreover $\varphi=(\varphi-4h^3G)+4h^3 G \in L^2([0,B]\times [-\pi,\pi))$ is bounded. Therefore, the set \eqref{eq:lm/b1/1} is bounded. By the Schauder fixed point theorem, there exists a fixed point $(G^\eps_B,\varphi^\eps_B) \in L^2([0,B])\times L^2([0,B]\times [-\pi,\pi))$ such that $(G^\eps_B,\varphi^\eps_B) = \mathcal{L}((G^\eps_B,\varphi^\eps_B))$. By the definition of $\mathcal{L}$, $(G^\eps_B,\varphi^\eps_B)$ solves system \eqref{eq:l1B}-\eqref{eq:l2B} with boundary conditions \eqref{eq:l1Bb}-\eqref{eq:l2Bb}.

	Taking $\sigma=1$ in the derivation of the estimate \eqref{eq:lm/b1/1/est} implies that $(G^\eps_B,\varphi_B^\eps)$ satisfy the estimate \eqref{eq:l1B/est} and finishes the proof.
\end{proof}

\subsection{Weighted estimate}
Next we derive the uniform weighted estimate on the solutions of system \eqref{eq:l1B}-\eqref{eq:l2B}. We have the following lemma.
\begin{lemma}
Assuming the hypothesis of Lemma \ref{lem.lB} hold and $(G_B^\eps,\varphi_B^\eps)$ is solution obtained from Lemma \ref{lem.lB}. Then for $0\le \beta'\le \min\{1/2,\beta,K\}$, the following estimate holds
	\begin{align}\label{eq:lm/11/est}
		&\frac18 \int_0^B \int_{-\pi}^\pi  e^{2\beta'\eta} (\varphi_B^\eps - 4h^3 G_B^\eps)^2 d\phi d\eta  + \frac12 \int_{\sin\phi<0} |\sin\phi| (\varphi_B^\eps)^2(0,\cdot) d\phi \nonumber\\
		&\quad  \le  \frac12 \int_{\sin\phi>0} \sin\phi \varphi_b^2 d\phi + \|e^{\beta'\eta} S_2\|_{L^2([0,B]\times [-\pi,\pi))}^2,
	\end{align}
	and 
	\begin{align}\label{eq:lm/11/est1}
		|G_B^\eps(\eta)| \le C
	\end{align}
	for some constant $C>0$ dependent on $\varphi_b, S_1, S_2$ but independent of $\eps$. Moreover, the solution satisfies 
	\begin{align}\label{eq:lm/11/est2}
		\partial_\eta G_B^\eps(\eta) = \langle \sin\phi \varphi^\eps_B(\eta,\cdot)\rangle,
	\end{align}
	for any $\eta\in [0,B]$.
\end{lemma}
\begin{proof}
	First, we have $\partial_\eta G_B^\eps = \langle \sin\phi \varphi_B^\eps \rangle$ which is due to the same reason as \eqref{eq:lm/b1/=}. We multiply \eqref{eq:l1B} by $e^{2\beta'\eta }4h G_B^\eps$ and \eqref{eq:l2B} by $\varphi_B^\eps$ and integrate to obtain
	\begin{align}\label{eq:lB/1}
		&-\int_{0}^B e^{2\beta'\eta} 4h^3 G^\eps_B \partial_\eta^2 G^\eps_B d\eta + \frac12 \int_0^B \int_{-\pi}^\pi e^{2\beta'\eta}\sin\phi \partial_\eta (\varphi^\eps_B)^2 d\phi d\eta + \int_0^B \int_{-\pi}^\pi e^{2\beta' \eta} (\varphi^\eps_B - 4h^3 G^\eps_B)^2 d\phi d\eta \nonumber\\
		&\quad + \frac12 \int_0^B \int_{-\pi}^\pi e^{2\beta'\eta} F(\eps;\eta) \cos\phi \partial_\phi (\varphi^\eps_B)^2 d\phi d\eta - \int_0^B e^{2\beta'\eta} F(\eps;\eta)4h^3 G^\eps_B\partial_\eta G^\eps_B d\eta \nonumber\\
		&\qquad = \int_0^B\int_{-\pi}^\pi e^{2\beta'\eta} S_2(\varphi^\eps_B - 4h^3 G^\eps_B) d\phi d\eta.
	\end{align}
	Using integration by parts and the inequality \eqref{eq:sp/ineq} from the spectral assumption \ref{asA}, we can get 
	\begin{align*}
		&-\int_{0}^B e^{2\beta'\eta} 4h^3 G^\eps_B \partial_\eta^2 G^\eps_B d\eta \nonumber \\
		&\quad = \int_0^B e^{2\beta'\eta} 4h^3|\partial_\eta G_B^\eps|^2 d\eta + \int_0^B e^{2\beta'\eta} \partial_\eta (4h^3) G_B^\eps \partial_\eta G_B^\eps d\eta + 2\beta' \int_0^B e^{2\beta'\eta} 4h^3 G_B^\eps \partial_\eta G_B^\eps d\eta \nonumber\\
		&\quad \ge 2\beta'\int_0^B e^{\beta'\eta} 4h^3 G_B^\eps \partial_\eta G_B^\eps d\eta.
	\end{align*}
	Using integration by parts and the boundary conditions \eqref{eq:l2Bb}, we have 
	\begin{align*}
		&\frac12 \int_0^B \int_{-\pi}^\pi e^{2\beta'\eta}\sin\phi \partial_\eta (\varphi^\eps_B)^2 d\phi d\eta \\
		&\quad = - \beta'\int_0^B\int_{-\pi}^\pi e^{2\beta'\eta} \sin\phi (\varphi_B^\eps)^2 d\beta d\eta + \frac12 e^{2\beta'\eta} \int_{-\pi}^\pi \sin\phi (\varphi_B^\eps)^2(\eta,\cdot) d\phi\bigg|_0^B \nonumber\\
		&\quad = - \beta'\int_0^B\int_{-\pi}^\pi e^{2\beta'\eta} \sin\phi (\varphi_B^\eps)^2 d\beta d\eta - \frac12 \int_{\sin\phi>0} \sin\phi \varphi_b^2 d\phi + \frac12 \int_{\sin\phi<0} |\sin\phi| (\varphi_B^\eps)^2(0,\cdot) d\phi.
	\end{align*}
	The relation $\partial_\eta G_B^\eps = \langle \sin\phi \varphi_B^\eps \rangle$ gives 
	\begin{align*}
		&\frac12 \int_0^B \int_{-\pi}^\pi e^{2\beta'\eta} F(\eps;\eta) \cos\phi \partial_\phi (\varphi^\eps_B)^2 d\phi d\eta - \int_0^B e^{2\beta'\eta} F(\eps;\eta)4h^3 G^\eps_B\partial_\eta G^\eps_B d\eta \nonumber\\
		&\quad = \frac12 \int_0^B \int_{-\pi}^\pi e^{2\beta'\eta} F(\eps;\eta)\sin\phi (\varphi_B^\eps)^2 d\phi d\eta - \int_0^B\int_{-\pi}^\pi e^{2\beta'\eta} F(\eps;\eta) 4h^3 G^\eps_B \sin\phi \varphi_B^\eps d\phi d\eta \nonumber\\
		&\quad = \frac12 \int_0^B \int_{-\pi}^\pi e^{2\beta' \eta} F(\eps;\eta) \sin\phi (\varphi_B^\eps - G_B^\eps)^2 d\phi d\eta,
	\end{align*}
	and also 
	\begin{align*}
		&- \beta'\int_0^B\int_{-\pi}^\pi e^{2\beta'\eta} \sin\phi (\varphi_B^\eps)^2 d\beta d\eta +  2\beta'\int_0^B e^{2\beta'\eta} 4h^3 G_B^\eps \partial_\eta G_B^\eps d\eta \nonumber\\
		&\quad = - \beta' \int_0^B \int_{-\pi}^\pi \left(e^{2\beta'\eta} \sin\phi(\varphi_B^\eps)^2 - 2 e^{2\beta'\eta} 4h^3 G_B^\eps \sin\phi \varphi_B^\eps\right) d\phi d\eta \nonumber\\
		&\quad = - \beta' \int_0^B\int_{-\pi}^\pi e^{2\beta'\eta} \sin\phi (\varphi_B^\eps - 4h^3 G_B^\eps)^2 d\phi d\eta.
	\end{align*}
	The right-hand side of equation \eqref{eq:lB/1} can be estimated by 
	\begin{align*}
		\int_0^B\int_{-\pi}^\pi e^{2\beta'\eta} S_2(\varphi^\eps_B - 4h^3 G^\eps_B) d\phi d\eta \le \frac14 \int_0^B \int_{-\pi}^\pi e^{2\beta'\eta} (\varphi_B^\eps - 4h^3 G_B^\eps)^2 d\phi d\eta + \|e^{\beta'\eta} S_2\|_{L^2([0,B]\times [-\pi,\pi))}^2.
	\end{align*}
	Taking the above inequalities into \eqref{eq:lB/1}, we obtain
	\begin{align*}
		&\frac34\int_0^B \int_{-\pi}^\pi  e^{2\beta'\eta} (\varphi_B^\eps - 4h^3 G_B^\eps)^2 d\phi d\eta +\frac12 \int_0^B\int_{-\pi}^\pi (F(\eps;\eta)-2\beta') \sin\phi e^{2\beta'\eta} (\varphi_B^\eps - 4h^3 G_B^\eps)^2 d\phi d\eta \nonumber\\
		& \quad + \frac12 \int_{\sin\phi<0} |\sin\phi| (\varphi_B^\eps)^2(0,\cdot) d\phi \le  \frac12 \int_{\sin\phi>0} \sin\phi \varphi_b^2 d\phi + \|e^{\beta'\eta} S_2\|_{L^2([0,B]\times [-\pi,\pi))}^2.
	\end{align*}
	For $\eps$ sufficiently small, for example $\eps \le (4-3\delta)/16$, inequality \eqref{eq:Fin4} implies for any $\eta\in \mathbb{R}_+$,
	\begin{align}\label{eq:F/beta'}
		|F(\eps;\eta)-2\beta'| \le 2\beta' + \frac{4}{4-3\delta} \eps \le 1 + \frac14 = \frac54.
	\end{align}
	With this, the previous inequality implies \eqref{eq:lm/11/est}.

	Moreover, we can get from equation \eqref{eq:l1B} that 
	\begin{align*}
		&\int_0^B e^{2\beta'\eta} |\partial_\eta^2 G_B^\eps|^2 d\eta \nonumber
		\\
		&\quad \le \int_0^B e^{2\beta' \eta} |F(\eps;\eta) \partial_\eta G_B^\eps|^2d\eta + \int_0^B e^{2\beta' \eta} |\langle \varphi_B^\eps - 4h^3 G_B^\eps\rangle |^2  d\eta + \int_0^B e^{2\beta'\eta} |S_1|^2 d\eta \nonumber\\
		&\quad \le \int_0^B |F(\eps;\eta)|^2 e^{2\beta'\eta} \langle \sin\phi (\varphi_B^\eps - 4h^3 G_B^\eps)\rangle d\eta + \int_0^B \left(\int_{-\pi}^\pi e^{2\beta'\eta} (\varphi_B^\eps-4h^3 G_B^\eps )^2 d\phi \cdot \int_{-\pi}^\pi d\phi\right) d\eta \nonumber\\&\qquad + \int_0^B e^{2\beta'\eta}|S_1|^2 d\eta \nonumber\\
		&\quad \le \|F\|_{L^\infty(\mathbb{R}_+)}^2 2\pi \int_0^B \int_{-\pi}^\pi e^{2\beta'\eta} (\varphi_B^\eps - 4h^3 G_B^\eps)^2 d\phi d\eta + 2\pi \int_0^B \int_{-\pi}^\pi e^{2\beta'\eta} (\varphi_B^\eps - 4h^3 G_B^\eps)^2 d\phi d\eta \nonumber\\
		&\qquad + \int_0^B e^{2\beta'\eta}|S_1|^2 d\eta \nonumber\\
		&\quad \le 2\pi\left(\frac{4}{4-3\delta}\eps+1\right) \int_0^B \int_{-\pi}^\pi e^{2\beta'\eta}(\varphi_B^\eps - 4h^3 G_B^\eps)^2 d\phi d\eta +\int_0^B e^{2\beta'\eta}|S_1|^2 d\eta.
	\end{align*}
	For $\eps$ sufficiently small, for example $\eps \le (4-3\delta)/16$, the above inequality becomes 
	\begin{align*}
		&\int_0^B e^{2\beta'\eta} |\partial_\eta^2 G_B^\eps|^2 d\eta \nonumber\\
		&\quad \le \frac{5}{2}\pi  \int_0^B \int_{-\pi}^\pi e^{2\beta'\eta}(\varphi_B^\eps - 4h^3 G_B^\eps)^2 d\phi d\eta +\int_0^B e^{2\beta'\eta}|S_1|^2 d\eta \nonumber\\
		&\quad \le  10\pi \int_{\sin\phi>0} \sin\phi \varphi_b^2 d\phi + 20\pi \|e^{\beta'\eta} S_2\|_{L^2([0,B]\times [-\pi,\pi))}^2 + \|e^{\beta'\eta} S_1\|_{L^2([0,B])}^2,
	\end{align*}
	and hence is uniformly bounded in $B$. Let $C_1:= 10\pi \int_{\sin\phi>0} \sin\phi \varphi_b^2 d\phi + 20\pi \|e^{\beta'\eta} S_2\|_{L^2([0,B]\times [-\pi,\pi))}^2 + \|e^{\beta'\eta} S_1\|_{L^2([0,B])}^2$. Using the bounded condition $\partial_\eta G_B^\eps (B)=0$, we have for any $\eta\in[0,B]$,
	\begin{align*}
		|\partial_\eta G_B^\eps(\eta)| &= -\left|\int_\eta^B \partial_s^2 G_B^\eps(s) ds\right| = \left|\int_\eta^B e^{-\beta' s} e^{\beta' s}  \partial_s^2 G_B^\eps(s) ds\right| \nonumber\\
		&\le \left(\int_\eta^B e^{-2\beta's} ds\right)^{\frac12} \int_0^B e^{2\beta's} |\partial_s^2 G_B^\eps(s)|^2 ds \nonumber\\
		&\le C_1 \frac{\sqrt{e^{-2\beta'\eta}-e^{-2\eta'B }}}{\sqrt{2\beta'}} \le C_1 (2\beta')^{-\frac12} e^{-\beta'\eta}.
	\end{align*}
	Using the boundary condition $G_B^\eps(0)=0$, we have 
	\begin{align*}
		|G_B^\eps(\eta)| &= \left|\int_0^\eta \partial_s G_B^\eps(s)ds\right| \le \int_0^\eta |\partial_s G_B^\eps(s)| ds \le C_1 (2\beta')^{-\frac12}\int_0^\eta e^{- \beta's} ds \nonumber\\
		&\le \sqrt{2}C_1 (\beta')^{-\frac32} (1-e^{-\beta'\eta}) \le \sqrt{2}C_1 (\beta')^{-\frac32}.
	\end{align*}

	\subsection{Existence on the half line} 
	Next we pass to the limit $B\to\infty$ in \eqref{eq:l1B}-\eqref{eq:l2B} to show the existence for system \eqref{eq:l1}-\eqref{eq:l2} on the half line. 
	\begin{proof}[Proof of Theorem \ref{thm.linearmilne}]
		By the uniform estimate \eqref{eq:lm/11/est1}, $G_B^\eps \in L^\infty_{\rm loc}([0,\infty))$. Hence there exists a subsequence such that 
		\begin{align}\label{eq:hl/1}
			&G_B^\eps \rightharpoonup^* G^\eps,\quad \text{weakly star in }L_{\rm loc}^\infty([0,\infty)),\\
			&G_B^\eps \rightharpoonup G^\eps,\quad \text{weakly  in }L_{\rm loc}^2([0,\infty)).\nonumber
		\end{align}
		Moreover, by the estimate \eqref{eq:lm/11/est}, $\varphi_B^\eps - 4h^3 G_B^\eps$ is uniformly bounded in $L^2_{\rm loc}([0,\infty);L^2([-\pi,\pi)))$, hence 
		\begin{align*}
			\varphi_B^\eps - 4h^3 G_B^\eps \rightharpoonup \varphi^\eps - 4h^3 G^\eps,\quad \text{weakly in }L^2_{\rm loc}([0,\infty);L^2([-\pi,\pi))),
		\end{align*}
		and so
		\begin{align}\label{eq:hl/2}
			\varphi_B^\eps \rightharpoonup \varphi^\eps ,\quad \text{weakly in }L^2_{\rm loc}([0,\infty);L^2([-\pi,\pi))).
		\end{align}
		With \eqref{eq:hl/1} and \eqref{eq:hl/2}, we can pass to the limit in \eqref{eq:l1B}-\eqref{eq:l2B} with boundary conditions \eqref{eq:l1Bb}-\eqref{eq:l2Bb} and obtain that $(G^\eps,\varphi^\eps)$ solve system \eqref{eq:l1}-\eqref{eq:l2} in the sense of distributions and satisfy the boundary conditions \eqref{eq:l1b}-\eqref{eq:l2b}. Moreover, we can pass to the limit $B\to\infty$ in \eqref{eq:lm/11/est2} and obtain 
		\begin{align}\label{eq:hl/=}
			\partial_\eta G^\eps = \langle \sin\phi \varphi^\eps(\eta,\cdot)\rangle
		\end{align}
		for $\eta \in \mathbb{R}_+$ almost everywhere. By the semi-lower continuity of the $L^2$ norm, we can take $\liminf$ in \eqref{eq:lm/11/est} and use 
		\begin{align*}
			\liminf_{B\to\infty} \int_0^B \int_{-\pi}^\pi e^{2\beta'\eta} (\varphi_B^\eps-4h^3G_B^\eps)^2 d\phi d\eta \ge \int_0^\infty\int_{-\pi}^\pi e^{2\beta'\eta} (\varphi^\eps-4h^3 G^\eps)^2 d\phi d\eta
		\end{align*}
		and obtain \eqref{eq:linearmilne/est}.
	\end{proof}

	\subsection{Decay properties of the solutions}
	We can use the above derived uniform estimates to show that solutions to system \eqref{eq:l1}-\eqref{eq:l2} decay to constants as $\eta\to\infty$. We set 
	\begin{align*}
		C_2:= 4\int_{\sin\phi>0} \sin\phi \varphi_b^2 d\phi + 8\|e^{\beta'\eta} S_2\|_{L^2([0,\infty)\times [-\pi,\pi))}^2.
	\end{align*}
	Then 
	\begin{align}
		\int_0^\infty \int_{-\pi}^\pi e^{2\beta'\eta} (\varphi^\eps-4h^3 G^\eps)^2d\phi d\eta \le C_2.
	\end{align}
	By the relation \eqref{eq:hl/=}, we get 
	\begin{align*}
		\int_{0}^\infty e^{2\beta'\eta}|\partial_\eta G^\eps|^2 d\eta &= \int_0^\infty e^{2\beta'\eta}\left|\int_{-\pi}^\pi \sin\phi (\varphi^\eps) d\phi \right|^2d\eta = \int_0^\infty e^{2\beta'\eta} \left|\int_{-\pi}^\pi \sin\phi (\varphi^\eps -4h^3 G^\eps) d\phi \right|^2d\eta \nonumber\\
		& \le \int_0^\infty e^{2\beta'\eta} \left(\int_{-\pi}^\pi\sin^2\phi  d\phi\right) \left(\int_{-\pi}^\pi (\varphi^\eps-4h^3 G^\eps)^2 d\phi\right)d\eta \nonumber\\
		& = \pi \int_0^\infty \int_{-\pi}^\pi e^{2\beta'\eta} (\varphi^\eps-4h^3 G^\eps)^2d\phi \le \pi C_2.
	\end{align*}
	Therefore, $\lim_{\eta\to\infty} G^\eps(\eta)$ exists and is finite. To show this, we can take $M_2>M_1>0$ to be some large numbers and 
	\begin{align*}
		|G^\eps(M_1)-G^\eps(M_2)| &= \left|\int_{M_1}^{M_2} \partial_\eta G^\eps d\eta\right| \le \int_{M_1}^{M_2} e^{-\beta'\eta} e^{\beta'\eta} |\partial_\eta G^\eps|d\eta \nonumber\\
		&\le \left(\int_{M_1}^{M_2} e^{-2\beta'\eta}d\eta\right)^\frac12 \left(\int_{M_1}^{M_2} e^{2\beta'\eta}|\partial_\eta G^\eps|^2 d\eta \right)^{\frac12},\nonumber\\
		& = \sqrt{(e^{-2\beta' M_1}-e^{-2\beta'M_2})}\pi\sqrt{ C_2/(2\beta')}.
	\end{align*}
	Since for $M_1,M_2$ sufficiently large, $e^{-2\beta' M_1}-e^{-2\beta'M_2}$ is sufficiently small. $\lim_{\eta\to\infty} G^\eps(\eta)$ exists. Let $G^\eps_\infty=\lim_{\eta\to\infty}G^\eps(\eta)$. Then we take $M_1=\eta$, $M_2=\infty$ in the above inequality and obtain 
	\begin{align*}
		|G^\eps(\eta)-G^\eps_\infty| \le \pi \sqrt{C_2/(2\beta')} e^{-\beta'\eta}.
	\end{align*}
	Moreover, taking $M_1=0$, $M_2=\infty$ gives 
	\begin{align*}
		|G^\eps_\infty| \le \pi\sqrt{C_2/(2\beta')}.
	\end{align*}
	Due to the exponentially convergence of $h$, we have 
	\begin{align}
		|4h^3(\eta)G^\eps(\eta)-4h_\infty^3 G^\eps_\infty| &\le 4 h^3(\eta)|G^\eps(\eta)-G^\eps_\infty| + 4|h^3(\eta)-h_\infty^3||G^\eps_\infty| \le C e^{-\alpha \eta} + C e^{-\beta'\eta} \nonumber\\
		&\le 2C e^{-\beta'\eta},
	\end{align}
	taking $1>\alpha>\beta'$.

	To show the exponential convergence of $\varphi^\eps$, we use the formulas \eqref{eq:varphi--1}, \eqref{eq:varphi--2} and \eqref{eq:varphi--3}. 	
	In the case $\sin\phi>0$, we have 
	\begin{align*}
		|\varphi^\eps(\eta,\phi) - 4h_{\infty}^3 G_{\infty}^\eps| &= \left| \int_0^\eta \frac{4h^3(\xi) G^\eps(\xi)+S_2(\xi,\phi'(\phi,\eta,\xi))}{\sin \phi'(\phi,\eta,\xi)} e^{-\int_\xi^\eta \frac{1}{\sin \phi'(\phi,\eta,\rho)}d\rho} d\xi - 4h^3_\infty G_\infty^\eps\right| \nonumber\\
		&\le\int_0^\eta \frac{|4h^3(\xi)G^\eps(\xi)-4h_\infty^3 G^\eps_\infty|}{\sin \phi'(\phi,\eta,\xi)} e^{-\int_\xi^\eta \frac{1}{\sin \phi'(\phi,\eta,\rho)}d\rho} d\xi \nonumber\\
		&\quad + \int_0^\eta \frac{|S_2(\xi,\phi'(\phi,\eta,\xi))|}{\sin \phi'(\phi,\eta,\xi)} e^{-\int_\xi^\eta \frac{1}{\sin \phi'(\phi,\eta,\rho)}d\rho} d\xi \nonumber\\
		&\quad  +\left|1- \int_0^\eta \frac{1}{\sin\phi'(\phi,\eta,\xi)} e^{-\int_\xi^\eta \frac{1}{\sin \phi'(\phi,\eta,\rho)}d\rho} d\xi\right| |4h_\infty^3 G_\infty^\eps| \nonumber\\
		&\quad \le 2C \int_0^\eta \frac{e^{-\beta'\xi}}{\sin \phi'(\phi,\eta,\xi)} e^{-\int_\xi^\eta \frac{1}{\sin \phi'(\phi,\eta,\rho)}d\rho} d\xi + C e^{-\int_0^\eta \frac{1}{\sin \phi'(\phi,\eta,\rho)}d\rho} \nonumber \nonumber\\
		&\quad \le 2C \int_0^\eta \frac{1}{1-\beta' \sin \phi'(\phi,\eta,\xi)} d e^{-\int_\xi^\eta \frac{1}{\sin \phi'(\phi,\eta,\rho)}-\frac12\beta'\xi} + C e^{-\int_0^\eta \frac{1}{\sin \phi'(\phi,\eta,\rho)}d\rho} \nonumber\\
		&\quad \le \frac{2C}{1-\beta'}(e^{-\beta'\eta}-e^{-\int_\xi^\eta \frac{1}{\sin \phi'(\phi,\eta,\rho)}}) + C e^{-\int_0^\eta \frac{1}{\sin \phi'(\phi,\eta,\rho)}d\rho} \nonumber\\
		&\quad \le \frac{3C}{1-\beta'} e^{-\beta'\eta},
	\end{align*}
	where we use $\beta'<1/2$ and $\sin\phi' \le 1$ in the above inequality. One can also check for $\sin\phi<0$, there exists positive constants $C>0$, 
	\[
	|\varphi^\eps(\eta,\phi) - 4h_{\infty}^3 G_{\infty}^\eps| \le C e^{-\beta'\eta},
	\]
	which follows a similar estimate as the case for $\sin\phi>0$ above and the details are similar to the proof of \eqref{eq:non/conv}.

	\subsection{Uniqueness}
	In order to show the uniqueness, let $(G_1^\eps,\varphi_1^\eps)$ and $(G_2^\eps,\varphi_2^\eps)$ be two solutions of the problem \eqref{eq:l1}-\eqref{eq:l2b}.  Let $g:=G_1^\eps-G_2^\eps$ and $\psi:=\varphi_1^\eps-\varphi_2^\eps$, then they satisfy
	\begin{align}
		&\partial_\eta^2 g + F(\eps;\eta)\partial_\eta g + \langle \psi - 4h^3 g\rangle = 0,\label{eq:hl/01}\\
		&\sin\phi \partial_\eta \psi + F(\eps;\eta)\cos\phi \partial_\phi \psi + \psi - 4h^3 g =0,\label{eq:hl/02}
	\end{align}
	with boundary conditions 
	\begin{align*}
		g(0) = 0,\quad \psi(0,\phi) = 0,\quad \text{for any }\sin\phi>0.
	\end{align*}
	Moreover, as $\eta\to\infty$, $G_1^\eps\to G_{1,\infty}$, $G_2^\eps \to G_{2,\infty}$ and $\psi \to 4h_\infty^3 (G_{1,\infty}-G_{2,\infty})$, where $G_{1,\infty}$, $G_{2,\infty}$ are constants.

	We multiply the first equation by $4h^3g$ and the second by $\psi$, and integrate over $[0,\infty)\times [-\pi,\pi)$ to get 
	\begin{align}\label{eq:uq1}
		&\int_0^\infty 4h^3 |\partial_\eta g|^2 d\eta + \int_0^\infty \partial_\eta(4h^3)g \partial_\eta g d\eta - \int_0^\infty F(\eps;\eta) 4h^3g \partial_\eta g d\eta + \frac12 \int_0^\infty \int_{-\pi}^\pi F(\eps;\eta)\cos\phi \partial_\phi \psi^2 d\phi d\eta\nonumber\\
		&\quad + \int_{-\pi}^0 |\sin\phi| \psi^2(0,\cdot) d\phi + \int_0^\infty \int_{-\pi}^\pi (\psi-4h^3g)^2d\phi d\eta =0 .
	\end{align}
	From the relation \eqref{eq:hl/=}, we have $\partial_\eta g = \langle \sin\phi \varphi \rangle$. Thus 
	\begin{align*}
		&-\int_0^\infty F(\eps;\eta) 4h^3 g \partial_\eta g d\eta + \frac12\int_0^\infty \int_{-\pi}^\pi F(\eps;\eta) \cos\phi \partial_\phi \psi^2 d\phi d\eta \nonumber\\
		&\quad = -\int_0^\infty \int_{-\pi}^\pi F(\eps;\eta) 4h^3 g \sin\phi \varphi d\phi d\eta + \frac12 \int_0^\infty \int_{-\pi}^\pi F(\eps;\eta) \sin\phi \psi^2 d\phi d\eta \nonumber\\
		&\quad = -\frac12 \int_0^\infty\int_{-\pi}^\pi F(\eps;\eta) \sin\phi (\psi-4h^3g)^2 d\phi d\eta.
	\end{align*}
	The spectral assumption implies 
	\begin{align*}
		\int_0^\infty 4h^3 |\partial_\eta g|^2 d\eta + \int_0^\infty \partial_\eta(4h^3)g \partial_\eta g d\eta  \ge 0.
	\end{align*}
	Therefore, we get from \eqref{eq:uq1} that 
	\begin{align*}
		\int_{-\pi}^0 |\sin\phi| \psi^2(0,\cdot) d\phi +\int_0^\infty \int_{-\pi}^\pi \left(1+\frac12 F(\eps;\eta)\right) (\psi-4h^3g)^2d\phi d\eta=0.
	\end{align*}
	For $\eta$ sufficiently small, $1+\tfrac12 F(\eps;\eta)>0$ for any $\eta\in \mathbb{R}_+$, hence the above equation implies $\psi=4h^3g$ almost everywhere on $\mathbb{R}_+\times [-\pi,\pi)$ and $\psi(0,\phi)=0$ for almost every $\phi \in [-\pi,\pi)$. Hence $\partial_\eta g = \langle \sin\phi \psi\rangle = \langle \sin\phi 4h^3g \rangle=0$, which with $g(0)=0$ implies $g=0$ on $\mathbb{R}_+$. We also have $\psi=4h^3g=0$. Hence system \eqref{eq:hl/01}-\eqref{eq:hl/02} only has trival solution and the solution to the problem \eqref{eq:l1}-\eqref{eq:l2b}.


	
\end{proof}
\begin{rem}
	Compared to the flat boundary case without geometric corrections \cite{Bounadrylayer2019GHM2}, the effect of the geometric corrections could be seen from \eqref{eq:F/beta'}. for $\eps$ not sufficiently small, one cannot get a $L^2$ estimate on $\|e^{\beta'\eta}(\varphi^\eps - 4h^3 G^\eps)\|_{L^2(\mathbb{R}_+\times[-\pi,\pi])}$. However, for $\eps$ sufficiently small, the geometric correction contributes to a small error of this $L^2$ norm so does not affect the boundness of this norm, which is crucial for proving Theorem \ref{thm.linearmilne}.
\end{rem}

\section{Convergence of the diffusive limit}\label{sec5}
\subsection{Errors of the approximation}
Next we calculate the approximation errors.
	Using systems \eqref{eq:interior/1}-\eqref{eq:interior/3} and \eqref{eq:g0}-\eqref{eq:g2}, we can calculate 
	\begin{align}\label{eq:R1p}
		&\mathcal{R}_1\left(\sum_{k=0}^N \eps^k(T_k^\eps+\bar{T}_k^\eps),\sum_{k=0}^N \eps^k(\psi_k^\eps + \bar{\psi}_k^\eps)\right)\nonumber\\
		&\quad= \varepsilon^2 \Delta \sum_{k=0}^N \varepsilon^k (T_k^\eps + \bar{T}_k^\eps) + \left\langle \sum_{k=0}^N \varepsilon^k (\psi_k^\eps + \bar{\psi}_k^\eps) + \left(\sum_{k=0}^N \varepsilon^k(T_k^\eps+\bar{T}_k^\eps)\right)^4 \right\rangle \nonumber\\
		&\quad =\varepsilon^2 \Delta \sum_{k=0}^N \varepsilon^k T_k^\eps + \sum_{k=0}^N\varepsilon^k \langle \psi_k^\eps \rangle + \partial_\eta^2 \sum_{k=0}^N \varepsilon^k \bar{T}_k^\eps +  F(\eps;\eta)\partial_\eta \sum_{k=0}^N\varepsilon^k \bar{T}_k^\eps - \frac{\chi_0\eps^2}{(1-\eps\eta)^2}\partial_\theta^2  \sum_{k=0}^N\varepsilon^k \bar{T}_k^\eps \nonumber\\
		&\qquad+ \left\langle \sum_{k=0}^N \varepsilon^k  \bar{\psi}_k^\eps - \left(\sum_{k=0}^N \varepsilon^k(T_k^\eps+\bar{T}_k^\eps)\right)^4 \right\rangle  \nonumber\\
		&\quad=\mathcal{R}_1\left(\sum_{k=0}^N \eps^k T_k^\eps,\sum_{k=0}^N \eps^k \psi_k^\eps\right) + \left\langle \left(\sum_{k=0}^N \varepsilon^k T_k^\eps\right)^4\right\rangle \nonumber\\
		&\qquad + \sum_{k=0}^N \eps^k(\partial_\eta^2\bar{T}_k^\eps - \frac{\eps}{1-\eps\eta}\partial_\eta \bar{T}_{k}^\eps + \langle \bar{\psi}_k^\eps - \mathcal{C}(T^\eps+\bar{T}^\eps,k)\rangle + \frac{1}{(1-\eps\eta)^2}\partial_\theta^2 \bar{T}_{k-2}^\eps)  \nonumber\\
		&\qquad + \frac{\eps^{N+1}}{(1-\eps\eta)^2}\partial_\theta^2 \bar{T}_{N-1}^\eps + \frac{\eps^{N+2}}{(1-\eps\eta)^2} \partial_\theta^2 \bar{T}_{N}^\eps - \sum_{k=N+1}^{4N}  \varepsilon^k \langle\mathcal{C}(T^\eps+\bar{T}^\eps,k) \rangle \nonumber\\
		&\quad=\mathcal{R}_1\left(\sum_{k=0}^N \eps^k T_k,\sum_{k=0}^N \eps^k \psi_k\right) - \sum_{k=0}^N \eps^k \langle \mathcal{C}(T^\eps+\bar{T}^\eps,k) - \mathcal{C}(T^\eps,k)\rangle \nonumber\\
		&\qquad - \sum_{k=N+1}^{4N} \eps^k\langle \mathcal{C}(T^\eps+\bar{T}^\eps,k) - \mathcal{C}(T^\eps,k)\rangle \nonumber\\
		&\qquad + \frac{\eps^{N+1}}{(1-\eps\eta)^2}\partial_\theta^2 \bar{T}_{N-1}^\eps + \frac{\eps^{N+2}}{(1-\eps\eta)^2} \partial_\theta^2 \bar{T}_{N}^\eps + \sum_{k=0}^N \eps^k \langle \mathcal{C}(\bar{T}^\eps+P^\eps,k) - \mathcal{C}(P^\eps,k) \rangle\nonumber\\
		&\qquad + \sum_{k=0}^N \eps^k (\partial_\eta^2 \bar{T}_k^\eps - \frac{\eps}{(1-\eps\eta)}\partial_\eta \bar{T}_{k}^\eps +\frac{1}{(1-\eps\eta)^2}\partial_\theta^2 \bar{T}_{k-2}^\eps  + \langle \bar{\psi}_k^\eps - \mathcal{C}(\bar{T}^\eps+P^\eps,k) + \mathcal{C}(P^\eps,k) \rangle),
	\end{align}
	and 
	\begin{align}\label{eq:R2p}
		&\mathcal{R}_2\left(\sum_{k=0}^N \eps^k(T_k^\eps+\bar{T}_k^\eps),\sum_{k=0}^N \eps^k(\psi_k^\eps + \bar{\psi}_k^\eps)\right)\nonumber\\
		&\quad = \varepsilon \beta\cdot \nabla \sum_{k=0}^N \varepsilon^k (T_k^\eps + \bar{T}_k^\eps) +  \sum_{k=0}^N \varepsilon^k (\psi_k^\eps + \bar{\psi}_k^\eps) - \left(\sum_{k=0}^N \varepsilon^k(T_k^\eps+\bar{T}_k^\eps)\right)^4  \nonumber\\
		&\quad = \varepsilon \beta\cdot \nabla \sum_{k=0}^N \eps^k \psi_k^\eps + \sum_{k=0}^N\varepsilon^k  \psi_k^\eps  +\sin\phi \partial_\eta \sum_{k=0}^N \varepsilon^k \bar{\psi}_k^\eps -\frac{\eps}{1-\eps\eta}\cos\phi\partial_\phi \sum_{k=0}^N\varepsilon^k \bar{\psi}_k^\eps \nonumber\\
		&\qquad - \frac{\eps}{1-\eps\eta} \cos\phi \partial_\theta \sum_{k=0}^N \eps^k \bar{\psi}_k^\eps +  \sum_{k=0}^N \varepsilon^k  \bar{\psi}_k^\eps - \left(\sum_{k=0}^N \varepsilon^k(T_k^\eps+\bar{T}_k^\eps)\right)^4   \nonumber\\
		&\quad=\mathcal{R}_2\left(\sum_{k=0}^N \eps^k T_k^\eps,\sum_{k=0}^N \eps^k \psi_k^\eps\right) +  \left(\sum_{k=0}^N \varepsilon^k T_k^\eps\right)^4 \nonumber\\
		&\qquad + \sum_{k=0}^N \eps^k(\sin\phi \partial_\eta\bar{\psi}_k^\eps -\frac{\eps}{1-\eps\eta}\cos\phi \partial_\phi \bar{\psi}_{k}^\eps +  \bar{\psi}_k^\eps - \mathcal{C}(T^\eps+\bar{T}^\eps,k) - \frac{1}{(1-\eps\eta)}\cos\phi\partial_\theta \bar{\psi}_{k-1})  \nonumber\\
		&\qquad  - \frac{\eps^{N+1}}{1-\eps\eta}\cos\phi \partial_\theta \bar{\psi}_N^\eps- \sum_{k=N+1}^{4N} \varepsilon^k \mathcal{C}(T^\eps+\bar{T}^\eps,k) \nonumber\\
		&\quad=\mathcal{R}_2\left(\sum_{k=0}^N \eps^k T_k^\eps,\sum_{k=0}^N \eps^k \psi_k^\eps\right) -\sum_{k=0}^N \eps^k (\mathcal{C}(T^\eps+\bar{T}^\eps,k) - \mathcal{C}(T^\eps,k)) - \sum_{k=N+1}^{4N} \eps^k( \mathcal{C}(T^\eps+\bar{T}^\eps,k) - \mathcal{C}(T^\eps,k)) \nonumber\\
		&\qquad   - \frac{\eps^{N+1}}{1-\eps\eta}\cos\phi \partial_\theta \bar{\psi}_N^\eps + \sum_{k=0}^N \eps^k ( \mathcal{C}(\bar{T}^\eps+P^\eps,k) - \mathcal{C}(P^\eps,k) )  \nonumber\\
		&\qquad +\sum_{k=0}^N \eps^k (\sin\phi \partial_\eta \bar{\psi}_k^\eps -\frac{\eps}{1-\eps\eta}\cos\phi \partial_\phi \bar{\psi}_{k}^\eps - \frac{1}{1-\eps\eta} \partial_\theta \bar{\psi}^\eps_{k-1}+  \bar{\psi}_k^\eps - \mathcal{C}(\bar{T}^\eps+P^\eps,k) + \mathcal{C}(P^\eps,k) ).	
	\end{align}
	By the definition \eqref{eq:g0} of $(\bar{T}_0^\eps,\bar{\psi}_0^\eps)$, we get (we drop the superscripts $\eps$ in the calculations below) 
	\begin{align}\label{eq:e00}
		E_0^0:=&\partial_\eta^2 \bar{T}_0 -\frac{\eps}{1-\eps\eta}\partial_\eta \bar{T}_0+ \langle\bar{\psi}_0 - (\bar{T}_0+T_0(0))^4 + T_0^4(0) \rangle  \nonumber\\
		=&\partial_\eta^2 (\chi(\tilde{T}_0-\tilde{T}_{0,\infty})) -\frac{\eps}{1-\eps\eta}\partial_\eta (\chi(\tilde{T}_0-\tilde{T}_{0,\infty})) + \langle \chi(\tilde{\psi}_0-\tilde{\psi}_{0,\infty}) - (\bar{T}_0+T_0(0))^4 + T_0^4(0) \rangle \nonumber\\
		 = &(\tilde{T}_0-\tilde{T}_{0,\infty})\partial_\eta^2\chi + 2\partial_\eta\chi \partial_\eta \tilde{T}_0 + \chi \partial_\eta^2\tilde{T}_0 - \frac{\eps }{(1-\eps\eta)}\partial_\eta \chi (\tilde{T}_0-\tilde{T}_{0,\infty}) \nonumber\\&-\frac{\eps}{(1-\eps\eta)} \chi \chi_0 \partial_\eta\tilde{T}_0 + \langle \chi(\tilde{\psi}_0-\tilde{\psi}_{0,\infty}) - (\bar{T}_0+T_0(0))^4 + T_0^4(0) \rangle \nonumber\\
		= &(\tilde{T}_0-\tilde{T}_{0,\infty})\partial_\eta^2\chi + 2\partial_\eta\chi \partial_\eta \tilde{T}_0 -\frac{\eps}{1-\eps\eta} \partial_\eta \chi (\tilde{T}_0 - \tilde{T}_{0,\infty}) + \chi(\partial_\eta^2 \tilde{T}_0  + F(\eps;\eta)  \partial_\eta\tilde{T}_0 + \langle \tilde{\psi}_0 - \tilde{\psi}_{0,\infty}\rangle) \nonumber\\&- \langle (\bar{T}_0+T_0(0))^4 - T_0^4(0)\rangle \nonumber\\
		 =& (\tilde{T}_0-\tilde{T}_{0,\infty})\partial_\eta^2\chi + 2\partial_\eta\chi \partial_\eta \tilde{T}_0 - \frac{\eps}{1-\eps\eta}\partial_\eta \chi (\tilde{T}_0 -\tilde{T}_{0,\infty}) \nonumber\\
		 &+ \chi \langle \tilde{T}_0^4-\tilde{T}_{0,\infty}^4\rangle - \langle (\chi(\tilde{T}_0-\tilde{T}_{0,\infty}) + \tilde{T}_{0,\infty})^4 - \tilde{T}_{0,\infty}^4 \rangle,
	\end{align}
	where we use the fact that $\partial_\eta^2 \chi =0$ and $\chi\chi_0=\chi$. We can also calculate 
	\begin{align}\label{eq:e01}
		E_0^1:=& \sin\phi \partial_\eta \bar{\psi}_0 - \frac{\eps}{1-\eps\eta}\cos\phi \partial_\phi \bar{\psi}_0 +  \bar{\psi}_0 - (\bar{T}_0 + T_0(0))^4 + T_0^4(0) \nonumber\\
		=& \sin\phi (\tilde{\psi}_0-\tilde{\psi}_{0,\infty}) \partial_\eta \chi + \chi(\sin\phi \partial_\eta \tilde{\psi}_0 + F(\eps;\eta)\cos\phi\partial_\phi \tilde{\psi}_0 + \tilde{\psi}_0 - \tilde{T}_0^4) \nonumber\\
		& + \chi(\tilde{T}_0^4 - \tilde{T}_{0,\infty}^4)  - ((\chi(\tilde{T}_0-\tilde{T}_{0,\infty}) + \tilde{T}_{0,\infty})^4 - \tilde{T}_{0,\infty}^4)\nonumber\\
		=& \sin\phi (\tilde{\psi}_0-\tilde{\psi}_{0,\infty}) \partial_\eta \chi + \chi(\tilde{T}_0^4 - \tilde{T}_{0,\infty}^4)  - ((\chi(\tilde{T}_0-\tilde{T}_{0,\infty}) + \tilde{T}_{0,\infty})^4 - \tilde{T}_{0,\infty}^4).
	\end{align}
	Similarly, by \eqref{eq:g2}, we get 
	\begin{align}\label{eq:ek0}
		E_{k}^0:=& \partial_\eta^2 \bar{T}_k - \frac{\eps}{1-\eps\eta}\partial_\eta \bar{T}_k + \frac{1}{(1-\eps\eta)^2} \partial_\theta^2 \bar{T}_{k-2} + \langle \bar{\psi}_k - \mathcal{C}(\bar T+P,k)+\mathcal{C}(P,k)\rangle \\
		=& \partial_\eta^2(\chi(\tilde{T}_k-\tilde{T}_{k,\infty})) - \frac{\eps}{1-\eps\eta}\partial_\eta (\chi (\tilde{T}_{k}-\tilde{T}_{k,\infty})) + \frac{1}{(1-\eps\eta)^2} \partial_\theta^2 \bar{T}_{k-2} \nonumber\\
		&+ \langle \chi(\tilde{\psi}_k - \tilde{\psi}_{k,\infty}) - 4(P_0+\bar{T}_0)^3 (P_k+\bar{T}_k) + 4P_0^3 P_k \rangle - \langle \mathcal{E}(P+\bar{T},k-1) - \mathcal{E}(P,k-1) \rangle \nonumber\\
		=& (\tilde{T}_k-\tilde{T}_{k,\infty})\partial_\eta^2 \chi + 2 \partial_\eta \chi \partial_\eta \tilde{T}_k -\frac{\eps}{1-\eps\eta}\partial_\eta \chi (\tilde{T}_k - \tilde{T}_{k,\infty}) \nonumber\\
		&+ \chi\bigg\langle\partial_\eta^2 \tilde{T}_k + F(\eps;\eta)\partial_\eta \tilde{T}_k + \tilde{\psi}_k - 4\tilde{T}_0^3 \tilde{T}_k - 4(\tilde{T}_0^3-P_0^3)(P_k-P_k(0)) - \mathcal{E}(\bar T+P,k-1)\nonumber\\
		& + \mathcal{E}(P,k-1) + \frac{\chi_0}{(1-\eps\eta)^2}  \partial_\theta^2\bar{T}_{k-2} \bigg\rangle  + \chi\langle 4\tilde{T}_0^3\tilde{T}_k-4\tilde{T}_{0,\infty}^3\tilde{T}_{k,\infty}\rangle \nonumber\\
		&- \langle 4(P_0+\bar{T}_0)^3(P_k+\bar{T}_k) - 4P_0^3P_k - \chi(4(\tilde{T}_0^3-P_0^3)(P_k-P_k(0)))\rangle \nonumber\\
		& - (1-\chi)\langle \mathcal{E}(\bar{T}+P,k-1) - \mathcal{E}(P,k-1)\rangle + \frac{1-\chi}{(1-\eps\eta)^2} \partial_\theta^2 \bar{T}_{k-2} \\
		=&(\tilde{T}_k-\tilde{T}_{k,\infty})\partial_\eta^2 \chi + 2 \partial_\eta \chi \partial_\eta \tilde{T}_k -\frac{\eps}{1-\eps\eta}\partial_\eta \chi (\tilde{T}_k - \tilde{T}_{k,\infty}) + \chi\langle 4\tilde{T}_0^3\tilde{T}_k-4\tilde{T}_{0,\infty}^3\tilde{T}_{k,\infty}\rangle \nonumber\\
		&- \langle 4(P_0+\bar{T}_0)^3(P_k+\bar{T}_k) - 4P_0^3P_k - \chi(4(\tilde{T}_0^3-P_0^3)(P_k-P_k(0)))\rangle \nonumber\\
		& - (1-\chi)\langle \mathcal{E}(\bar{T}+P,k-1) - \mathcal{E}(P,k-1)\rangle + \frac{1-\chi}{(1-\eps\eta)^2} \partial_\theta^2 \bar{T}_{k-2} 
	\end{align}
	and 
	\begin{align}\label{eq:ek1}
		E_k^1:=& \sin\phi \partial_\eta \bar\psi_k - \frac{\eps}{1-\eps\eta} \cos \phi \partial_{\phi} \bar{\psi}_k - \frac{1}{1-\eps\eta}\cos\phi\partial_\theta \bar{\psi}_{k-1} + \bar{\psi}_k - \mathcal{C}(\bar{T}+P,k) + \mathcal{C}(P,k) \nonumber\\
		=& \sin\phi \partial_\eta \chi (\tilde{\psi}_k-\tilde{\psi}_{k,\infty}) + \chi \bigg(\sin\phi\partial_\eta \tilde{\psi}_k + F(\eps;\eta)\cos\phi \partial_\eta \tilde{\psi}_k - \frac{\chi_0}{(1-\eps\eta)}\partial_\theta\bar{\psi}_{k-1}\nonumber\\
		&+ \tilde{\psi}_k - 4\tilde{T}_0^3 \tilde{T}_k - 4(\tilde{T}_0^3-P_0^3)(P_k-P_k(0)) - \mathcal{E}(\bar T+P,k-1) + \mathcal{E}(P,k-1) \bigg) \nonumber\\
		& + \chi(4\tilde{T}_0^3\tilde{T}_k-4\tilde{T}_{0,\infty}^3\tilde{T}_{k,\infty}) - (4(P_0+\bar{T}_0)^3(P_k+\bar{T}_k) - 4P_0^3P_k - \chi(4(\tilde{T}_0^3-P_0^3)(P_k-P_k(0)))) \nonumber\\
		& - (1-\chi)( \mathcal{E}(\bar{T}+P,k-1) - \mathcal{E}(P,k-1)) - \frac{1-\chi}{1-\eps\eta}\cos\phi\partial_\theta \bar{\psi}_{k-1},\nonumber\\
		=&\sin\phi \partial_\eta \chi (\tilde{\psi}_k-\tilde{\psi}_{k,\infty}) + \chi(4\tilde{T}_0^3\tilde{T}_k-4\tilde{T}_{0,\infty}^3\tilde{T}_{k,\infty}) \nonumber\\
		&- (4(P_0+\bar{T}_0)^3(P_k+\bar{T}_k) - 4P_0^3P_k - \chi(4(\tilde{T}_0^3-P_0^3)(P_k-P_k(0)))) \nonumber\\
		& - (1-\chi)( \mathcal{E}(\bar{T}+P,k-1) - \mathcal{E}(P,k-1)) - \frac{1-\chi}{1-\eps\eta}\cos\phi\partial_\theta \bar{\psi}_{k-1}.
	\end{align}
	Using the formulas \eqref{eq:R1p} and \eqref{eq:R2p}, we get from the above equations  
	\begin{align}\label{eq:R1b}
		&\mathcal{R}_1(T^{a,\eps},\psi^{a,\eps})\nonumber\\
		&\quad=  \eps^{N+1} \Delta T_{N-1} + \eps^{N+2} \Delta T_N^\eps   - \sum_{k=0}^N \eps^k \langle \mathcal{C}(T^\eps+\bar{T}^\eps,k) - \mathcal{C}(T^\eps,k) - \mathcal{C}(\bar{T}^\eps+P^\eps,k) + \mathcal{C}(P^\eps,k)\rangle \nonumber\\
		&\qquad- \sum_{k=N+1}^{4N} \eps^k\langle \mathcal{C}(T+\bar{T},k) \rangle  +\frac{\eps^{N+1}}{(1-\eps\eta)^2}\partial_\theta^2 \bar{T}_{N-1}^\eps + \frac{\eps^{N+2}}{(1-\eps\eta)^2} \partial_\theta^2 \bar{T}_{N}^\eps + \sum_{k=0}^N \eps^k E_k^0,
	\end{align}
	and 
	\begin{align}\label{eq:R2b}
		\mathcal{R}_2(T^{a,\eps},\psi^{a,\eps}) &=  \eps^{N+1} \beta\cdot \nabla \psi_{N}^\eps  - \sum_{k=0}^N \eps^k ( \mathcal{C}(T^\eps+\bar{T}^\eps,k) - \mathcal{C}(T^\eps,k) - \mathcal{C}(\bar{T}^\eps+P^\eps,k) + \mathcal{C}(P^\eps,k)) \nonumber\\
		&\quad- \sum_{k=N+1}^{4N} \eps^k(\mathcal{C}(T^\eps+\bar{T}^\eps,k))  - \frac{\eps^{N+1}}{1-\eps\eta}\cos\phi \partial_\theta \bar{\psi}_N^\eps  + \sum_{k=0}^N \eps^k E_k^1.
	\end{align}	

\subsection{Estimation on the approximation errors}\label{sec52}
The approximation errors calculated are expected to be small for large $N$. Here we estimate on the terms in \eqref{eq:R1b}-\eqref{eq:R2b} and prove the following lemma. 
\begin{lemma}
	Assume the boundary data in \eqref{eq:1b}-\eqref{eq:2b} satisfy $T_b \in C^2(\Omega)$, $\psi_b \in C^1(\Gamma_-)$. Let $(T^{a,\eps},\psi^{a,\eps})$ be the approximate solution constructed in section \ref{sec:2}. Assume the solution $\tilde{T}_0^\eps$  to \eqref{eq:g0} satisfy the spectral assumption \ref{asA} and $\tilde{T}_0^\eps \ge a$ for some constant $a>0$. Then for sufficient large $N$,
	\begin{align}
		&\|\mathcal{R}_1(T^a,\psi^a)\|_{L^2(\Omega)}, ~\|\mathcal{R}_1(T^a,\psi^a)\|_{L^\infty(\Omega)} \le C \eps^{N+1} + (N+2) Ce^{-\frac{\lambda \delta}{4\eps}},\label{eq:err1/est}\\
		&\|\mathcal{R}_2(T^a,\psi^a)\|_{L^2(\Omega\times\mathbb{S}^2)}, ~\|\mathcal{R}_2(T^a,\psi^a)\|_{L^\infty(\Omega\times\mathbb{S}^2)} \le C \eps^{N+1} + (N+2) Ce^{-\frac{\lambda \delta}{4\eps}}.\label{eq:err2/est}
	\end{align}
	for some constant $0\le \lambda \le 1/2$, and $C>0$ a constant independent of $\eps$.
\end{lemma}
\begin{proof}
	The proof of the lemma follows from the decay property of solutions for the nonlinear Milne and linear Milne problems (Theorem \ref{thm.nm} and Theorem \ref{thm.linearmilne}). Below we give the estimates on the terms related to geometric corrections. Other terms also appear in the flat case and the estimates can be found in \cite{Bounadrylayer2019GHM2}. Since the solutions to Milne problems with geometric corrections decay to constant solutions exponentially, same as the Milne problem without geometric corrections in \cite{Bounadrylayer2019GHM2}, these terms can be estimated in the same way.

	First by Theorem \ref{thm.nm} and Theorem \ref{thm.linearmilne}, there exists a constant $0\le \lambda<1/2$ such that 
	\begin{align*}
		|\tilde{T}_k^\eps(\eta)-\tilde{T}_{k,\infty}^\eps| \le C e^{-\lambda\eta},\quad |\tilde{\psi}_k^\eps(\eta)-\tilde{\psi}_{k,\infty}^\eps| \le C e^{-\lambda\eta}.
	\end{align*}
	By the definition of $\bar{T}_k^\eps,\bar{\psi}_k^\eps$
	The terms involving $\partial_\theta^2$ in \eqref{eq:R1b}-\eqref{eq:R2b} can be estimated by 
	\begin{align*}
		\left|\frac{\eps^{N+1}}{(1-\eps\eta)^2}\partial_\theta^2 \bar{T}_{N-1}^\eps\right| + \left|\frac{\eps^{N+2}}{(1-\eps\eta)^2} \partial_\theta^2 \bar{T}_{N}^\eps\right |+\left| \frac{\eps^{N+1}}{1-\eps\eta}\cos\phi \partial_\theta \bar{\psi}_N^\eps\right| \le C \eps^{N+1},
	\end{align*}
	since $\bar{T}_k^\eps,\bar{\psi}_k^\eps$ are bounded functions. For the term $E_0^0$ given in \eqref{eq:e00}, the term related to geometric correction can be estimated by 
	\begin{align*}
		\left|\frac{\eps}{1-\eps\eta}\partial_\eta\chi (\tilde{T}_0^\eps-\tilde{T}_{0,\infty}^\eps)\right| \le \frac{C\eps}{1-\eps\eta} 1_{\tfrac14\delta \le \eps\eta\le \tfrac38\delta}e^{-\lambda\eta} \le \frac{C}{1-\tfrac38\delta} e^{-\frac{\lambda\delta}{4\eps}}.
	\end{align*}
	The term $E_0^1$ is the same as the case without geometric correction \cite{Bounadrylayer2019GHM2}. For $E_{k}^0$, the related terms are 
	\begin{align*}
		&\left|-\frac{\eps}{1-\eps\eta} \partial_\eta \chi (\tilde{T}_k^\eps-\tilde{T}_{k,\infty}^\eps)\right| + \left|\frac{1-\chi}{(1-\eps\eta)^2}\partial_\theta^2 \bar{T}_{k-2}^\eps\right| \nonumber\\
		&\quad \le \frac{C\eps}{1-\eps\eta} 1_{\tfrac14\delta \le \eps\eta\le \tfrac38\delta}e^{-\lambda\eta}  + \frac{1}{(1-\eps\eta)^2}1_{\tfrac14\delta \le\eps\eta\le \tfrac38\delta} e^{-\lambda\eta} \nonumber\\
		&\quad \le \frac{C}{1-\tfrac38\delta} e^{-\frac{\lambda\delta}{4\eps}} + \frac{C}{(1-\tfrac38\delta)^2} e^{-\frac{\lambda\delta}{4\eps}}.
	\end{align*}
	The geometric correction related term in $E_k^1$ can be estimated similarly by 
	\begin{align*}
		\left|-\frac{1-\chi}{1-\eps\eta}\cos\phi\partial_\theta \bar{\psi}_{k-1}\right| \le \frac{C}{1-\tfrac38\delta} e^{-\frac{\lambda\delta}{4\eps}}.
	\end{align*}
	Combing the above estimates, we can obtain

\end{proof}

\subsection{Proof of Theorem \ref{thm:mainresult}}
The proof of Theorem \ref{thm:mainresult} is based on Banach fixed point theorem by showing that system \eqref{eq:1}-\eqref{eq:2} has a unique solution around $(T^{a,\eps},\psi^{a,\eps})$ and is given in \cite[section 4]{Bounadrylayer2019GHM2}. A crucial step in the proof is based on the following lemma.
\begin{lemma} \label{lm.g2}
    Let $(T^{a,\eps},\psi^{a,\eps})$ be the approximate solution constructed in the section \ref{sec:2}. Assuming the spectral assumption \ref{asA} holds for the solution $\tilde{T}_0^\eps$ of the nonlinear Milne problem \eqref{eq:g0} and $\tilde{T}_0^\eps\ge a$ for some constant $a>0$, and $T^{a,\eps}$ and $\tilde{T}^\eps_0$ are positive. The the following inequality holds
    \begin{align}\label{eq:est>g2}
       - \int_\Omega 4(T^{a,\eps})^3 g \Delta g dx = \int_\Omega 4(T^{a,\eps})^3|\nabla g|^2 dx - \int_\Omega \nabla (4(T^{a,\eps})^3)\cdot g \nabla g dx  \ge \kappa \int_\Omega |\nabla g|^2 dx  -C \|g\|_{L^2(\Omega)}^2,
    \end{align}
for any function $g$ satisfying $g(0)=0$ and for some constants $\kappa>0$, $C>0$ depending on $M$, where $M<1$ is the constant  of the spectral assumption.
\end{lemma}
\begin{proof}
   Note that $T^{a,\eps}=\sum_{k=0}^N \eps^k(T_k^\eps + \bar{T}_k^\eps)$ where $\bar{T}_k^\eps = \chi(1-r)(\tilde{T}_k^\eps-\tilde{T}_{k,\infty}^\eps)$. In the domain $1-r>\tfrac38 \delta$, $\chi(1-r)=0$ and $T^{a,\eps} = \sum_{k=0}^N \eps^k T_k^\eps$, which only contain the interior approximations. Since $\|T_k^\eps\|_{C^s(\Omega)}$ is bounded for any $s>0$ and $k=1,\ldots,N$,
   \begin{align}
    &\int_{\Omega\cap \{1-r>\tfrac38\delta\}} 4(T^{a,\eps})^3|\nabla g|^2 dx - \int_{\Omega\cap \{1-r>\tfrac38\delta\}}  \nabla (4(T^{a,\eps})^3)\cdot g \nabla g dx \nonumber\\
    &\quad = \int_{\Omega\cap \{1-r>\tfrac38\delta\}} 4(T^{a,\eps})^3|\nabla g|^2 dx - \int_{\Omega\cap \{1-r>\tfrac38\delta\}} 2 (T^{a,\eps})^{3/2}\nabla g \cdot 6 (T^{a,\eps})^{1/2} \nabla T^{a,\eps}  g dx \nonumber\\
    &\quad \ge \int_{\Omega\cap \{1-r>\tfrac38\delta\}} 4(T^{a,\eps})^3|\nabla g|^2 dx - \frac12 \int_{\Omega\cap \{1-r>\tfrac38\delta\}} 4(T^{a,\eps})^3|\nabla g|^2 dx - \frac12 \int_{\Omega\cap \{1-r>\tfrac38\delta\}} 36 T^{a,\eps} |\nabla T^{a,\eps}|^2 g^2 dx \nonumber\\
    &\quad \ge \int_{\Omega\cap \{1-r>\tfrac38\delta\}} 2(T^{a,\eps})^3|\nabla g|^2 dx - C \|g\|_{L^2(\Omega)}^2. \label{eq:366}
\end{align}
In the domain $1-r\le \frac{3}{8}\delta$, boundary layer effects play a role. First we split the integral as
   \begin{align*}
       &\int_{\Omega\cap \{1-r\le\tfrac38\delta\}} 4(T^{a,\eps})^3|\nabla g|^2 dx - \int_{\Omega\cap \{1-r\le\tfrac38\}}  \nabla (4(T^{a,\eps})^3)\cdot g \nabla g dx \nonumber\\
       &\quad = \int_{\Omega\cap \{1-r\le\tfrac38\delta\}} 4(T^{a,\eps})^3 |\partial_\eta T^{a,\eps}|^2 dx + \int_{\Omega\cap }
       =: I_1 + I_2.
   \end{align*}
   Since $\|\nabla_{x'}(T^a)\|_{L^2(\Omega)}$ is bounded, we can estimate $I_1$ the same as \eqref{eq:366}:
   \begin{align}\label{eq:I1/e}
       I_1 \ge \int_{\Omega\cap \{x_1\le \tfrac38\}} 2(T^a)^3|\nabla_{x'} g|^2 dx - C \|g\|_{L^2(\Omega)}^2.
   \end{align} 
   To estimate $I_2$, we use the spectral assumption \ref{asA}. The composite approximate solution $(T^a,\psi^a)$ is close to the solution $(\tilde{T}_0,\tilde{\psi}_0)$ of the nonlinear Milne problem \eqref{eq:g0}.
   Using the equation
   \begin{align*}
       (T^a)^3= \left(\sum_{k=0}^N \eps^k (T_k+\bar{T}_k) \right)^3 = (T_0+\bar{T}_0)^3 + \varepsilon G,
   \end{align*}
   where $G=3 (T_0+\bar{T}_0)^2 \sum_{k=1}^N \eps^{k-1} (T_k+\bar{T}_k) + 6(T_0+\bar{T}_0)^2 (\sum_{k=1}^N \eps^{k-1} (T_k+\bar{T}_k))^2 + 3(T_0+\bar{T}_0) (\sum_{k=1}^N \eps^{k-1} (T_k+\bar{T}_k))^2$,
   we can rewrite $I_2$ as 
   \begin{align*}
       I_2 &=\frac{1}{\varepsilon^2} \int_{\mathbb{T}^2} \int_0^{\frac{3}{8\eps}}  4(T_0+\bar{T}_0)^3 + 4 \varepsilon G) |\partial_\eta g|^2 d\eta dx' -\frac{1}{\varepsilon^2}\int_{\mathbb{T}^2} \int_0^{\frac{3}{8\eps}} \partial_{\eta} (4(T_0+\bar{T}_0)^3 + 4 \varepsilon G) g \partial_\eta g d\eta dx' \nonumber\\
       &= \frac{1}{\varepsilon^2}\int_{\mathbb{T}^2} dx' \int_{0}^{\frac{3}{8\varepsilon}} (4\tilde{T}_0^3 |\partial_\eta g|^2 - \partial_\eta (4 \tilde{T}_0^3)g\partial_\eta g )d\eta + \frac{1}{\varepsilon^2} \int_{\mathbb{T}^2} dx' \int_{0}^{\frac{3}{8\varepsilon}} (4(T_0+\bar{T}_0)^3-4\tilde{T}_0^3)|\partial_\eta g|^2 d\eta \nonumber\\
       &\quad - \frac{1}{\varepsilon^2} \int_{\mathbb{T}^2}dx' \int_0^{\frac{3}{8\varepsilon}} \partial_\eta (4(T_0+\bar{T}_0)^3 - 4\tilde{T}_0^3) g\partial_\eta g d\eta dx' + \frac{1}{\varepsilon^2} \int_{\mathbb{T}^2}dx' \int_0^{\frac{3}{8\varepsilon}}  \varepsilon 4 (G |\partial_\eta g|^2 - \partial_\eta G g \partial_\eta g) d\eta \nonumber\\
       &=:I_{21}+I_{22}+I_{23}+I_{24}.
   \end{align*}
   The spectral assumption \ref{asA} implies  
   \begin{align*}
       I_{21} &= \frac{1}{\varepsilon^2}\int_{\mathbb{T}^2} dx' \int_{0}^{\frac{3}{8\varepsilon}} (4\tilde{T}_0^3 |\partial_\eta g|^2 - \partial_\eta (4 \tilde{T}_0^3)g\partial_\eta g )d\eta \nonumber\\
       &\ge \frac{1}{\varepsilon^2}\int_{\mathbb{T}^2} dx' \int_{0}^{\frac{3}{8\varepsilon}}( 4\tilde{T}_0^3 |\partial_\eta g|^2 - \frac12 (4\tilde{T}_0^3 |\partial_\eta g|^2+36\tilde{T}_0|\partial_\eta \tilde{T}_0|^2 g^2) ) d\eta \nonumber\\
       &\ge  \frac{1}{2\varepsilon^2}\int_{\mathbb{T}^2} dx' \int_{0}^{\frac{3}{8\varepsilon}} (4 \tilde{T}_0^3 |\partial_\eta g|^2 - 36\tilde{T}_0|\partial_\eta \tilde{T}_0|^2 g^2 ) d\eta \nonumber\\
       &\ge \frac{1-M}{2\varepsilon^2} \int_{\mathbb{T}^2} dx' \int_{0}^{\frac{3}{8\varepsilon}} 4\tilde{T}_0^3|\partial_\eta g|^2 d\eta.
   \end{align*}
   For $I_{22}$, since $\bar{T}_0=T_0 + \chi(\varepsilon\eta) (\tilde{T}_0-T_0(0))$, it holds that 
   \begin{align*}
       (T_0+\bar{T}_0)^3 - \tilde{T}_0^3 &= (T_0 +\chi(\varepsilon\eta) (\tilde{T}_0-T_0(0)) -\tilde{T}_0)(\tilde{T}_0^2+\tilde{T}_0(T_0+\bar{T}_0)+(T_0+\bar{T}_0)^2) \nonumber\\
       &=((T_0-T_0(0)) - (1-\chi(\varepsilon\eta))(\tilde{T}_0 - \bar{T}_0))(\tilde{T}_0^2+\tilde{T}_0(T_0+\bar{T}_0)+(T_0+\bar{T}_0)^2) \nonumber\\
       &= (\partial_{x_1}T_0(\xi) \varepsilon \eta  - (1-\chi(\varepsilon\eta))(\tilde{T}_0 - \bar{T}_0))(\tilde{T}_0^2+\tilde{T}_0(T_0+\bar{T}_0)+(T_0+\bar{T}_0)^2).
   \end{align*}
   Since we are considering the integration over $x_1=\varepsilon\eta \in [0,\tfrac38\delta]$ and  $(1-\chi(\varepsilon\eta))$ is supported on $[\tfrac14\delta,\tfrac38\delta]$,
   \begin{align*}
       I_{22} = \frac{1}{\varepsilon^2} \int_{\mathbb{T}^2} dx' \int_{0}^{\frac{3}{8\varepsilon}} (4(T_0+\bar{T}_0)^3-4\tilde{T}_0^3)|\partial_\eta g|^2 d\eta \le \frac{3\delta C}{8\varepsilon^2} \int_{\mathbb{T}^2} dx' \int_0^{\frac{\delta}{2\eps}} |\partial_\eta g|^2 d\eta.
   \end{align*}
   For $I_{23}$, due to 
   \begin{align*}
       \partial_\eta  ((T_0+\bar{T}_0)^3 - \tilde{T}_0^3) &= \partial_\eta(T_0 +\chi(\varepsilon\eta) (\tilde{T}_0-T_0(0)) -\tilde{T}_0)(\tilde{T}_0^2+\tilde{T}_0(T_0+\bar{T}_0)+(T_0+\bar{T}_0)^2) \nonumber\\
       &\quad  + (T_0 +\chi(\varepsilon\eta) (\tilde{T}_0-T_0(0)) -\tilde{T}_0)\partial_\eta (\tilde{T}_0^2+\tilde{T}_0(T_0+\bar{T}_0)+(T_0+\bar{T}_0)^2) \nonumber\\
       &= \varepsilon \chi'(\eps\eta)(\tilde{T}_0-{T}_0(0))(\tilde{T}_0-T_0(0)) -\tilde{T}_0)(\tilde{T}_0^2+\tilde{T}_0(T_0+\bar{T}_0)+(T_0+\bar{T}_0)^2) \nonumber\\
       &\quad + (\chi(\varepsilon\eta)-1) \partial_\eta \tilde{T}_0 (\tilde{T}_0^2+\tilde{T}_0(T_0+\bar{T}_0)+(T_0+\bar{T}_0)^2) \nonumber\\
       &\quad +  (\partial_{x_1}T_0(\xi) \varepsilon \eta  - (1-\chi(\varepsilon\eta))(\tilde{T}_0 - \bar{T}_0))\partial_\eta(\tilde{T}_0^2+\tilde{T}_0(T_0+\bar{T}_0)+(T_0+\bar{T}_0)^2),
   \end{align*}
   with consideration of $\eps\eta\in [0,\tfrac38\delta]$ and and  $(1-\chi(\varepsilon\eta))$ being supported on $[\tfrac14\delta,\tfrac38\delta]$,
   it holds that 
   \begin{align*}
       I_{23} &= - \frac{1}{\varepsilon^2} \int_{\mathbb{T}^2}dx' \int_0^{\frac{3\delta}{8\varepsilon}} \partial_\eta (4(T_0+\bar{T}_0)^3 - 4\tilde{T}_0^3) g\partial_\eta g d\eta  \le  \frac{C}{\varepsilon^2} \int_{\mathbb{T}^2}dx' \int_0^{\frac{3\delta}{8\varepsilon}} (\varepsilon |g| |\partial_\eta g| + \frac{3\delta}{8} |g| |\partial_\eta g| )d\eta \nonumber\\
       &\le \frac{3\delta C}{8\eps^2} \int_{\mathbb{T}^2}dx' \int_0^{\frac{3\delta}{8\varepsilon}} (g^2 + |\partial_\eta g|^2) d\eta.
   \end{align*}
   For $I_{24}$, we have 
   \begin{align*}
       I_{24} =  \frac{1}{\varepsilon^2} \int_{\mathbb{T}^2}dx' \int_0^{\frac{3\delta}{8\varepsilon}}  \varepsilon 4 (G |\partial_\eta g|^2 - \partial_\eta G g \partial_\eta g) d\eta \le \frac{C }{\varepsilon} \int_{\mathbb{T}^2}dx' \int_0^{\frac{3\delta}{8\varepsilon}}( |g|^2 + |\partial_\eta g|^2) d\eta.
   \end{align*}
   Combing the above estimates gives 
   \begin{align*}
       I_2&\ge \frac{1-M}{2\varepsilon^2} \int_{\mathbb{T}^2} dx' \int_{0}^{\frac{3\delta}{8\varepsilon}} 4\tilde{T}_0^3|\partial_\eta g|^2 d\eta - \frac{3\delta C}{8\varepsilon^2}\int_{\mathbb{T}^2}dx' \int_0^{\frac{3\delta}{8\varepsilon}}( |g|^2 + |\partial_\eta g|^2) d\eta \nonumber\\
       &\quad -  \frac{C}{\varepsilon} \int_{\mathbb{T}^2}dx' \int_0^{\frac{3\delta}{8\varepsilon}}( |g|^2 + |\partial_\eta g|^2) d\eta.
   \end{align*}
   By the assumption of the lemma, $\tilde{T}_0\ge a$, hence $4\tilde{T}_0^3 \ge 4a^3$ for some constant $a>0$.
   We can take sufficiently small $\eps$ and $\delta$ such that $\varepsilon<(1-M)a^3/C$ and $3\delta C/8 \le (1-M)/8$, and we get from the above inequality
   \begin{align*}
       I_2 \ge \frac{1-M}{4\varepsilon^2} \int_{\mathbb{T}^2} dx' \int_{0}^{\frac{3\delta}{8\varepsilon}} 4\tilde{T}_0^3|\partial_\eta g|^2 d\eta. 
   \end{align*}
   Combing this with \eqref{eq:366} and \eqref{eq:I1/e} implies
   \begin{align*}
       - \int_\Omega 4(T^a)^3 g \Delta g dx &= \int_\Omega 4(T^a)^3|\nabla g|^2 dx - \int_\Omega \nabla (4(T^a)^3)\cdot g \nabla g dx \nonumber\\
       &\ge\int_{\Omega\cap \{x_1>\tfrac38\delta\}} 2(T^a)^3|\nabla g|^2 dx  + \int_{\Omega\cap \{x_1\le \tfrac38\delta\}} 2(T^a)^3|\nabla_{x'} g|^2 dx  \nonumber\\
       &\quad +\frac{1-M}{4\varepsilon^2} \int_{\mathbb{T}^2} dx' \int_{0}^{\frac{3\delta}{8\varepsilon}} 4\tilde{T}_0^3|\partial_\eta g|^2 d\eta - C\|g\|_{L^2(\Omega)}^2\nonumber\\
       &\ge \kappa \|\nabla g\|_{L^2(\Omega)}^2-C\|g\|_{L^2(\Omega)}^2,
   \end{align*}
   where $\kappa=\min\{2a^3,(1-M)a^3\}$,
   which finishes the proof.
\end{proof}

With the above lemma, Theorem \ref{thm:mainresult} can be proved in the same way as \cite[Theorem 1]{Bounadrylayer2019GHM2}.
\appendix
\section{Transport equation with geometric correction in half-space}\label{sec:apA}
Consider the equation in $(\eta,\phi)\in \mathbb{R}_+\times [-\pi,\pi)$:
\begin{align}
	&\sin\phi \partial_\eta f + F(\eta)\cos\phi \partial_\phi f = H(\eta,\phi),\\
	& f(0,\phi) = h(\phi),\quad \text{for }\sin\phi >0.
\end{align}
The following lemma was proved in \cite[Pages 1512-1514]{wu2015geometric}.
\begin{lemma}\label{lem:formula_inf}
There exists a solution to the above problem given by 
\begin{align}
	f(\eta,\phi) = \mathcal{A} h(\phi) + \mathcal{T} H(\eta,\phi),
\end{align}
with $\mathcal{A}$, $\mathcal{H}$ are defined in the following: for $\sin\phi>0$,
\begin{align}\label{eq:varphi--1}
	\mathcal{A} h(\phi) =	h(\phi'(\phi,\eta,0)) e^{-\int_0^\eta \frac{1}{\sin \phi'(\phi,\eta,\xi)} d\xi},\quad \mathcal{T}H(\eta,\phi) = \int_0^\eta \frac{h(\xi,\phi'(\phi,\eta,\xi))}{\sin \phi'(\phi,\eta,\xi)} e^{-\int_\xi^\eta \frac{1}{\sin \phi'(\phi,\eta,\rho)}d\rho} d\xi,
\end{align}
for $\sin\phi<0$ and $|E(\phi,\eta)| \le e^{-V(\eps;\infty)}$,
\begin{align}\label{eq:varphi--2}
	\mathcal{A}h(\phi) &= 0,\quad \mathcal{T} H(\eta,\phi)= \int_\eta^\infty \frac{h(\xi,-\phi'(-\phi,\eta,\xi))}{\sin \phi'(-\phi,\eta,\xi)} e^{\int_\xi^\eta \frac{1}{\sin \phi'(-\phi,\eta,\rho)}d\rho} d\xi,
\end{align}
and for $\sin\phi<0$ and $|E(\phi,\eta)| \ge e^{-V(\eps;\infty)}$,
\begin{align}\label{eq:varphi--3}
	&\mathcal{A} h(\phi) = h(\phi'(-\phi,\eta,0)) e^{-\int_0^{\eta_+} + \int_\eta^{\eta_+} \frac{1}{\sin\phi'(-\phi,\eta,\xi)}d\xi},\nonumber \\
	& \mathcal{T} H(\eta,\phi) = \int_0^{\eta_+} \frac{h(\xi,\phi'(-\phi,\eta,\xi))}{\sin\phi'(-\phi,\eta,\xi)} e^{-\int_\xi^{\eta_+} + \int_{\eta}^{\eta_+} \frac{1}{\sin\phi'(-\phi,\eta,\rho)}d\rho} d\xi \nonumber\\
	&\qquad\qquad\quad + \int_\eta^{\eta_+} \frac{h(\xi,-\phi'(-\phi,\eta,\xi))}{\sin \phi'(-\phi,\eta,\xi)} e^{\int_\xi^{\eta} \frac{1}{\sin \phi'(-\phi,\eta,\rho)}d\rho} d\xi.
\end{align}
Moreover, $\mathcal{A}$ and $\mathcal{H}$ satisfies 
\begin{align}
	&\|e^{\beta \eta}\mathcal{A} h\|_{L^\infty(\mathbb{R}_+\times [-\pi,\pi))} \le \|h\|_{L^\infty((0,\pi))},\quad \text{for any } 0\le \beta\le 1,\\
	&\|e^{\beta\eta}\mathcal{T} H\|_{L^\infty(\mathbb{R}_+\times [-\pi,\pi))} \le C \|e^{\beta\eta} H\|_{L^\infty(\mathbb{R}_+\times [-\pi,\pi))},\quad \text{for any } 0\le \beta \le \tfrac12.
\end{align}
\end{lemma}

\section{Existence for the linearized Milne problem without geometric correction} \label{sec:apB}
Consider the following linear Milne problem on the half-line:
\begin{align}
	&\partial_\eta G + \langle \varphi - 4h^3 G \rangle = S_1,\label{eq:apB1}\\
	&\sin\phi \partial_\eta\varphi + \varphi - 4h^3 G = S_2,\label{eq:apB2}\\
\end{align}
with boundary conditions 
\begin{align}\label{eq:apB3}
	G(0)=G_b,\quad \varphi(0,\phi) = \varphi_b(\phi),\quad \text{for }\sin\phi>0.
\end{align}
The following theorem was proved in \cite[Theorem 2]{Bounadrylayer2019GHM2}.
\begin{theorem}\label{thm:apB1}
	Assume $h$ satisfies the spectral assumption \ref{asA}. Assume $S_1=S_1(\eta)$, $S_2S_2(\eta,\phi)$ satisfy $|S_1(\eta)|, |S_2(\eta,\phi)| \le C_S e^{-\beta \eta}$ for any $\eta\in\mathbb{R}_+$, $\phi \in [-\pi,\pi)$. Then there exists a unique bounded solution $(G,\varphi)\in L^2_{\rm loc}(\mathbb{R}_+)\times L^2(\mathbb{R}_+\times [-\pi,\pi))$ to system \eqref{eq:apB1}-\eqref{eq:apB2} with boundary conditions \eqref{eq:apB3}. Moreover there exists a constant $G_\infty\in\mathbb{R}$ such that 
	\begin{align}
		|G(\eta)-G_\infty| \le C e^{-\beta \eta},\quad |\varphi(\eta,\phi) - 4h^3_\infty G_\infty| \le C e^{-\beta \eta},\quad \text{for any }\eta\in \mathbb{R}_+,~\phi\in [-\pi,\pi),
	\end{align}
	where $C>0$ is a positive constant depending on $\beta$ and $\int_{\sin\phi>0}\sin\phi \varphi_b^2 d\phi$.
\end{theorem}
The theorem was proved in \cite{Bounadrylayer2019GHM2} by first showing the existence on bounded interval and then using uniform estimate to extend the solution to the half-line. Consider system \eqref{eq:apB1}-\eqref{eq:apB2} on the interval $\eta\in[0,B]$ with the following boundary conditions at $\eta=B$,
\begin{align}\label{eq:apB4}
	\partial_\eta G(B)=0,\quad \varphi(B,\phi) = \varphi(B,-\phi),\quad \text{for }\sin\phi>0.
\end{align}
The following theorem holds.
\begin{theorem}\label{thm:apB2}
	Let the assumptions of Theorem \ref{thm:apB2} hold. Then there exists a unique solution $(G,\varphi) \in C^2([0,B])\times C^1([0,B]\times [-\pi,\pi))$ to system \eqref{eq:apB1}-\eqref{eq:apB2} with boundary condtions \eqref{eq:apB3}-\eqref{eq:apB4}.
\end{theorem}

\bibliographystyle{siam}
\bibliography{GeometricRHT}

\end{document}